\documentclass[10pt]{article}

\usepackage[utf8]{inputenc}
\usepackage{amsfonts,mathrsfs}
\usepackage{graphicx}
\usepackage{amsmath, mathtools}
\usepackage{amsthm}
\usepackage{amssymb}
\usepackage{amsxtra}
\usepackage{pmboxdraw}
\usepackage{xparse}
\usepackage{verbatim}
\usepackage{subcaption}
\usepackage{multirow,makecell} 
\usepackage{color}
\usepackage{array,dcolumn}
\usepackage{enumitem}
\usepackage{bbm}
\allowdisplaybreaks
\usepackage[colorlinks=true,citecolor=blue,linkcolor=black]{hyperref}
\usepackage[numbers,sort]{natbib}
\numberwithin{equation}{section}

\usepackage[]{todonotes} %

\usepackage{pgf}
\usepackage{tikz}
\usetikzlibrary{arrows,automata}

\usepackage{setspace}

\usepackage{xparse}

\ExplSyntaxOn
\NewDocumentCommand{\MeijerG}{smmmm}
 {
  \IfBooleanTF{#1}
   {
    \vic_meijerg:nnnnnn { #2 } { #3 } { #4 } { #5 } { small } { }
   }
   {
    \vic_meijerg:nnnnnn { #2 } { #3 } { #4 } { #5 } { } { \; }
   }
 }

\seq_new:N \l__vic_meijerg_args_in_seq
\seq_new:N \l__vic_meijerg_args_out_seq

\cs_new_protected:Nn \vic_meijerg:nnnnnn
 {
  \seq_set_split:Nnn \l__vic_meijerg_args_in_seq { | } { #3 }
  \seq_clear:N \l__vic_meijerg_args_out_seq  
  \seq_map_inline:Nn \l__vic_meijerg_args_in_seq
   {
    \seq_put_right:Nn \l__vic_meijerg_args_out_seq
     {
      \begin{#5matrix} ##1 \end{#5matrix}
     }
   }
  G\sp{#1}\sb{#2}
  \left(
  \seq_use:Nn \l__vic_meijerg_args_out_seq { #6\middle|#6 }
  #6\middle|#6
  #4
  \right)
 }
\ExplSyntaxOff

\catcode`,\active

\catcode`\,12

\newcommand{\RN}[1]{%
	\textup{\uppercase\expandafter{\romannumeral#1}}%
}

\def\Re{ \mathrm{Re}}

\def\cK{\mathcal{K}}

\def\C{\mathbb{C}}

\def\N{\mathbb{N}}

\def\R{\mathbb{R}}

\def\Z{\mathbb{Z}}

\newcommand{\Pf}{{\textup{Pf}}}
\newcommand{\erfc}{\operatorname{erfc}}
\newcommand{\erf}{\operatorname{erf}}
\newcommand{\bfR}{\mathbf{R}}

\newcommand{\bfK}{\mathbf{K}}

\newcommand{\Ai}{\operatorname{Ai}}

\newcommand{\sgn}{\operatorname{sgn}}

\newcommand{\re}{\operatorname{Re}}
\newcommand{\im}{\operatorname{Im}}

\newtheorem*{thm*}{Theorem}
\newtheorem{thm}{Theorem}[section]
\newtheorem{lem}[thm]{Lemma}

\newtheorem{prop}[thm]{Proposition}
\newtheorem*{prop*}{Proposition}
\newtheorem*{lem*}{Lemma}

\theoremstyle{definition}

\newtheorem{rem}{Remark}[section]
\newtheorem{eg*}{Example}
\newtheorem{egs*}{Examples}
\newtheorem*{Q*}{Question}

\theoremstyle{remark}
\newtheorem*{rmk*}{Remark}
\newtheorem*{rmks*}{Remarks}

\title{Real Eigenvalues of Asymmetric Wishart Matrices: \\
Expected Number, Global Density and Integrable Structure}
\date{\today}
\author{ 
Sung-Soo Byun\footnote{
Department of Mathematical Sciences and Research Institute of Mathematics, Seoul National University, Seoul 151-747, Republic of Korea
\newline Email: \href{sungsoobyun@snu.ac.kr}{\nolinkurl{sungsoobyun@snu.ac.kr}}}\quad and\quad
Kohei Noda\footnote{
Institute of Mathematics for Industry, Kyushu University, West Zone 1, 744 Motooka, Nishi-ku, Fukuoka 819-0395, Japan
\newline Email: \href{k-noda@imi.kyushu-u.ac.jp}{\nolinkurl{k-noda@imi.kyushu-u.ac.jp}}} 
}

\begin{document}

\maketitle

\begin{abstract}
We investigate the real eigenvalues of asymmetric Wishart matrices of size $N$, indexed by the rectangular parameter $\nu \in \mathbb{N}$ and the non-Hermiticity parameter $\tau \in [0,1]$. The rectangular parameter $\nu$ is either fixed or proportional to $N$. The non-Hermiticity parameter $\tau$ is either fixed or $\tau = 1 - O(1/N)$, corresponding to the strongly and weakly non-Hermitian regimes, respectively.

We establish a decomposition structure for the finite-$N$ correlation kernel of the real eigenvalues, which form Pfaffian point processes. Taking the symmetric limit $\tau = 1$, where the model reduces to the Laguerre orthogonal ensemble, this decomposition structure reduces to the known rank-one perturbation structure established by Adler, Forrester, Nagao, and van Moerbeke, as well as by Widom.
Using the decomposition structure, we show that the expected number of real eigenvalues is proportional to $\sqrt{N}$ in the strongly non-Hermitian regime and to $N$ in the weakly non-Hermitian regime, providing explicit leading coefficients in both cases. Furthermore, we derive the limiting real eigenvalue densities, which recovers the Marchenko-Pastur distribution in the symmetric limit.
\end{abstract}

\tableofcontents

\section{Introduction and main results}

For given non-negative integers \(N\) and \(\nu\), let \(\mathrm{P}\) and \(\mathrm{Q}\) be \(N \times (N+\nu)\) random matrices with independent \emph{real} Gaussian entries, each having mean 0 and variance \(1/(2N)\). The matrices \(\mathrm{P}\) and \(\mathrm{Q}\) are independent and are referred to as rectangular Ginibre matrices \cite{BF24}.
By definition, the \emph{asymmetric Wishart matrix} $X$ is given by 
\begin{equation}
\label{Asymmetric Wishart matrix}
	X:=X_+ X_{-}^{\mathsf{T}}, \qquad 
	X_{\pm}=\sqrt{1+\tau} \, \mathrm{P}\pm\sqrt{1-\tau} \, \mathrm{Q}, 
\end{equation}
where $\tau \in [0,1]$ and $A^\mathsf{T}$ denotes the transposition of a matrix $A$.
Here, the rectangular parameter $\nu \equiv \nu_N$ and the non-Hermiticity parameter $\tau \equiv \tau_N$ possibly depend on $N$.
The model $X$ appears in the literature also under the name of \emph{sample cross-covariance matrix} or \emph{chiral elliptic Ginibre ensemble}, and finds applications in various areas such as quantum chromodynamics and ecosystem stability; see e.g. \cite{BBD23,LHG24,APS10,LHLG25} and references therein. 
Note in particular that the parameter $\tau \in [0,1]$ quantifies the asymmetry of the model. In the extremal, symmetric case when $\tau=1$, the matrix model $X$ becomes the classical Wishart matrix (the Laguerre orthogonal ensemble) \cite[Chapter 3]{Fo10}.  

Beyond the matrix with real elements, the complex and quaternionic counterparts of non-Hermitian Wishart matrices have also been extensively studied in the literature; see e.g. \cite{Os04,BN24,ABK21,Ak05,AB07,AB10,AB23a,AEP22,CJQ20,KS10} and references therein. A remarkable property of the eigenvalues in the complex and quaternionic cases is their interpretation as one-component Coulomb gas ensembles \cite{Fo10,Se24,BF24}. In contrast, such a one-component realisation is no longer valid for matrices with real elements. The key distinction lies in the fact that, in the real case, there is a non-trivial probability of having purely real eigenvalues; see Figure~\ref{Fig_eigenvalues}.  

\medskip 

In this paper, we investigate the real eigenvalues of asymmetric Wishart matrices. The main contributions of this work are summarised as follows:
\begin{itemize}
    \item \textbf{Expected number of real eigenvalues} (Theorem~\ref{Thm_expected number}); 
    %\smallskip
    \item \textbf{Macroscopic density of real eigenvalues}  (Theorem~\ref{Thm_global density});
    %\smallskip 
    \item \textbf{Decomposition structure of the correlation kernel} (Theorem~\ref{thm relationship R and C}). 
\end{itemize}
Analogous results in the context of the elliptic Ginibre matrices were established in \cite{EKS94,FN08,BKLL23,Efe97}.

\begin{figure}[t]
	\begin{subfigure}{0.45\textwidth}
    	\begin{center}	
    		\includegraphics[width=\textwidth]{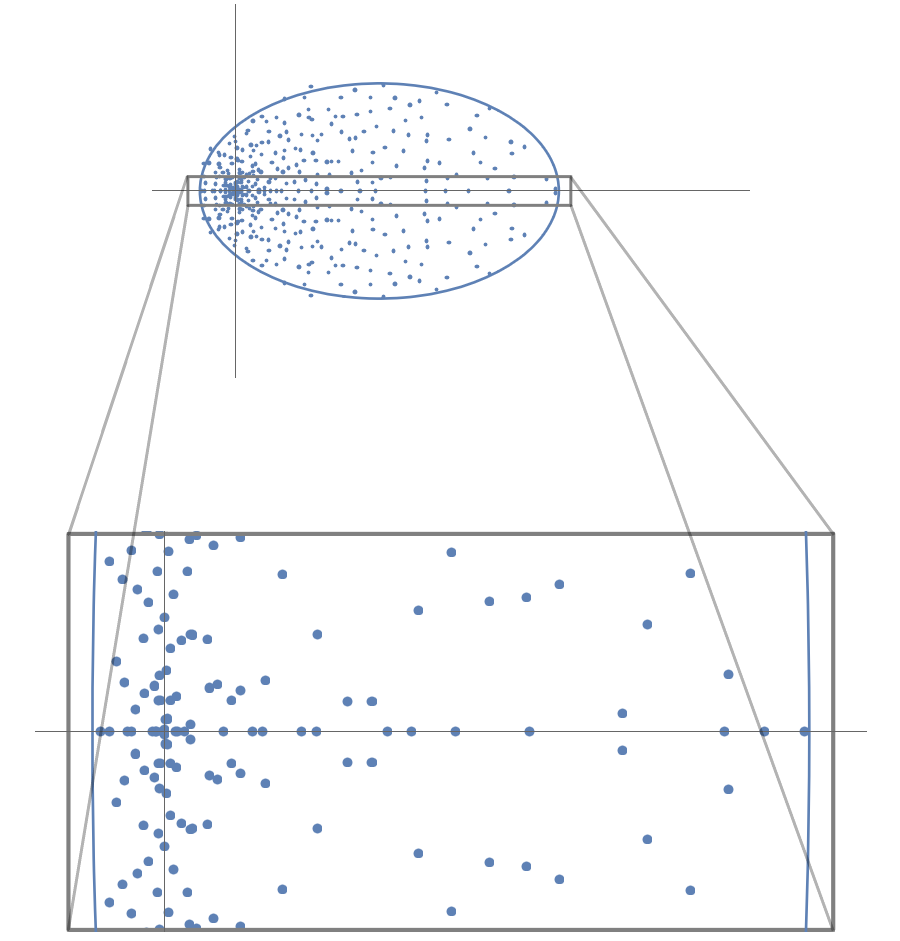}
    	\end{center}
    	\subcaption{$\nu=0$}
    \end{subfigure}	 \qquad 
	\begin{subfigure}{0.45\textwidth}
	\begin{center}	
		\includegraphics[width=\textwidth]{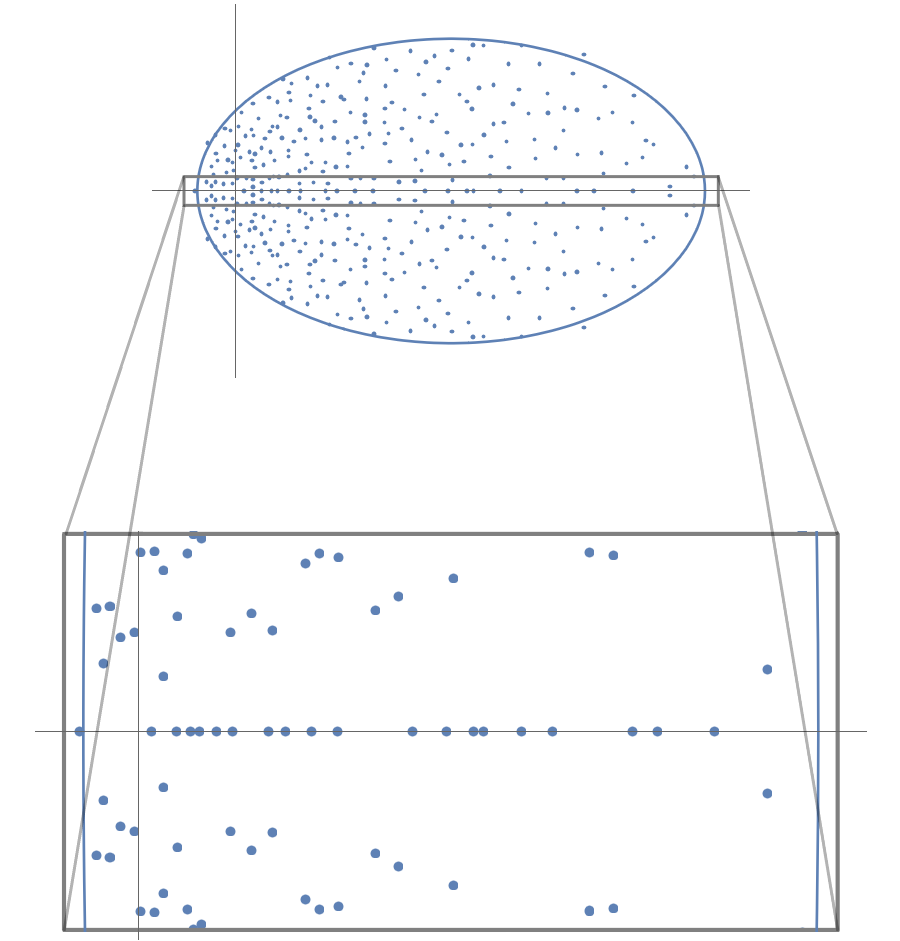}
	\end{center}
	\subcaption{$\nu=N$}
\end{subfigure}	
	\caption{The plots show the eigenvalues of asymmetric Wishart matrices with $N=400$ and $\tau=1/2$. In the zoomed-in images, the presence of purely real eigenvalues is clearly visible. Furthermore, in the case of $\nu=0$, there is an accumulation of eigenvalues near the origin, reflecting the divergent density.} \label{Fig_eigenvalues}
\end{figure} 

\smallskip 

In the study of asymmetric Wishart matrices, various interesting regimes emerge, depending on the scaling properties of the parameters.

\begin{itemize}
  \item[(a1)] \emph{Strongly non-Hermitian regime}. This is the case where \( \tau \) stays away from the extremal, symmetric case \( \tau = 1 \). In particular, we usually consider the case where \( \tau \in [0, 1) \) is fixed. 
    \item[(a2)]  \emph{Weakly non-Hermitian regime}. This is the case when $\tau \uparrow 1$ with a proper speed. This regime, first pioneered in \cite{FKS97,FKS97a,FKS98} for the elliptic Ginibre matrices, allows  one to observe a non-trivial transition to Hermitian matrices. In particular, we mainly consider the regime $1-\tau=O(N^{-1})$.   
    \item[(b1)] \emph{Singular density regime}. This is the case where \( \nu = o(N) \). In this case, for general \( \tau \in [0, 1) \), the origin is contained in the limiting spectrum with divergent density; cf. Figure~\ref{Fig_eigenvalues} (A). 
    \item[(b2)] \emph{Regular density regime}. This is the case where $\nu = \varrho N$ with a positive $\varrho > 0$. Whether the origin is contained in the limiting spectrum is determined by the sign of $\tau \sqrt{1+\varrho} - 1$; see \eqref{droplet} and \eqref{def of limiting density strong}. If the sign is positive, the origin is included in the spectrum; however, the density at the origin remains bounded. See Figure~\ref{Fig_eigenvalues} (B).
    \end{itemize}
    
Following the notations in the previous literature, we use the parametrisation
        \begin{equation} \label{tau AH}
	\tau=1-\frac{\alpha^2}{2N}, \qquad \alpha>0 
\end{equation}      
for the weakly non-Hermitian regime. (See also a recent work \cite{ADM24} where various mesoscopic regimes were explored.)
We also note that when discussing macroscopic properties, it is often unnecessary to distinguish between the two cases of rectangular parameters. By combining these cases, one can assume that
\begin{equation}
\label{nu parameter}
\lim_{N\to\infty} \frac{\nu}{N}=
\varrho  \in [0,\infty). 
\end{equation}   

In the following, we focus on the case where \( N \) is an even integer, as this choice simplifies the overall presentation and analysis. While the case of odd \( N \) is also tractable, it requires additional computations. For instance, the odd $N$ case was addressed in \cite{Si07, FM09} for the real Ginibre matrix.

Let $\mathcal{N} \equiv \mathcal{N}_{\tau,\nu}$ denote the number of real eigenvalues.  
We denote by $\bfR_{N,k}$ the $k$-point correlation function of real eigenvalues; see \cite[Chapter 7]{BF24}. 
In particular, the $1$-point function $\bfR_N \equiv \bfR_{N,1}$ can be characterised by 
\begin{equation} \label{1pt test function}
\mathbb{E} \Big[ \sum_{j=1}^{ \mathcal{N} } f(x_j) \Big] = \int_{ \mathbb{R} } f(x)\, \bfR_N(x) \,dx,
\end{equation}
where $f$ is a test function. 
For general $k$, the correlation function $\bfR_{N,k}$ can be written explicitly in terms of a certain Pfaffian; see \eqref{def of rescaled k point correlation} below.

When exploring the real eigenvalues of asymmetric random matrices, a natural object of interest is their \textit{number/counting statistics}. A fundamental aspect of this is determining the typical number of real eigenvalues. This has been studied in various models, beginning with the seminal work on the real Ginibre matrix \cite{EKS94, FN07} and extending to its variants, including the elliptic Ginibre matrix in both the strongly \cite{FN08} and weakly \cite{BKLL23} non-Hermitian regimes, as well as truncated matrices and their products \cite{FIK20, LMS22, FI16}, spherical ensembles \cite{EKS94,FM12}, and products of Ginibre matrices \cite{AB23, Si17a, Fo14}. Furthermore, \cite{TV15} established that matrices with general i.i.d. elements asymptotically exhibit the same number of real eigenvalues, thereby demonstrating universality.
(We also refer the reader to \cite{Ch22,ACCL24,ABES23,FKP24,DLMS24,GPX24} and the references therein for the counting statistics of various two-dimensional point processes.)

\medskip 

In this work, we address these questions for asymmetric Wishart matrices. 
For this purpose, let 
\begin{equation} \label{EN tau bfRN}
E_{N,\tau}^\nu := \mathbb{E} \mathcal{N}  = \int_\R \bfR_{N}(x)\,dx
\end{equation}
be the expected number of real eigenvalues of $X$. Here, the second identity follows from \eqref{1pt test function}. 
We write 
\begin{equation} \label{droplet right left edge}
\xi_{\pm}: =\tau (2+\varrho) \pm (1+\tau^2) \sqrt{1+\varrho} .
\end{equation} 
Recall that the modified Bessel function $I_\nu$ of the first kind is given by  
\begin{equation} \label{I nu}
I_\nu(z):= \sum_{k=0}^\infty \frac{(z/2)^{2k+\nu}}{k!\,\Gamma(\nu+k+1)}; 
\end{equation}
see e.g. \cite[Chapter 10]{NIST}.   

\begin{thm}[\textbf{Expected number of real eigenvalues}] \label{Thm_expected number} 
Suppose that \(\nu = \varrho N\) with \(\varrho > 0\) fixed, or that \(\nu \geq 0\) is fixed. 
\begin{itemize}
    \item[\textup{(i)}] \textup{\textbf{(Strongly non-Hermitian regime)}} Let $\tau \in [0,1)$ be fixed. Then as $N \to \infty,$ we have 
\begin{equation} \label{EN SH}
 E_{N,\tau}^\nu =  \Big( \frac{1}{4\pi}\frac{N}{1-\tau^2} \Big)^{\frac12} 
 c(\tau,\varrho) + o( \sqrt{N} ),
\end{equation}
where 
\begin{align} \label{c tau varrho}
	\begin{split} 
	c(\tau,\varrho) := \int_{\xi_-}^{\xi_+} (x^2+(1-\tau^2)^2 \varrho^2/4 )^{-\frac14}\,dx.
	\end{split}
\end{align} 
Here, $\xi_\pm$ are given by \eqref{droplet right left edge}. If $\nu \ge 0$ is fixed, then the result holds with $\varrho=0$. 
\smallskip 
\item[\textup{(ii)}] \textup{\textbf{(Weakly non-Hermitian regime)}} Let $\tau$ be given by \eqref{tau AH}. Then as $N \to \infty,$ we have 
\begin{equation} \label{EN AH}
	E_{N,\tau}^\nu = c(\alpha)\,N + o(N), 
\end{equation}
where 
\begin{equation} \label{c(alpha)}
c(\alpha):=e^{-\alpha^2/2} [ I_0(\tfrac{\alpha^2}{2})+I_1(\tfrac{\alpha^2}{2}) ]. 
\end{equation}
Here, $I_\nu$ is the modified Bessel function \eqref{I nu}.
\end{itemize}
\end{thm}

See Figure~\ref{Fig_expected number} for the numerics on Theorem~\ref{Thm_expected number}.

\begin{figure}[t]
	\begin{subfigure}{0.44\textwidth}
    	\begin{center}	
    		\includegraphics[width=\textwidth]{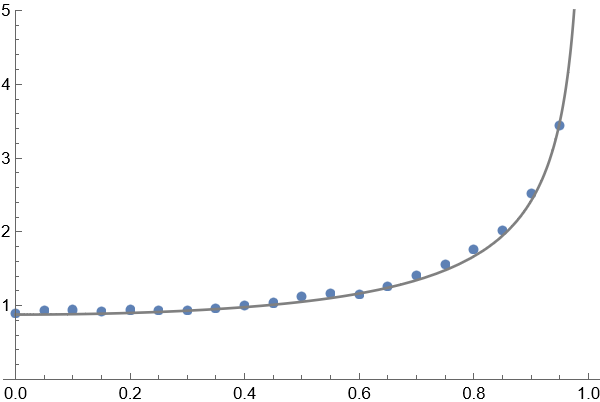}
    	\end{center}
    	\subcaption{Strong non-Hermiticity}
    \end{subfigure}	 \quad 
	\begin{subfigure}{0.44\textwidth}
	\begin{center}	
		\includegraphics[width=\textwidth]{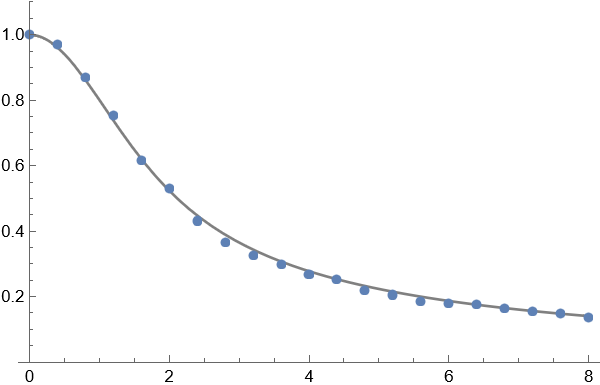}
	\end{center}
	\subcaption{Weak non-Hermiticity}
\end{subfigure}	
	\caption{Plot (A) shows \( \tau \mapsto E_{N,\tau}^\nu / \sqrt{N} \) (blue dots) and its comparison with \( c(\tau, \varrho) \) (solid gray line). Plot (B) illustrates \( \alpha \mapsto E_{N,\tau}^\nu / N \) (blue dots) for \( \tau \) given by \eqref{tau AH}, compared to \( c(\alpha) \) (solid gray line). Here, we use 20 samples of \( X \) with size \( N = 400 \) and \( \nu = N \).} \label{Fig_expected number}
\end{figure}

\begin{rem}[Alternative expressions]
The coefficient $c(\tau,\varrho)$ in \eqref{c tau varrho} can also be expressed as 
\begin{equation}
c(\tau,\varrho)	= \sqrt{\frac{2}{\varrho}} (1-\tau^2)^{-\frac12} \bigg[   {}_2 F_1\Big( \frac14,\frac12; \frac32; -\frac{4\xi_+^2}{\varrho^2 (1-\tau^2)^2} \Big) \xi_+ -   {}_2 F_1\Big( \frac14,\frac12;\frac32;-\frac{4\xi_-^2}{\varrho^2 (1-\tau^2)^2} \Big)  \xi_- \bigg],
\end{equation}
where ${}_2F_1$ is the hypergeometric function, which can be defined by the Euler integral formula   
\begin{equation} \label{2F1 Euler integral}
	{}_2F_1(a,b,c,z)=\frac{1}{\Gamma(b)\Gamma(c)\Gamma(c-b)} \int_0^1 \frac{ t^{b-1} (1-t)^{c-b-1} }{ (1-z t)^a }\,dt, \qquad (c>b>0);
\end{equation}
see e.g. \cite[Chapter 15]{NIST}. On the other hand, by using \eqref{I nu}, the coefficient $c(\alpha)$ in \eqref{c(alpha)} can be written as 
\begin{equation}   \label{c(alpha) summation}
	c(\alpha) =\sum_{k=0}^\infty \frac{(2k-1)!!}{2^k \,k!\,(k+1)!} (-1)^k \alpha^{2k}. 
\end{equation}
\end{rem}

\begin{rem}[Special case $\varrho=0$ and comparison with previous results at strong non-Hermiticity] \label{Rem_special case}
By letting $\varrho=0$ in  \eqref{c tau varrho}, one can see that $c(\tau,0)=4$. Therefore, if $\nu$ is fixed, the asymptotic behaviour \eqref{EN tau nu fixed} becomes simplified as 
\begin{equation} \label{EN tau nu fixed}
	E_{N,\tau}^\nu \Big|_{ \varrho=0 } = \Big( \frac{4}{\pi}\frac{ N }{1-\tau^2} \Big)^{\frac12} +o(\sqrt{N}).
\end{equation}
This formula was previously proposed in \cite[Chapter 6]{Phillips10} using the Dirac picture of the matrix model. 

For the special case when $\tau=\nu=0$, the asymmetric Wishart matrix $X$ reduces to the product of two Ginibre matrices.
For general fixed $m \in \mathbb{N}$, it was obtained in \cite{Si17a} that the expected number of real eigenvalues of products of $m$ independent Ginibre matrices is asymptotically $(2mN/\pi)^{1/2}$. One can observe that this formula for $m=2$ agrees with \eqref{EN tau nu fixed} for $\tau=0$. 
\end{rem}

\begin{rem}[Independence of the rectangular parameter at weak non-Hermiticity]
In the strongly non-Hermitian regime, the leading-order coefficient in \eqref{EN SH} depends on the rectangular parameter $\nu = \varrho N$. In contrast, the leading asymptotic behaviour at weak non-Hermiticity is independent of the order of the rectangular parameter $\nu$. In other words, in \eqref{EN AH}, the coefficient $c(\alpha)$ depends only on $\alpha$ and not on $\varrho$.  

Moreover, the function $c(\alpha)$ also arises in the context of the elliptic Ginibre matrices at weak non-Hermiticity; see \cite[Theorem 2.1]{BKLL23}. These observations suggest a possible universal phenomenon, indicating that the function $c(\alpha)$ might appear in a broader class of random matrix ensembles.
\end{rem}

We now present the limiting density of the real eigenvalues. 
First, recall that in the extremal symmetric case when $\tau=1$, the limiting empirical distribution is given by the celebrated Marchenko-Pastur law \cite{MP67}
\begin{equation} \label{MP law}
\rho_{\textup{MP},\varrho}(x):=\mathbbm{1}_{[\lambda_-,\lambda_+]}(x)\cdot \frac{1}{2\pi} \frac{ \sqrt{(\lambda_+-x)(x-\lambda_-)} }{x}, \qquad 
\lambda_\pm=(\sqrt{\varrho+1}\pm 1)^2.
\end{equation} 
For a fixed $\tau \in [0,1)$, it was shown that the limiting empirical distribution of the total (complex and real) eigenvalues is given by the \emph{non-Hermitian extension of the Marchenko-Pastur law} (also called \emph{shifted elliptic law}) \cite{VB14,ABK21}
\begin{equation} \label{equilibrium msr}
d\mu(z):=\frac{1}{1-\tau^2} \frac{1}{ \sqrt{4|z|^2+(1-\tau^2)^2 \varrho^2} } \cdot \mathbbm{1}_{S_\varrho} (z)\,\frac{d^2z}{\pi},
\end{equation}
where 
\begin{equation} \label{droplet}
S_\varrho:=\Big\{ (x,y)\in \R^2 : \Big( \frac{x-\tau(2+\varrho)}{(1+\tau^2) \sqrt{1+\varrho} } \Big)^2+\Big( \frac{y}{ (1-\tau^2) \sqrt{1+\varrho} } \Big)^2 \le 1 \Big\}.
\end{equation}  
Observe here that for each $\tau \in [0,1)$ the right and left endpoints  of $S_\varrho$ are given by $\xi_{\pm}$ in \eqref{droplet right left edge}. 
Furthermore, one can also notice that the endpoints $\lambda_\pm$ of the Marchenko-Pastur law \eqref{MP law} corresponds to the symmetric limits of $\xi_{\pm}$, i.e. $\lambda_{\pm}: = \lim_{ \tau \to 1 } \xi_\pm.$ 

However, the limiting distribution \eqref{equilibrium msr} does not provide direct information about the real eigenvalues. As determined in Theorem~\ref{Thm_expected number}, in the strongly non-Hermitian regime, the typical number of real eigenvalues is of order \( O(\sqrt{N}) \), which is subdominant compared to the total \( N \) eigenvalues. 
Furthermore, in the weakly non-Hermitian regime, no previous work has demonstrated the macroscopic behaviour of the eigenvalues. In the next result, we address both of these aspects.

Let us write 
\begin{equation} \label{rho N tau bfRN}
\rho_{N,\tau}^\nu(x) := \frac{1}{E_{N,\tau}} \bfR_{N}(x)
\end{equation}
for the averaged density of real eigenvalues.

\begin{thm}[\textbf{Limiting density of real eigenvalues}]\label{Thm_global density}  
Suppose that \(\nu = \varrho N\) with \(\varrho > 0\) fixed, or that \(\nu \geq 0\) is fixed. 
Then for any bounded continuous test function $f$, we have the following. 
\begin{itemize}
    \item[\textup{(i)}] \textup{\textbf{(Strongly non-Hermitian regime)}} Let $\tau \in [0,1)$ be fixed. Define 
    \begin{equation} 
\label{def of limiting density strong}
\rho_{\tau,\varrho}^s(x):=\mathbbm{1}_{[\xi_-,\xi_+]}(x)\cdot \frac{(x^2+(1-\tau^2)^2 \varrho^2/4 )^{-\frac14}}{ c(\tau,\varrho) },  
\end{equation}
where $c(\tau,\varrho)$ is given by \eqref{c tau varrho}. 
    Then we have 
\begin{equation} \label{rho N strong test limit}
    \lim_{N\to\infty}
    \int_{\R}f(x)\rho_{N,\tau}^\nu(x) \, dx = \int_{\R}f(x)\rho_{\tau,\varrho}^s(x) \, dx.
\end{equation}
Furthermore, for any fixed $x \in (\xi_-,\xi_+)$, we have the pointwise limit     
\begin{equation}
\label{scaling limit of macroscopic density at strong}
	\lim_{N\to\infty}\rho_{N,\tau}^\nu(x)= \rho_{\tau,\varrho}^s(x).
\end{equation} 
\item[\textup{(ii)}] \textup{\textbf{(Weakly non-Hermitian regime)}} Let $\tau$ be given by \eqref{tau AH}. Define 
\begin{equation} 
\rho_{\alpha,\varrho}^w(x):=\mathbbm{1}_{[\lambda_-,\lambda_+]}(x)\cdot \frac{1}{c(\alpha)} \frac{1}{2\alpha \sqrt{\pi}} \frac{1}{\sqrt{x}} \erf\Big( \frac{\alpha}{2} \sqrt{ \frac{(\lambda_+-x)(x-\lambda_-)}{x} } \Big),
\end{equation} 
where $\erf(x)=\frac{2}{\sqrt{\pi}} \int_0^x e^{-t^2}\,dt$ is the error function. 
Then we have 
\begin{equation}  \label{rho N weak test limit}
    \lim_{N\to\infty}
    \int_{\R}f(x)\rho_{N,\tau}^\nu(x) \, dx = \int_{\R}f(x)\rho_{\alpha,\varrho}^w(x) \, dx. 
\end{equation}
Furthermore, for any fixed $x \in (\lambda_-,\lambda_+)$, we have the pointwise limit
\begin{equation}
\lim_{N\to\infty}\rho_{N,\tau}^\nu(x)= \rho_{\alpha,\varrho}^w(x). 
\end{equation}
\end{itemize}
\end{thm}

See Figure~\ref{Fig_Density} for the numerics. 

\begin{figure}[t]
	\begin{subfigure}{0.32\textwidth}
	\begin{center}	
		\includegraphics[width=\textwidth]{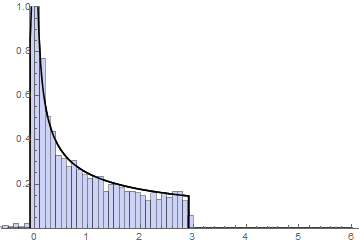}
	\end{center}
	\subcaption{$\nu=0$}
\end{subfigure}	
		\begin{subfigure}{0.32\textwidth}
		\begin{center}	
			\includegraphics[width=\textwidth]{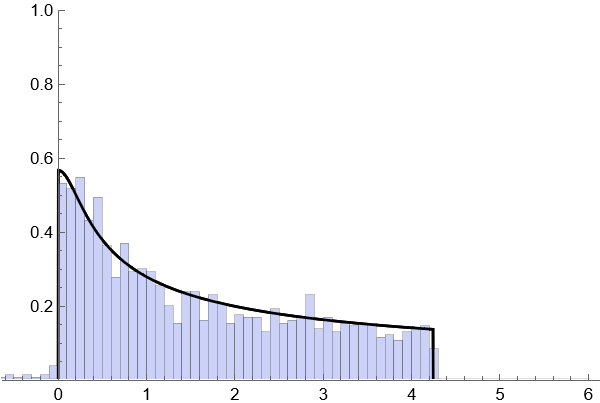}
		\end{center}
		\subcaption{$\nu=N$}
	\end{subfigure}	
	\begin{subfigure}{0.32\textwidth}
		\begin{center}	
			\includegraphics[width=\textwidth]{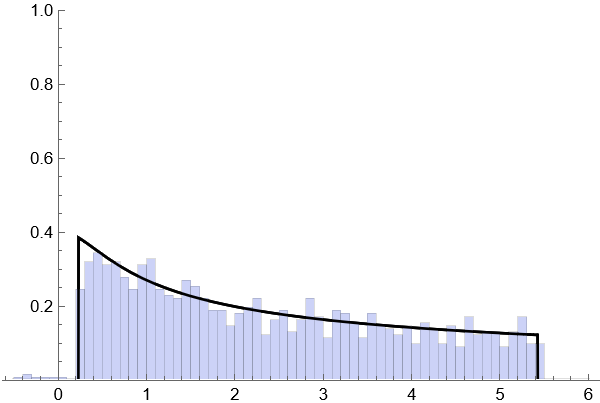}
		\end{center}
		\subcaption{$\nu=2N$}
	\end{subfigure}	

        \begin{subfigure}{0.32\textwidth}
		\begin{center}	
			\includegraphics[width=\textwidth]{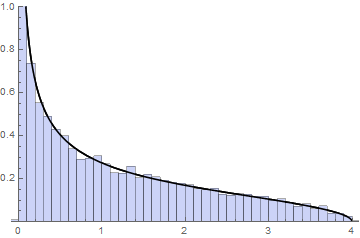}
		\end{center}
		\subcaption{$\nu=0$}
	\end{subfigure}	
	\begin{subfigure}{0.32\textwidth}
		\begin{center}	
			\includegraphics[width=\textwidth]{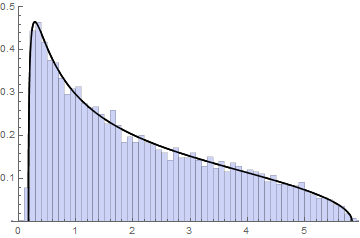}
		\end{center}
		\subcaption{$\nu=N$}
	\end{subfigure}	
	\begin{subfigure}{0.32\textwidth}
		\begin{center}	
			\includegraphics[width=\textwidth]{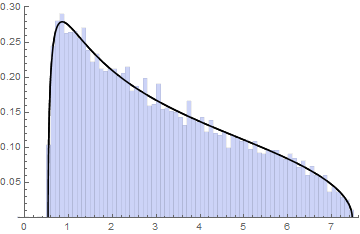}
		\end{center}
		\subcaption{$\nu=2N$}
	\end{subfigure}	
	\caption{The plots (A)--(C) present histograms of real eigenvalues for 20 samples of $X$ with size $N=2000$, compared to $\rho_{\tau,\varrho}^s$, where $\tau = 1/\sqrt{2}$. Similarly, the plots (D)--(F) illustrate histograms of real eigenvalues for 20 samples of $X$ with size $N=400$, where $\tau$ is defined by \eqref{tau AH}, compared to $\rho_{\alpha,\varrho}^w$, with $\alpha = 1$.} \label{Fig_Density}
\end{figure}

\begin{rem}[Normalisation constants]
By definition \eqref{c tau varrho}, $c(\tau,\varrho)$ is the normalisation constant that turns $\rho_{\tau,\varrho}^s$ into a probability density function.
For the weak non-Hermiticity, the mass-one condition of $\rho_{\alpha,\varrho}^w$ follows from the identity: for any $\varrho \ge 0$, 
	\begin{equation} \label{c(alpha) integral}
		c(\alpha)=\frac{1}{2\alpha \sqrt{\pi}} \int_{\lambda_-}^{\lambda_+} \frac{1}{\sqrt{x}} \erf\Big( \frac{\alpha}{2} \sqrt{ \frac{(\lambda_+-x)(x-\lambda_-)}{x} } \Big)\,dx. 
	\end{equation}
To show this identity \eqref{c(alpha) integral}, one can use the expansion of the error function \cite[Eq. (7.6.1)]{NIST} to write 
%\begin{equation*}
%		\erf(z)=\frac{2}{\sqrt{\pi}} \sum_{k=0}^{\infty}\frac{(-1)^k }{ k! \, (2k+1) } z^{2k+1}, 
%	\end{equation*}
%	see \cite[Eq.(7.6.1)]{NIST}. 
%	Using this, we have
	\begin{align*}
\frac{1}{2\alpha \sqrt{\pi}}  \erf\Big( \frac{\alpha}{2} \sqrt{ \frac{(\lambda_+-x)(x-\lambda_-)}{x} } \Big) 
	=\frac{1}{2\pi} \sum_{k=0}^{\infty}\frac{ 1 }{ 2^{2k}\,k!\, (2k+1) } \Big(  \frac{(\lambda_+-x)(x-\lambda_-)}{x} \Big) ^{k+\frac12} (-1)^k \alpha^{2k}. 
	\end{align*} 
Then, \eqref{c(alpha) integral} follows from term-by-term integration using the Euler integral formula \eqref{2F1 Euler integral} and the identity \cite[Eq. (15.4.17)]{NIST}.
%	$$
%	{}_2F_1(a,a+\tfrac12;2a+1;z)=(\tfrac12+\tfrac12 \sqrt{1-z})^{-2a}; 
%	$$
%	see e.g. \cite[Eq.(15.4.17)]{NIST}.
%	Therefore, by \eqref{c(alpha) summation}, it suffices to show that for each non-negative integer $k$, 
%	$$
%	\frac{1}{2\pi}\int_{\lambda_-}^{\lambda_+} \frac{ ( (\lambda_+-x)(x-\lambda_-) )^{k+\frac12} }{ x^{k+1} }\,dx= 2^k \frac{(2k+1)!!}{(k+1)!}. 
%	$$
%	For this, one can make use of   
	%Combining these, we obtain 
	%\begin{align*}
	%	&\quad \int_{\lambda_-}^{\lambda_+} \frac{ ( (\lambda_+-x)(x-\lambda_-) )^{k+\frac12} }{ x^{k+1} }\,dx = \frac{(\lambda_+-\lambda_-)^{2k+2}}{\lambda_-^{k+1}} \int_{0}^{1} \frac{t^{k+\frac12}(1-t)^{k+\frac12}}{ (1+\frac{\lambda_+-\lambda_-}{\lambda_-}t)^{k+1} }\,dt 
		%\\
		%&= \frac{\Gamma(k+\frac32)^2}{\Gamma(2k+3)} \frac{(\lambda_+-\lambda_-)^{2k+2}}{\lambda_-^{k+1}} {}_2F_1(k+1,k+\tfrac32;2k+3;1-\tfrac{\lambda_+}{\lambda_-})
	%	\\
	%	&=\frac{\Gamma(k+\frac32)^2}{\Gamma(2k+3)} \Big( 2\, \frac{\lambda_+-\lambda_-}{\sqrt{\lambda_+}+\sqrt{\lambda_-}} \Big)^{2k+2}=   \frac{\Gamma(k+\frac32)^2}{\Gamma(2k+3)}  2^{4k+4}= \frac{1}{2\pi} \frac{\Gamma(k+\frac32)^2}{\Gamma(2k+3)}  2^{4k+4}= \frac{(2k+1)!!}{(k+1)!}  2^k. 
	%\end{align*}
%	Here we have used $\lambda_\pm=(\sqrt{\varrho+1}\pm 1)^2$ for the second last identity. 
\end{rem}

\begin{rem}[Square root relation between complex and real eigenvalue densities at strong non-Hermiticity]
In \cite{Ta22}, Tarnowski introduced an interesting relation for general asymmetric random matrices at strong non-Hermiticity: the continuous part of the limiting density of real eigenvalues is proportional to the \emph{square root} of the limiting density of complex eigenvalues evaluated on the real line. Furthermore, it was observed that this relation holds for certain known real and complex eigenvalue densities, such as those of product matrices. 
In Theorem~\ref{Thm_global density} (i), one can observe that such a square root relation also holds for asymmetric Wishart matrix at strong non-Hermiticity. Namely, up to a multiplicative constant, the density $\rho_{\tau,\varrho}^s$ is the square root of that of the equilibrium measure \eqref{equilibrium msr}.
\end{rem}

\begin{rem}[Interpolating properties]
By using the well-known asymptotic behaviours of the error function \cite[Eqs. (7.6.1), (7.12.1)]{NIST} and \eqref{c(alpha) integral}, it follows that the function $c(\alpha)$ satisfies 
\begin{equation}
c(\alpha) \sim 
	\begin{cases}
		1 &\text{as} \quad \alpha \to 0 \, ,
		\\
		\frac{2}{\alpha \sqrt{\pi}}  &\text{as} \quad \alpha \to \infty \, .
	\end{cases} 
\end{equation}
These asymptotic behaviours allow us to observe the interpolating properties. 
In the symmetric limit $\alpha \to 0$, it is obvious that the expected number of real eigenvalues equals the total number of eigenvalues, which is consistent with $N c(\alpha) \sim N$ as $\alpha \to 0$. 
Conversely, in the non-Hermitian limit $\alpha \to \infty$, the asymptotic behaviour aligns with the result of Theorem~\ref{Thm_expected number} (i) at strong non-Hermiticity, since for $\tau$ given in \eqref{tau AH}, we have
\begin{align}
	\begin{split} 
	\Big( \frac{1}{\pi}\frac{1}{1-\tau^2}	c(\tau,\varrho) N \Big)^{\frac12} 	&\sim  	\Big( \frac{1}{\pi}\frac{2N}{1-\tau^2} (\sqrt{\lambda_+}-\sqrt{\lambda_-})  \Big)^{\frac12 }    \sim \frac{2N}{\alpha \sqrt{\pi}} .
	\end{split}
\end{align}
Similarly, one can also observe that 
\begin{equation}
\lim_{\alpha \to 0} \rho_{\alpha,\varrho}^w (x)  =\rho_{\textup{MP},\varrho}(x),  \qquad 
	\lim_{\alpha \to \infty} \rho_{\alpha,\varrho}^w (x) =\mathbbm{1}_{[\lambda_-,\lambda_+]}(x)\cdot \frac{1}{4\sqrt{x}}. 
\end{equation}
Here, the latter limit can be interpreted as $	\rho_{\tau,\varrho}^s(x) \big|_{\tau=1}$. Therefore, one can see that the density $\rho_{\alpha,\varrho}^w$ also exhibits a natural interpolating property, recovering the Marchenko-Pastur law in the symmetric limit.
\end{rem}

\begin{rem}[Comparison with the elliptic Ginibre matrix]
The asymmetric Wishart matrix is often referred to as the chiral version of the elliptic Ginibre matrices. 
As previously mentioned, analogous results to Theorems~\ref{Thm_expected number} and~\ref{Thm_global density} for the real eigenvalues of the elliptic Ginibre matrices were obtained in \cite{EKS94,FN07,FN08,BKLL23}. 
In general, the analysis of asymmetric Wishart matrices is more technical and involved than that of the elliptic Ginibre matrices due to the additional rectangular parameter. Compared to \cite{FN08, BKLL23}, determining the expected number of real eigenvalues is particularly more challenging. For the elliptic Ginibre matrices, the symmetry with respect to the origin leads to significant cancellations, particularly canceling the integral term in the $1$-point function (see \cite[Appendix A]{By24} for more details). Furthermore, this simplification, along with the evaluation of the integral \cite[Eq. (A.15)]{By24} involving the Hermite polynomial, allows the use of several properties of the hypergeometric function. However, such a simplification is no longer valid in our present case due to the absence of both the cancellations and an integral evaluation. In addition, contrary to the elliptic Ginibre matrices, where the associated potential is given by an exponential, the potential for the asymmetric Wishart matrices involves the modified Bessel function. This makes it impossible to utilise many advantageous properties, particularly the factorisation of exponentials.
\end{rem}

We conclude this section with a comment on our result on the integrable structure, Theorem~\ref{thm relationship R and C} below, which we regard as one of the most significant contributions of our work. 
In the study of correlation kernels of random matrices with real entries, one of the few successful approaches for asymptotic analysis involves establishing a relationship with their complex counterparts, for which more analytical techniques are available. This idea was first applied to the classical orthogonal ensembles \cite{AFNV00,Wi99}, where it was shown that the correlation kernel of real matrices can be decomposed into two components: one corresponding to the correlation kernel of their complex counterpart and the other expressed in terms of certain integrals involving the associated (skew-)orthogonal polynomials; cf. \eqref{Relationship LOE LUE}.  
This relationship, often referred to as the \textit{rank-one perturbation} of the correlation kernel, has proven to be highly useful in studying universality for Hermitian random matrices \cite{DG07a}, as well as in various applications, such as establishing connections between Toda and Pfaff lattices \cite{AV01}.  
However, beyond the classical ensembles, such a rank-one perturbation structure has been less extensively developed. (Nonetheless, recent progress has been made in \cite{FL20, LSYF22} on the \( q \)-deformed models.)

For the real eigenvalues of the elliptic Ginibre matrices, such a decomposition structure was established by Forrester and Nagao in \cite{FN07,FN08}. This result arises from remarkable manipulations of the correlation kernel, utilising various properties of Hermite polynomials. This decomposition has proven to be highly useful, as it facilitates the derivation of many important properties of the elliptic Ginibre matrices, including the expected number of real eigenvalues and their densities \cite{FN08, BKLL23, ABES23}, scaling limits \cite{BS09,FN08,FN07,AP14}, finite-size corrections \cite{BL23, BFK21}, spectral moments \cite{BF24b,By24}, fluctuations/large deviations in the number of real eigenvalues \cite{BMS23,Fo24,KPTTZ16}, and eigenvector statistics \cite{Fyo18,CFW24,FT21}.

Beyond the elliptic Ginibre matrices, the decomposition structure has not been established for any other class of real asymmetric matrices, to the best of our knowledge. In the following section, we derive it for asymmetric Wishart matrices, which will serve as a crucial tool for the subsequent asymptotic analysis. 
The statement in Theorem~\ref{thm relationship R and C} requires introducing additional notation, so it is not included in this section. Nevertheless, from a technical perspective and considering its potential applications in future studies, we regard this as one of the most significant contributions of our work.
In addition, we stress that the decomposition structure in Theorem~\ref{thm relationship R and C} is even new for the $\tau=0$, the product of rectangular GinOE matrices, while it recovers the previous findings in \cite{AFNV00,Wi99} for the $\tau=1$, the Laguerre orthogonal ensemble; see Remarks~\ref{Rem_tau=0 case} and ~\ref{Rem_tau=1 case} for further details.

\subsection*{Organisation of the paper} %The structure of this paper is as follows. 
In Section~\ref{Section_integrable} we present and prove the decomposition theorem (Theorem~\ref{thm relationship R and C}). Section~\ref{Section_prelim} provides the necessary preliminaries, including the strong asymptotic behaviour of orthogonal polynomials. Section~\ref{Section_strong} contains the proofs of the main results, Theorems~\ref{Thm_expected number} and~\ref{Thm_global density}, concerning strong non-Hermiticity, while Section~\ref{Section_weak} addresses analogous results for weak non-Hermiticity.
The analysis in Sections~\ref{Section_strong} and~\ref{Section_weak} focuses on the case where \(\nu = \varrho N\) with \(\varrho > 0\). The case of fixed \(\nu\) follows a similar argument but requires additional treatment due to the singularity at the origin. This aspect is addressed separately in Appendix~\ref{Appendix_fixed}, which may be helpful for readers interested in full details.

\section{Integrable structure of real eigenvalues} \label{Section_integrable}

In this section, we introduce the integrable structure of real eigenvalue and prove Theorem~\ref{thm relationship R and C}. 

As in the classical Laguerre orthogonal ensemble, the integrable properties of real eigenvalues of asymmtric Wishart matrices can be effectively described using 
the generalised Laguerre polynomial  
\begin{equation}
L_j^{(\nu)}(z):=\sum_{k=0}^{j}\frac{\Gamma(j+\nu+1)}{(j-k)!\,\Gamma(\nu+k+1)}\frac{(-z)^k}{k!}. 
\end{equation} 
Following \cite{AKP10}, we define  
\begin{align}
\begin{split}
\label{SOP Laguerre}
p_{2j}(x)&:= \tau^{2j} (2j)! \,L_{2j}^{\nu}\Bigl( \frac{x}{\tau}\Bigr), 
\\
p_{2j+1}(x) & := -\tau^{2j+1}(2j+1)! \, L_{2j+1}^{\nu}\Bigl( \frac{x}{\tau}\Bigr) +\tau^{2j-1}  (2j)! (2j+\nu) \, L_{2j-1}^{\nu}\Bigl( \frac{x}{\tau}\Bigr) .
\end{split}
\end{align}  
The modified Bessel function $K_\nu$ of the second kind (see e.g. \cite[Chapter 10]{NIST}) is defined by
\begin{equation}
K_\nu(z) = \frac{\pi}{2} \frac{I_{-\nu}(z)- I_\nu(z)}{\sin (\nu \pi)}, 
\end{equation}
where $I_\nu$ is given by \eqref{I nu}. 
We write 
\begin{align}
\begin{split}
\label{Phi function 1}
\Phi_j(y)
&:= \int_{\R} \sgn(y-v)|v|^{\nu/2} K_{\nu/2}\Bigl(\frac{|v|}{1-\tau^2}\Bigr)
\exp\Bigl(\frac{\tau\,v}{1-\tau^2}\Bigr) p_j(v) \, dv,
\end{split}
\end{align}
and 
\begin{align}
\label{SOP Laguerre nomalization}
r_j :=2\pi(1-\tau^2)\Gamma(2j+1)\Gamma(2j+1+\nu).
\end{align}
These are building blocks to define 
\begin{align}
\begin{split}
\label{NRW SNtau}
S_{N}(x,y) &:= |x|^{\nu/2}K_{\nu/2}\Bigl(\frac{|x|}{1-\tau^2}\Bigr)
\exp\Bigl(\frac{\tau\,x}{1-\tau^2} \Bigr)   
\sum_{j=0}^{N/2-1}
\frac{p_{2j+1}(x)\Phi_{2j}(y)-\Phi_{2j+1}(x)p_{2j}(y)}{r_j}.
\end{split}    
\end{align}

It is well known that the real eigenvalues of an asymmetric Wishart matrix form a Pfaffian point process. To be more precise, the $k$-th correlation function $\bfR_{N,k}$ can be expressed as
\begin{equation}
\label{def of rescaled k point correlation}
 \mathbf{R}_{N,k}(x_1,\cdots,x_k)
 = \Pf\Big[\mathbf{K}_{N}(x_j,x_l)\Big]_{j,l=1}^{k}, 
\end{equation}
where $\bfK_N$ is a $2\times 2$ correlation kernel of the form 
\begin{equation}
\label{def of rescaled 2 times 2 matrix kernel}
\mathbf{K}_{N}(x,y)
= \begin{pmatrix}
\mathbf{D}_{N}(x,y) &     \mathbf{S}_{N}(x,y)
\smallskip 
\\
-\mathbf{S}_{N}(y,x) &\widetilde{\mathbf{I}}_{N}(x,y)
\end{pmatrix}.  
\end{equation}
Here, $\mathbf{S}_N$ is expressed in terms of $S_N$ given in \eqref{NRW SNtau} as 
\begin{equation} \label{bfSN in terms of Laguerre rescaling}
\mathbf{S}_{N}(x,y) = N S_{N}(Nx,Ny);
\end{equation}
see \cite{AKP10,APS10}. 
The other functions $\mathbf{D}_{N}$ and $\widetilde{\mathbf{I}}_{N}$ are given in terms of $\mathbf{S}_N$ as
\begin{equation}
\label{def of rescaled DN IN}
\mathbf{D}_{N}(x,y):=-\frac{1}{N}\frac{\partial}{\partial y}\mathbf{S}_{N}(x,y),\qquad
\widetilde{\mathbf{I}}_{N}(x,y):=\int_x^y\mathbf{S}_{N}(t,y)\, dt+\frac{1}{2}\sgn(x-y).
\end{equation} 
In particular, the $1$-point function is written as 
\begin{equation} \label{def of RN 1pt real}
\bfR_{N}(x)= \mathbf{S}_{N}(x,x) = N R_N(Nx), \qquad R_N(x):=S_N(x,x). 
\end{equation}

As previously mentioned, in order to establish the decomposition structure, we define 
\begin{equation}  \label{def of mathcalK}
\mathcal{K}_N(x,y) \equiv \mathcal{K}_N^{(\nu)}(x,y) :=  \sum_{j=0}^{N-1}\frac{j!\, \tau^{2j} }{\Gamma(j+\nu+1)} L_{j}^{(\nu)}\Bigl(\frac{x}{\tau}\Bigr) L_{j}^{(\nu)}\Bigl(\frac{y}{\tau}\Bigr) . 
\end{equation}
This corresponds to the kernel of the complex non-Hermitian Wishart ensemble, which forms a determinantal point process; see e.g. \cite{BN24}.

\begin{thm}[\textbf{Decomposition structure of the correlation kernel}]\label{thm relationship R and C}
For any even integer $N$, $\tau \in (0,1)$ and $\nu >-1$, we have  
\begin{equation} \label{SN tau xy rewritten}
S_N(x,y)= S_{N,1}(x,y)+ \widetilde{S}_{N,1}(x,y)+S_{N,2}(x,y)  , 
\end{equation}
where 
\begin{align}
\begin{split}   \label{SN 1 xy nonrescaled}
S_{N,1}(x,y)&:= \frac{1}{\pi} \frac{1}{1-\tau^2} |xy|^{\nu/2}  K_{\nu/2}\Bigl(\frac{|x|}{1-\tau^2}\Bigr)  e^{ \frac{\tau}{1-\tau^2}x }
\\
&\quad \times    \bigg[  \tau yK_{\nu/2}\Bigl(\frac{|y|}{1-\tau^2}\Bigr) +|y| K_{\nu/2+1} \Bigl(\frac{|y|}{1-\tau^2}\Bigr)  \bigg] e^{\frac{\tau}{1-\tau^2}y}
   \cK_{N-1}(x,y), 
\end{split}
\\
\label{SN 1 xy nonrescaled tilde}
\widetilde{S}_{N,1}(x,y)&:= \frac{1}{\pi} y|x y|^{\nu/2}  K_{\nu/2}\Bigl(\frac{|x|}{1-\tau^2}\Bigr)   K_{\nu/2}\Bigl(\frac{|y|}{1-\tau^2}\Bigr)   e^{\frac{\tau}{1-\tau^2} (x+y) }  \partial_y\cK_{N-1}^{(\nu)}(x,y),
\\
\begin{split}       \label{SN 2 xy nonrescaled}
S_{N,2}(x,y) &=  \frac{1}{\pi} \frac{ \tau^{2N-3} }{ 1-\tau^2 } \frac{(N-1)! }{\Gamma(N+\nu-1)}   |x|^{\nu/2}
   K_{\nu/2}\Bigl(\frac{|x|}{1-\tau^2}\Bigr)
   e^{\frac{\tau}{1-\tau^2}x}  L_{N-1}^{(\nu)}\Bigl(\frac{x}{\tau}\Bigr)
   \\
   &\quad \times \int_\R \Big( \frac12-  \mathbbm{  1 }_{ (-\infty,y) } (t) \Big)  |t|^{\nu/2}
 K_{\nu/2}\Bigl(\frac{|t|}{1-\tau^2}\Bigr)
e^{ \frac{\tau }{1-\tau^2}t}  L_{N-2}^{(\nu)}\Bigl(\frac{t}{\tau}\Bigr) \, dt. 
\end{split}
\end{align} 
\end{thm}

Theorem~\ref{thm relationship R and C} also holds for $\tau = 0$, where one needs to take the proper limits of the Laguerre polynomials. This results in the replacement of $L_{k}^{(\nu)}(z/\tau)$ with monic polynomials that have the same leading coefficient.

As an immediate consequence of Theorem~\ref{thm relationship R and C}, the 1-point function $R_N$ can be written as  
\begin{equation}
R_N(x) = R_{N,1}(x)+\widetilde{R}_{N,1}(x)+R_{N,2}(x),
\end{equation}
where 
\begin{equation}
\label{def of normal RN1 and RN2}
R_{N,1}(x):=S_{N,1}(x,x), \qquad \widetilde{R}_{N,1}(x) :=\widetilde{S}_{N,1}(x,x), \qquad R_{N,2}(x):=S_{N,2}(x,x).
\end{equation}
Before proceeding with the proof, we first discuss the extremal cases of Theorem~\ref{thm relationship R and C}.

\begin{rem}[Products of two rectangular GinOEs; $\tau=0$] \label{Rem_tau=0 case}
As mentioned in Remark~\ref{Rem_special case}, if $\tau=0$, the model \eqref{Asymmetric Wishart matrix} becomes the two products of rectangular GinOEs.
In this case, the skew-orthogonal polynomials \eqref{SOP Laguerre} becomes simplified as 
\begin{equation*}
p_{2k}(z)=z^{2k},\qquad
p_{2k+1}(z)=z^{2k+1}-2j(2j+\nu)z^{2k-1}; %\qquad  r_k:=(2\sqrt{2\pi})^2\frac{\Gamma(2k+1)\Gamma(2k+\nu+1)}{2^{\nu}},  
\end{equation*} 
see also \cite[Eqs. (9.46), (9.47)]{BF24}. 
Note that in this case, 
\begin{equation*}
R_{N}(x) \Big|_{\tau=0} = \frac{|x|^{\nu/2} }{2\pi} K_{\nu/2}(|x|)
\int_{\R}\sgn(x-y) (x-y)|y|^{\nu/2}K_{\nu/2}(|y|)f_{N-2}^{(\nu)}(xy)\, dy, 
\end{equation*}
where 
\begin{equation*}
f_{N-2}^{(\nu)}(x):=\sum_{j=0}^{N-2}\frac{x^j}{j!\Gamma(j+\nu+1)}.
\end{equation*}
Then it follows from Theorem~\ref{thm relationship R and C} that 
\begin{align}
\begin{split}
\label{def of one point density RN tau=0}
R_{N}(x) \Big|_{\tau=0} &= \frac{|x|^{\nu}K_{\nu/2}(|x|)}{2\pi}
\Bigl(
xK_{\nu/2}(|x|)\partial_x f_{N-2}^{(\nu)}(x^2)
+2|x|K_{\nu/2+1}(|x|)f_{N-2}^{(\nu)}(x^2)
\Bigr)
\\
&\quad + \frac{x^{N-1}|x|^{\nu/2}K_{\nu/2}(|x|)}{\pi(N-2)!(N-2+\nu)!}
\int_{0}^{x}
y^{N-2+\nu/2}K_{\nu/2}(y) \, dy. 
\end{split}
\end{align} 
Even for the $\tau=0$ case, the relation \eqref{def of one point density RN tau=0} is new to our knowledge. 
\end{rem}

\begin{rem}[Wishart matrices; $\tau=1$] \label{Rem_tau=1 case}
We now discuss how the result in \cite{AFNV00} can be recovered from Theorem~\ref{thm relationship R and C}. 
The eigenvalue density of the Laguerre unitary ensemble with the weight $x^{\nu}e^{-x}$ is given by 
\begin{equation*}
\rho(x;\mathrm{UE}_{N-1}(x^{\nu}e^{-x}))
= x^{\nu}e^{-x}\sum_{k=0}^{N-2}\frac{k!}{\Gamma(k+\nu+1)}
L_k^{\nu}(x)^2.
\end{equation*}
Its counterpart for the Laguerre orthogonal ensemble with weight $x^{(\nu-1)/2}e^{-x/2}$ is given by 
\begin{equation*}
\rho(x;\mathrm{OE}_{N}(x^{(\nu-1)/2}e^{-x/2})) = \frac{x^{(\nu-1)/2}e^{-x/2}}{4} \sum_{k=0}^{N/2-1}\frac{\Phi_{2k}(x)q_{2k+1}(x)-\Phi_{2k+1}(x)q_{2k}(x)}{(2k)!(2k+\nu)!}\bigg|_{\tau=1}. 
\end{equation*}
Then, by \cite[Proposition 4.1]{AFNV00}, we have
\begin{align}
    \begin{split}
    \label{Relationship LOE LUE}
\rho(x;\mathrm{OE}_{N}(x^{(\nu-1)/2}e^{-x/2})) = \rho(x;\mathrm{UE}_{N-1}(x^{\nu}e^{-x})) + \rho_2(x),
    \end{split}
\end{align}
where 
\begin{equation*}
 \rho_2(x):= \frac12 \frac{(N-1)!}{ \Gamma(\nu+N-1)}x^{(\nu-1)/2}e^{-x/2}L_{N-1}^{(\nu)}(x)
\int_0^{\infty} \Big( \frac12-  \mathbbm{  1 }_{ (0,x) } (t) \Big)  t^{(\nu-1)/2}e^{-t/2}L_{N-2}^{(\nu)}(t)\, dt .
\end{equation*}
One can observe that the relation \eqref{Relationship LOE LUE} can be recovered from Theorem~\ref{thm relationship R and C} by taking $\tau \to 1$ limit. 
For this, recall that the modified Bessel function satisfies the asymptotic behaviour
\begin{equation}
\label{BesselK asymp infty}
K_{\nu/2}(z)
= \sqrt{\frac{\pi}{2z}} e^{-z} \Bigl( 1+\frac{\nu^2-1}{8z}+O(z^{-2}) \Bigl),\qquad z\to\infty,
\end{equation}
see e.g.  \cite[Eq. (10.40.2)]{NIST}.
Then as $\tau \to 1$, we have 
\[
K_{\nu/2}\Bigl(\frac{ |x| }{1-\tau^2} \Bigr)
e^{\frac{\tau}{1-\tau^2}x} 
 \sim  \sqrt{\frac{\pi(1-\tau^2)}{2|x|}} \times 
\begin{cases}
\displaystyle
e^{ -x/2 } & \text{if $x\geq0$},
\smallskip 
\\
e^{x/(1-\tau)}  & \text{if $x<0$}.
\end{cases} 
\]
Therefore for $ x \ge 0$, one can observe that 
\begin{equation*}  
R_{N,1}(x) \to \rho(x;\mathrm{UE}_{N-1}(x^{\nu}e^{-x})), \qquad \widetilde{R}_{N,1}(x) \to 0, \qquad R_{N,2}(x) \to \rho_2(x), 
\end{equation*}
as $\tau \to 1$. This gives \eqref{Relationship LOE LUE}.
\end{rem}

We now prove Theorem~\ref{thm relationship R and C}. The following identities of the generalised Laguerre polynomials will be used in the proof below.

\begin{itemize}
    \item Recurrence relations:
    \begin{align}
 L_k^{(\nu)}(z)&=L_{k}^{(\nu+1)}(z)-L_{k-1}^{(\nu+1)}(z),
\label{Laguerre1}
\\
 kL_{k}^{(\nu)}(z)&=(k+\nu)L_{k-1}^{(\nu)}(z)-zL_{k-1}^{(\nu+1)}(z)
\label{Laguerre2}.
\end{align}
\item Differentiation rules:
\begin{align}
 \partial_zL_k^{(\nu)}(z) & =-L_{k-1}^{(\nu+1)}(z),
\label{DLaguerre1}
\\
 \partial_z\bigl(z^{\nu}L_k^{(\nu)}(z)\bigr)&=(k+\nu)z^{\nu-1}L_{k}^{(\nu-1)}(z).
\label{DLaguerre2}
\end{align}
\end{itemize}
We also make use of some properties of the modified Bessel function.
\begin{itemize}
    \item Asymptotic behaviours: %\cite[Eq. (10.30.2), (10.30.3)]{NIST} 
\begin{equation} \label{K nu asymp 0}
K_{\nu}(z)  
\sim
\begin{cases}
    -\log z & \text{if $\nu=0$},
    \smallskip 
    \\
   (2/z)^\nu \,\Gamma(\nu)/2  & \text{if $\re\nu>0$},
\end{cases}
\quad\text{as $z\to0$}.
\end{equation}
\item Differentiation rules: 
\begin{align}
\label{DBessel 1}
\partial_t\Bigl[
t^{\frac{\nu}{2}+1}K_{\frac{\nu}{2}}(at)
\Bigr]
&=
t^{\frac{\nu}{2}}
\Bigl[
(\nu+1)K_{\frac{\nu}{2}}(at)-atK_{\frac{\nu}{2}+1}(at)
\Bigr],\\
\label{DBessel 2}
\partial_t\Bigl[
t^{\frac{\nu}{2}+1}K_{\frac{\nu}{2}+1}(at)
\Bigr]
&=
-at^{\frac{\nu}{2}+1}K_{\frac{\nu}{2}}(at).
\end{align}
\end{itemize}

Recall that $\mathcal{K}_N$ is defined in \eqref{def of mathcalK}.
Let us write 
\begin{equation} \label{def of KN hat}
\widehat{\mathcal{K}}_N(z,w):= e^{ \frac{\tau}{1-\tau^2}z } \mathcal{K}_N(z,w). 
\end{equation}
We first show the following lemma.

\begin{lem} \label{Lem_JN nu express}
Let  
\begin{align} \label{def of JN nu}
\mathrm{J}_{N}(x,y)
=\int_{y}^{\infty} |t|^{\frac{\nu}{2}}
 K_{\frac{\nu}{2}}\Bigl(\frac{|t|}{1-\tau^2}\Bigr)
(x-t)\,\widehat{\mathcal{K}}_{N-1}(t,x) \, dt . 
\end{align} 
Then we have 
\begin{align}
\begin{split} \label{JN nu v1}
\mathrm{J}_N(x,y)
&=  -  (1-\tau^2)^2 \,y\,|y|^{\frac{\nu}{2}} K_{\frac{\nu}{2}}\Bigl(\frac{y}{1-\tau^2} \Bigr)  \partial_y  \widehat{\cK}_{N-1}(y,x)  
\\
&\quad - (1-\tau^2) |y|^{\frac{\nu}{2}+1}K_{\frac{\nu}{2}+1}\Bigl(\frac{y}{1-\tau^2} \Bigr)
\widehat{\cK}_{N-1}(y,x)   - \mathrm{III}_N(x,y), 
\end{split}
\end{align}
where 
\begin{align}
\begin{split}
\mathrm{III}_N(x,y)
&:= \frac{(N-1)! \, \tau^{2N-3}}{\Gamma(N+\nu-1)}
\int_y^{\infty}
 |t|^{\frac{\nu}{2}}
 K_{\frac{\nu}{2}}\Bigl(\frac{|t|}{1-\tau^2}\Bigr)
e^{ \frac{\tau }{1-\tau^2}t} 
\\
&\quad\times
\bigg[ 
L_{N-2}^{(\nu)}\Bigl(\frac{t}{\tau}\Bigr)
L_{N-1}^{(\nu)}\Bigl(\frac{x}{\tau}\Bigr)
-
\tau^2 \, L_{N-1}^{(\nu)}\Bigl(\frac{t}{\tau}\Bigr)
L_{N-2}^{(\nu)}\Bigl(\frac{x}{\tau}\Bigr)
 \bigg]\, dt.
 \end{split}
\end{align}
In particular, we have 
\begin{equation}
\mathrm{J}_{N}(x) := \mathrm{J}_{N}(x,-\infty) = - \mathrm{III}_N(x,-\infty).
\end{equation} 
\end{lem}
\begin{proof}
One of the key ingredients of our proof is the following generalised Christoffel-Darboux formula recently established in \cite[Theorem 1.1]{BN24}: for any $\tau \in [0,1]$, $\nu \in \R$ and $N \in \mathbb{N}$,  
\begin{align}
\begin{split}  \label{CDI for complex kernel}
&\quad \bigg[ (1-\tau^2)\,z\,\partial_{z}^2  +\Big((1-\tau^2)(\nu+1)+2\tau z \Big)\partial_{z} + \frac{ \tau^2 z -w}{1-\tau^2}+(\nu+1)\tau \bigg]\cK_N(z,w)
\\
&= \frac{N!}{\Gamma(N+\nu)} \frac{\tau^{2N-1}}{1-\tau^2} 
 \bigg[ 
L_{N-1}^{(\nu)}\Bigl(\frac{z}{\tau}\Bigr)L_{N}^{(\nu)}\Bigl(\frac{w}{\tau}\Bigr)
-
\tau^2L_{N}^{(\nu)}\Bigl(\frac{z}{\tau}\Bigr)L_{N-1}^{(\nu)}\Bigl(\frac{w}{\tau}\Bigr)
 \bigg]. 
\end{split}
\end{align} 
By using \eqref{def of KN hat} and \eqref{CDI for complex kernel}, we have 
\begin{align}
\begin{split}
\label{x cKN}
x\widehat{\cK}_{N-1}(t,x)
&
=
(1-\tau^2)^2
\bigl[  t\,\partial_t^2
+  (\nu+1)\partial_t   \bigr]
\widehat{\cK}_{N-1}(t,x)
\\
&\quad
-e^{ \frac{\tau }{1-\tau^2}t} \frac{(N-1)! \, \tau^{2N-3}}{\Gamma(N+\nu-1)}
\bigg[ 
L_{N-2}^{(\nu)}\Bigl(\frac{t}{\tau}\Bigr)
L_{N-1}^{(\nu)}\Bigl(\frac{x}{\tau}\Bigr)
-
\tau^2 \, L_{N-1}^{(\nu)}\Bigl(\frac{t}{\tau}\Bigr)
L_{N-2}^{(\nu)}\Bigl(\frac{x}{\tau}\Bigr)
 \bigg].
 \end{split}
\end{align}

We first assume that $y\geq0$.
Note that by \eqref{x cKN}, we have 
\begin{align*}
&\quad \int_y^{\infty} t^{\frac{\nu}{2}} K_{\frac{\nu}{2}}\Bigl(\frac{t}{1-\tau^2} \Bigr) x \, \widehat{\mathcal{K}}_{N-1}(t,x)\, dt + \mathrm{III}_N(x,y)
\\
&=   (1-\tau^2)^2 \bigg[ \int_y^{\infty} t^{\frac{\nu}{2}+1} K_{\frac{\nu}{2}}\Bigl(\frac{t}{1-\tau^2} \Bigr)  \partial_t^2  \widehat{\cK}_{N-1}(t,x) \,dt 
+   (\nu+1) \int_y^{\infty} t^{\frac{\nu}{2}} K_{\frac{\nu}{2}}\Bigl(\frac{t}{1-\tau^2} \Bigr) \partial_t  
\widehat{\cK}_{N-1}(t,x) \,dt \bigg] . 
\end{align*}
Here, it follows from \eqref{DBessel 1} and integration by parts that 
\begin{align*}
&\quad \int_y^{\infty} t^{\frac{\nu}{2}+1} K_{\frac{\nu}{2}}\Bigl(\frac{t}{1-\tau^2} \Bigr)  \partial_t^2  \widehat{\cK}_{N-1}(t,x) \,dt  +  y^{\frac{\nu}{2}+1} K_{\frac{\nu}{2}}\Bigl(\frac{t}{1-\tau^2} \Bigr)  \partial_y  \widehat{\cK}_{N-1}(x,y) 
\\
& = - \int_y^\infty t^{\frac{\nu}{2}} \bigg[
(\nu+1)K_{\frac{\nu}{2}}\Big(\frac{t}{1-\tau^2}\Big)- \frac{t}{1-\tau^2} K_{\frac{\nu}{2}+1}\Big(\frac{t}{1-\tau^2}\Big) 
 \bigg] \partial_t   \widehat{\cK}_{N-1}(t,x) \,dt ,
\end{align*}
where we have used \eqref{BesselK asymp infty}. 
Combining all of the above, it follows that for $y \ge 0$, 
\begin{align*}
\begin{split}
\mathrm{J}_{N}(x,y) &= - (1-\tau^2)^2 y^{\frac{\nu}{2}+1} K_{\frac{\nu}{2}}\Bigl(\frac{y}{1-\tau^2} \Bigr)  \partial_y  \widehat{\cK}_{N-1}(y,x)  
\\
&\quad
+ (1-\tau^2)\int_y^{\infty}t^{\frac{\nu}{2}+1} K_{\frac{\nu}{2}+1}\Bigl(\frac{t}{1-\tau^2} \Bigr)
\partial_t\widehat{\cK}_{N-1}(t,x)\, dt
\\
&\quad  - \int_y^{\infty} t^{\frac{\nu}{2}+1} K_{\frac{\nu}{2}}\Bigl(\frac{t}{1-\tau^2} \Bigr)
\widehat{\mathcal{K}}_{N-1}(t,x) \, dt -  
\mathrm{III}_N(x,y). 
\end{split}
\end{align*} 
Furthermore, by using \eqref{DBessel 2} and integration by parts, we have  
\begin{align*}
&\quad (1-\tau^2)\int_y^{\infty}t^{\frac{\nu}{2}+1} K_{\frac{\nu}{2}+1}\Bigl(\frac{t}{1-\tau^2} \Bigr)
\partial_t\widehat{\cK}_{N-1}(t,x)\, dt
\\
&= - (1-\tau^2) 
y^{\frac{\nu}{2}+1}K_{\frac{\nu}{2}+1}\Bigl(\frac{y}{1-\tau^2} \Bigr)
\widehat{\cK}_{N-1}(y,x) 
+ \int_{y}^{\infty}t^{\frac{\nu}{2}+1} K_{\frac{\nu}{2}}\Bigl(\frac{t}{1-\tau^2} \Bigr) \widehat{\cK}_{N-1}(t,x) \, dt.
\end{align*}
Therefore, we obtain the desired expression \eqref{JN nu v1}. 

For the case $y<0$, we consider the decomposition 
\begin{align*}
\mathrm{J}_N(x,y)
& =  \mathrm{J}_N(x,0) + \int_{y}^{0} (-t)^{\frac{\nu}{2}} K_{\frac{\nu}{2}}\Bigl(\frac{-t}{1-\tau^2} \Bigr) x\, \widehat{\mathcal{K}}_{N-1}(t,x) \, dt + \int_{y}^{0} (-t)^{\nu/2+1} K_{\frac{\nu}{2}}\Bigl(\frac{-t}{1-\tau^2} \Bigr) \widehat{\mathcal{K}}_{N-1}(t,x) \, dt. 
\end{align*}
Then the desired expression \eqref{JN nu v1} follows from similar computations as above together with \eqref{K nu asymp 0}. 
\end{proof}

Now, we are ready to show Theorem~\ref{thm relationship R and C}.

\begin{proof}[Proof of Theorem~\ref{thm relationship R and C}]
Recall that $S_N$ is given by \eqref{NRW SNtau}.
It was shown in \cite[Subsection 4.2]{APS10} that $S_N$ can be expressed as  
\begin{align}
\begin{split} \label{NRW SNtau v2}
S_{N}(x,y)  =\frac{\tau}{2\pi}
  \frac{1}{1-\tau^2}
& \int_{\R} \sgn(y-t)
 |xt|^{\frac{\nu}{2}}
 K_{\frac{\nu}{2}}\Bigl(\frac{|x|}{1-\tau^2}\Bigr)
 K_{\frac{\nu}{2}}\Bigl(\frac{|t|}{1-\tau^2}\Bigr)
 \exp\Bigl(\frac{\tau}{1-\tau^2}(x+t)\Bigr)
 \\
 &\quad
 \times
 \Bigl(t\frac{\partial}{\partial t}-x\frac{\partial}{\partial x}-\frac{t-x}{\tau} \Bigr)
 \mathcal{K}_{N-1}(t,x) \, dt.
\end{split}
\end{align} 
By using \eqref{Laguerre2} and \eqref{DLaguerre1}, we have 
\begin{align*}
&\quad\Bigl(
t\frac{\partial}{\partial t}
-x\frac{\partial}{\partial x}-\frac{t-x}{\tau}
    \Bigr)\mathcal{K}_{N-1}(x,t)
%    \\
%    &=
%\frac{1}{1-\tau^2}
%\frac{x-t}{\tau}\mathcal{K}_{N-1}^{(\nu)}(x,t)
%+
%\frac{1}{1-\tau^2}\frac{(N-1)!\tau^{2N-3}}{\Gamma(N+\nu)}
%\Bigl( 
%tL_{N-2}^{(\nu+1)}\Bigl(\frac{t}{\tau}\Bigr)
%L_{N-1}^{(\nu)}\Bigl(\frac{x}{\tau}\Bigr)
%-
%xL_{N-2}^{(\nu+1)}\Bigl(\frac{x}{\tau}\Bigr)
%L_{N-1}^{(\nu)}\Bigl(\frac{t}{\tau}\Bigr)
%\Bigr)
\\
&=
\frac{1}{1-\tau^2}
\frac{x-t}{\tau}\mathcal{K}_{N-1}(x,t)
+
\frac{1}{1-\tau^2}\frac{(N-1)!\tau^{2N-2}}{\Gamma(N+\nu-1)}
\bigg( 
L_{N-1}^{(\nu)}\Bigl(\frac{x}{\tau}\Bigr)
L_{N-2}^{(\nu)}\Bigl(\frac{t}{\tau}\Bigr)
-
L_{N-2}^{(\nu)}\Bigl(\frac{x}{\tau}\Bigr)
L_{N-1}^{(\nu)}\Bigl(\frac{t}{\tau}\Bigr)
\bigg).
\end{align*}
Therefore \eqref{NRW SNtau v2} can be written as 
\begin{equation*}
S_{N}(x,y) = \frac{1}{2\pi}
  \frac{1}{(1-\tau^2)^2}
   |x|^{\frac{\nu}{2}}
   K_{\frac{\nu}{2}}\Bigl(\frac{|x|}{1-\tau^2}\Bigr)
   \exp\Bigl(\frac{\tau}{1-\tau^2}x\Bigr)
\Big[
\mathrm{I}_{N}(x,y)
+
\mathrm{II}_{N}(x,y)
\Big],
\end{equation*}
where 
\begin{align*} 
\begin{split}
%\label{SN divided IN}
 \mathrm{I}_{N}(x,y) &:= \mathrm{J}_{N}(x) -2\,\mathrm{J}_{N}(x,y),
\end{split} 
\\
\begin{split}
%\label{SN divided IIN}
 \mathrm{II}_{N}(x,y)  &:= \frac{2(N-1)!\tau^{2N-1}}{\Gamma(N+\nu-1)} 
  \int_\R \Big( \frac12-  \mathbbm{  1 }_{ (-\infty,y) } (t) \Big)  
 |t|^{\frac{\nu}{2}}
 K_{\frac{\nu}{2}}\Bigl(\frac{|t|}{1-\tau^2}\Bigr)
 e^{\frac{\tau}{1-\tau^2}t}
 \\
 &\quad\times
\Bigl[
L_{N-1}^{(\nu)}\Bigl(\frac{x}{\tau}\Bigr)
L_{N-2}^{(\nu)}\Bigl(\frac{t}{\tau}\Bigr)
-
L_{N-2}^{(\nu)}\Bigl(\frac{x}{\tau}\Bigr)
L_{N-1}^{(\nu)}\Bigl(\frac{t}{\tau}\Bigr)
\Bigr]\, dt.
\end{split}  
\end{align*}
Here, $\mathrm{J}_{N}$ is given by \eqref{def of JN nu}. 
Then by combining Lemma~\ref{Lem_JN nu express} with 
\[
\partial_t
\widehat{\cK}_{N-1}(t,x)
=
\frac{\tau}{1-\tau^2}e^{\frac{\tau}{1-\tau^2}t}
\cK_{N-1}(t,x)
+
e^{\frac{\tau}{1-\tau^2}t}
\partial_t\cK_{N-1}(t,x),
\]
straightforward computations yield the theorem.  
\end{proof}

\begin{rem}[Negative non-Hermiticity parameter $\tau \in (-1,0)$]
We have primarily focused on the case where the non-Hermiticity parameter $\tau$ is nonnegative, but the model \eqref{Asymmetric Wishart matrix} is naturally well-defined for $\tau \in (-1,0)$ as well. 
Indeed, for $-1<\tau<0$, by \eqref{NRW SNtau} and the change of variables $x \mapsto -x$, $y \mapsto -y$, the building kernel \eqref{NRW SNtau} for $-1<\tau<0$ coincides with \eqref{NRW SNtau} for $0<\tau<1$ for all $x,y \in \mathbb{R}$.  
Therefore most of the results presented in this paper extend naturally to the case $-1<\tau<0$.
\end{rem}

\section{Preliminary analysis for the correlation kernel} \label{Section_prelim}

In this section, we gather some preliminary results needed to prove our main results. While this requires some book-keeping, it is otherwise quite elementary.

\subsection{Plancherel-Rotach asymptotis of the Laguerre polynomials} \label{Subsection_Laguerre polynomials asymptotics}

We present the Plancherel-Rotach asymptotics of the Laguerre polynomials.
There are several methods for obtaining the asymptotic behaviour of classical orthogonal polynomials with complex arguments, such as the Wentzel-Kramers-Brillouin (WKB) approximation (see e.g. \cite[Appendix B]{LR16}) or the Deift-Zhou nonlinear steepest descent method for Riemann--Hilbert problems (see e.g. \cite{DKMVZ99,DIW15}).

In this framework, the conformal map associated with the droplet \eqref{droplet} plays a crucial role.
For $\varrho\geq0$, we define 
\begin{equation}
\label{Conformal varrho Map2}
\psi_{\varrho}(z) := \frac{ z-\tau(2+\varrho)+\sqrt{(z-\tau(2+\varrho))^2-4(1+\varrho)\tau^2}  }{2\sqrt{1+\varrho}},  
\qquad z \in \C\backslash[F_-^{(\varrho)},F_+^{(\varrho)}], 
\end{equation}
where 
\begin{equation}
\label{foci ellips}
F_\pm \equiv F_\pm ^{(\varrho)} := \tau(\sqrt{1+\varrho} \pm 1)^2  
\end{equation}
are the foci of the ellipse \eqref{droplet}. In particular, we write $\psi(z)\equiv\psi_{0}(z)$. 
The function $\psi_{\varrho}$ defines a conformal map from the exterior of the droplet \eqref{droplet} to the exterior of the unit disc; see \cite{ABK21}.

\begin{lem}[\textbf{Exponential regime}]
\label{lem Exponential regime}
Fix $\tau\in[0,1)$, and let $m$ and $r$ be fixed integers. %Then we have the following.
\begin{itemize}
    \item[\textup{(i)}] \textup{\textbf{(Large parameter $\nu=\varrho\,N$ with $\varrho>0$)}} Assume $\nu=\varrho N$ with $\varrho>0$. 
   As $N\to\infty$, we have
\begin{equation}
\label{Laguerre tau fix exponential}
L_{N+r}^{(\nu+m)}\Bigl( \frac{N}{\tau}z \Bigr) 
=
\frac{1}{\sqrt{2\pi N}}
\frac{(-1)^{N+r}}{\tau^{N+r}}
\sqrt{\varrho+1}^{\frac{2r+1}{2}}\sqrt{\psi'_{\varrho}(z)}\psi_{\varrho}(z)^r
\varpi(z)^m e^{Ng_\tau^{(\varrho)}(z) }\Bigl(1+O(\frac{1}{N}) \Bigr),
\end{equation}
uniformly over any compact subset of $\C \backslash [F_-^{(\varrho)},F_+^{(\varrho)}]$, where
%\begin{comment}
\begin{align}
\begin{split}
\label{g tau varrho}
g_\tau^{(\varrho)}(z) &:= \frac{z-\sqrt{(z-(2+\varrho)\tau)^2-4(\varrho+1)\tau^2}}{2\tau} 
\\
&\quad + (1+\varrho)\log\Bigl(z-\tau(2+\varrho)+ \sqrt{(z-\tau(2+\varrho))^2-4(\varrho+1)\tau^2}\Bigr)
\\
& \quad -\varrho\log\Bigl(z-\tau\varrho+\sqrt{(z-\tau(2+\varrho))^2-4(\varrho+1)\tau^2} \Bigr) -\log 2 -\frac{\varrho}{2}
\end{split}
\end{align}
and 
\begin{equation}
\varpi(z) :=\frac{2\tau(1+\varrho)}{z+\varrho\tau-\sqrt{(z-\tau(2+\varrho))^2-4(\varrho+1)\tau^2}}.
\end{equation}
    \item[\textup{(ii)}] \textup{\textbf{(Fixed parameter $\nu>-1$)}}
 Assume $\nu>-1$ is fixed. As $N\to\infty$, we have
    \begin{equation}
\label{ExpRegimeLaguerre}
L_{N+r}^{(\nu)}\Bigl( \frac{N}{\tau}z \Bigr)=
\frac{1}{\sqrt{2\pi N}}
\frac{(-1)^{N+r}}{\tau^{N+r}}
\frac{\psi(z)^{r+\frac{\nu}{2}}\sqrt{\psi'(z)}}{z^{\nu/2}}e^{Ng_{\tau}(z)}
\Bigl(1+O(\frac{1}{N}) \Bigr),
\end{equation}
uniformly over any compact subset of $\C\backslash[0,4\tau]$, where $g_{\tau} \equiv g_{\tau}^{(0)}$.
\end{itemize}  
\end{lem}

Lemma~\ref{lem Exponential regime} for $\nu=\varrho\,N$ with $\varrho>0$ can be derived from \cite[Theorem~1 (d)]{DIW15}. For fixed $\nu>-1$, it can be found in \cite{Van07}.
 
As a consequence of Lemma~\ref{lem Exponential regime} and \cite[Theorem~1 (a), (b), and (c)]{DIW15}, we obtain the following estimate, which will be useful in the subsequent asymptotic analysis.

\begin{lem}\label{Lem BDLaguerre large order}
Fix $\tau\in(0,1)$, and we assume $\nu=\varrho N$ with $\varrho>0$. Let $r$ be a fixed integer. Then we have 
\begin{equation}
\label{def of Laguerre totoal bound}
\Bigl| e^{-N g_{\tau}^{(\varrho)}(z)}L_{N+r}^{(\nu)}\Bigl(\frac{N}{\tau}z\Bigr) \Bigr|
\leq 
\frac{1}{\tau^{N+r}}\cdot O(\frac{1}{\sqrt{N}}),
\end{equation}
where the $O(N^{-1/2})$-term is uniform for $z\in\C$.
Furthermore, \eqref{def of Laguerre totoal bound} also holds for fixed $\nu>-1$ including $\varrho=0$ over $\C\backslash B_{\delta}(0)$ with a small $\delta>0$.
\end{lem}

%Lemma~\ref{lem Exponential regime} plays a crucial role in the asymptotic analysis both in the regime of strong non-Hermiticity and in the exterior of the droplet at weak non-Hermiticity.

For the weakly non-Hermitian case, we require the asymptotic behaviour of the Laguerre polynomials in both the oscillatory and critical regimes. The following result, stated in for instance \cite[Theorem~1 (c)]{DIW15} and \cite[Theorem~2.4 (b)]{Van07}, is presented here using notation adapted to our setting.

\begin{lem}[\textbf{Oscillatory regime}]
\label{Lem Oscillatory regime}
Let $\tau=1-\frac{\alpha^2}{2N}$ with fixed $\alpha>0$, and let $r$ and $m$ be fixed integers. 
\begin{itemize}
    \item[\textup{(i)}] \textup{\textbf{(Large parameter $\nu=\varrho\,N$ with $\varrho>0$)}} Assume $\nu=\varrho N$ with $\varrho>0$. Given $\delta>0$ small, as $N\to\infty$, we have 
    \begin{align}
    \begin{split}
   L_{N+r}^{(\nu+m)}\Bigl(\frac{N}{\tau}z\Bigr)  &= 
   \frac{\sqrt{2}(-1)^{N+r}(1+\varrho)^{\frac{N(1+\varrho)+r+m}{2}+\frac{1}{4}}
   e^{\frac{2N+\alpha^2}{4}(z-\varrho)}\cos(\Psi_{N+r,m}(z))}{\sqrt{\pi N}(z-\lambda_-)^{\frac{1}{4}} (\lambda_+-z)^{\frac{1}{4}}z^{\frac{N\varrho+m}{2}}} 
\Bigl(1+O(\frac{1}{N})\Bigr),
    \end{split}
    \end{align}
    uniformly for $z\in[\lambda_-+\delta,\lambda_+-\delta]$.
    Here,  
    \begin{equation}
    \label{def of Psi N+r,m}
\Psi_{N+r,m}(z) := \boldsymbol{\psi}_N(z) -\frac{2r+1+m}{2} \arccos\boldsymbol{p}(z) + m\arccos \boldsymbol{q}(z) +\frac{\pi}{4}, 
    \end{equation}
where 
\begin{align}
\boldsymbol{\psi}_N(z)& := N\boldsymbol{\psi}^{(1)}(z)+\boldsymbol{\psi}^{(2)}(z), \qquad \boldsymbol{p}(z):=
\frac{z-(\varrho+2)}{2\sqrt{1+\varrho}},
\qquad \boldsymbol{q}(z):=\frac{(2+\varrho)z-\varrho^2}{2\sqrt{1+\varrho}z},
\\
\begin{split}
\boldsymbol{\psi}^{(1)}(z)
&:=
\frac{1}{2}
\Bigl[
\sqrt{(z-\lambda_-)(\lambda_+-z)}
+
\varrho\arccos\Bigl(\frac{(2+\varrho)z-\varrho^2}{2\sqrt{1+\varrho}z}\Bigr)
-
(2+\varrho)\arccos\Bigl(\frac{z-(\varrho+2)}{2\sqrt{1+\varrho}}\Bigr)
\Bigr],
\end{split}
\\
\boldsymbol{\psi}^{(2)}(z) &:=\frac{c^2}{2}\sqrt{(z-\lambda_-)(\lambda_+-z)} .
\end{align}
   \item[\textup{(ii)}] \textup{\textbf{(Fixed parameter $\nu>-1$)}}
 Assume $\nu>-1$ is fixed. Given $\delta,\delta'>0$ small, as $N\to\infty$, we have
   \begin{align}
L_{N+r}^{(\nu+m)}\Bigl(\frac{N}{\tau}x\Bigr) 
&=
\frac{(-1)^{N+r}N^{N+r}}{(N+r)!}
\frac{2\,e^{(\frac{x}{2}-1)N+\frac{\alpha^2}{4}x}}{x^{1/4}(4-x)^{1/4}x^{(\nu+m)/2}}
\mathcal{B}_1^{(m,r)}(x)
\Bigl( 
1+O(\frac{1}{N})
\Bigr),
\end{align}
uniformly for $x\in[\delta',4-\delta]$, where 
\begin{align*}
\mathcal{B}_{1}^{(m,r)}(x)&:= \cos\Bigl( b_{-1}(x)N+b_{0}^{(\nu,m,\alpha,r)}(x) \Bigr). 
\end{align*}
Here, 
\begin{align*}
b_{-1}(x) &:= \frac{1}{2}\sqrt{x(4-x)}-2\arccos\Big(\frac{\sqrt{x}}{2}\Big),
\\
b_{0}^{(\nu,m,\alpha,r)}(x)&:= \frac{\alpha^2}{4}\sqrt{(4-x)x} -(2r+\nu+m+1)\arccos \Big(\frac{\sqrt{x}}{2} \Big) + \frac{\pi}{4}.
%b_{0,q}^{(\nu,m,\alpha,r)}(x)&:= \frac{\alpha^2}{4}\sqrt{(4-x)x} -(2r+\nu+m+q)\arccos\frac{\sqrt{x}}{2} + \frac{\pi}{4}.
\end{align*}
\end{itemize}
\end{lem}

%%%%%%%%%%%%%%%%%%%%%%%%%%
%%%%%%%%%%%%%%%%%%%%%%%%%%

Recall that the Airy function of the first kind is given by 
\begin{equation}
\Ai(x):= \frac{1}{\pi} \int_0^\infty \cos\Big( \frac{t^3}{3}+xt \Big)\,dt; 
\end{equation}
see e.g. \cite[Chapter 9]{NIST}.
The following can be derived from \cite[Theorem~1 (a) and (b)]{DIW15} for the large parameter $\nu=\varrho\,N$ with $\varrho>0$ and \cite[Theorem 2.4]{Van07} or \cite[Theorem 8.22.8 (c)]{Szego39} for fixed $\nu>-1$.

%%%%%%%%%%%%%%%%%%%%%%%%%%
%%%%%%%%%%%%%%%%%%%%%%%%%%

\begin{lem}[\textbf{Airy regime}]
\label{Lem Critical regime}
Let $\tau=1-\frac{\alpha^2}{2N}$ with fixed $\alpha>0$, and let $r$ and $m$ be fixed integers.
%Then we have the following:
\begin{itemize}
    \item[\textup{(i)}] \textup{\textbf{(Large parameter $\nu=\varrho\,N$ with $\varrho>0$)}} Assume $\nu=\varrho N$ with $\varrho>0$.  
Given $\delta>0$ small, as $N\to\infty$, we have
\begin{align}
\begin{split}
\label{def of Critical Laguerre 1}
L_{N+r}^{(\nu+m)}\Bigl(\frac{N}{\tau}z\Bigr)   
&=
\frac{(-1)^{N+r}}{\sqrt{2}N^{\frac{1}{3}}}
\frac{(1+\varrho)^{\frac{(1+\varrho)N}{2}+\frac{r+m}{2}+\frac{1}{24}}}{(\sqrt{1+\varrho}+1)^{\frac{1}{3}}z^{\frac{N\varrho+m}{2}}}
(z-\lambda_-)^{\frac{1}{4}}
\exp\Bigl( 
\frac{2N+\alpha^2}{4}(z-\varrho)
\Bigr)
\\
&\quad
\times
\Ai\Bigl( 
N^{\frac{2}{3}} \frac{(1+\varrho)^{\frac{1}{6}}}{(\sqrt{1+\varrho}+1)^{\frac{4}{3}}}
(z-\lambda_+)
\Bigr)
\Bigl( 1 + O(|z-\lambda_+|^{\frac{3}{2}}) \Bigr),
\end{split}    
\end{align}
uniformly for $|z-\lambda_+|<\delta$, and 
\begin{align}
\begin{split}
\label{def of Critical Laguerre 2}
L_{N+r}^{(\nu+m)}\Bigl(\frac{N}{\tau}z\Bigr)   
&=
\frac{1}{\sqrt{2}N^{\frac{1}{3}}}
\frac{(1+\varrho)^{\frac{(1+\varrho)N}{2}+\frac{r+m}{2}+\frac{1}{24}}}{(\sqrt{1+\varrho}-1)^{\frac{1}{3}}z^{\frac{N\varrho+m}{2}}}
(\lambda_+-z)^{\frac{1}{4}}
\exp\Bigl( 
\frac{2N+\alpha^2}{4}(z-\varrho)
\Bigr)
\\
&\quad
\times
\Ai\Bigl( 
-N^{\frac{2}{3}} \frac{(1+\varrho)^{\frac{1}{6}}}{(\sqrt{1+\varrho}-1)^{\frac{4}{3}}}
(z-\lambda_-)
\Bigr)
\Bigl( 1 + O(|z-\lambda_-|^{\frac{3}{2}}) \Bigr),
\end{split}    
\end{align}
uniformly for $|z-\lambda_-|<\delta$. 
   \item[\textup{(ii)}] \textup{\textbf{(Fixed parameter $\nu>-1$)}} Assume $\nu>-1$ is fixed. Given $\delta>0$ small, as $N\to\infty$, we have
\begin{equation}
 L_{N+r}^{(\nu)}\Bigl( 
\frac{N}{\tau}u
\Bigr)
=
(-1)^{N+r}2^{-\nu-\frac{1}{3}}N^{-\frac{1}{3}}e^{\frac{N}{2}u+\frac{\alpha^2}{4}u}
\Bigl[
\Ai\Bigl(\frac{N^{\frac{2}{3}}}{2^{\frac{4}{3}}}(u-4)\Bigr)
+O(\frac{1}{N^{\frac{1}{3}}})
\Bigr],
\end{equation}
uniformly for $u\in\{x\in\R:|x-4|<\delta\}$. 
\end{itemize}
\end{lem}

Recall that the Bessel function $J_\nu$ is given by 
\begin{equation}
J_\nu(z):= \sum_{k=0}^\infty \frac{(-1)^k\,(z/2)^{2k+\nu}}{k!\,\Gamma(\nu+k+1)}; 
\end{equation}
see \cite[Chapter 10]{NIST}. 
The following asymptotic behaviour near the origin can be found in \cite{Szego39}.
%from \cite[Eq.(4)]{Szego39} near the origin, we will use the following. 

\begin{lem}[\textbf{Bessel regime}]
\label{lem Laguerre bessel}
For a fixed $\nu>-1$, as $N\to\infty$, we have 
\begin{equation}
\label{def of J bessel asymptotic 2}
 (N+\nu)^{-\nu}L_{N}^{(\nu)}\Bigl(\frac{z}{N+\nu}\Bigr)
 =
 z^{-\frac{\nu}{2}}J_{\nu}(2\sqrt{z})
 -
 \frac{1}{2N}z^{-\frac{\nu-2}{2}}J_{\nu-2}(2\sqrt{z})
 +
O(\frac{1}{N^2}), 
\end{equation}
uniformly for $z$ in a compact subset of $\C$.  
\end{lem}

%%%%%%%%%%%%%%%%%%%%%%%%%%%%%%
%%%%%%%%%%%%%%%%%%%%%%%%%%%%%%
%%%%%%%%%%End of Proof%%%%%%%%%%%%%%%
%%%%%%%%%%%%%%%%%%%%%%%%%%%%%%
%%%%%%%%%%%%%%%%%%%%%%%%%%%%%%

\subsection{Complex correlation kernel and rescaling}

Recall that $\mathcal{K}_N$ is given by \eqref{def of mathcalK}. In order to discuss the rescaled correlation kernels, it is convenient to introduce 
\begin{equation} \label{def of SN} 
S_N^{ \rm c }(x,y) := \frac{2N^{\nu+2}}{1-\tau^2} \mathcal{K}_N(Nx,Ny)= \frac{2N^{\nu+2}}{1-\tau^2}\sum_{j=0}^{N-2}\frac{j!\, \tau^{2j} }{\Gamma(j+\nu+1)}
L_{j}^{(\nu)}\Bigl(\frac{N}{\tau}x\Bigr) L_{j}^{(\nu)}\Bigl(\frac{N}{\tau}y\Bigr) 
\end{equation}
and $R_N^{\rm c}(x):= S_{N}^{\mathrm{c}}(x,x).$ 
For $x \not =0$, we denote by 
\begin{equation}
\label{omega c weight}
\omega^{\rm c }_N(x) 
:=  |x|^{\nu}  K_{\nu}\Bigl(\frac{2N}{1-\tau^2}|x|\Bigr) 
\exp\Bigl( \frac{2N\tau}{1-\tau^2}  x \Bigr)
\end{equation}
the weight function associated with the complex non-Hermitian Wishart ensemble. Here, by using the asymptotic behaviour of $K_\nu$ near the origin (see e.g. \cite[Eq. (10.30.2)]{NIST}), we set 
\begin{equation} \label{omega c weight at 0}
\omega^{\rm c }_N(0) = \frac{\Gamma(\nu)}{2} \Big( \frac{N}{1-\tau^2}\Big)^{-\nu}.  
\end{equation}
Then we define    
\begin{equation}
\label{hat bfK N complex} 
\bfK_{N}^{\mathrm{c}}(x,y) := \sqrt{\omega_N^{\text{c}}(x)\omega_N^{\text{c}}(y)}
S_{N}^{\mathrm{c}}(x,y),  \qquad \bfR_{N}^{\mathrm{c}}(x):=\bfK_{N}^{\mathrm{c}}(x,x). 
\end{equation}
This corresponds to the weighted kernel for the complex non-Hermitian Wishart ensemble; see e.g. \cite{ABK21,BN24}. 

We also define the weight functions  
\begin{align}
\label{def of omega N 1}
\omega_{N,1}(x) & := K_{\nu/2}\Big(\frac{N}{1-\tau^2}|x|\Big)^2  K_{\nu}\Big(\frac{2N}{1-\tau^2}|x|\Big)^{-1} , 
\\ \label{def of omega N 2}
\omega_{N,2}(x)& := K_{\nu/2}\Big(\frac{N}{1-\tau^2}|x|\Big) K_{\nu/2+1} \Big(\frac{N}{1-\tau^2}|x|\Big)  K_{\nu} \Big(\frac{2N}{1-\tau^2}|x|\Big)^{-1 } , 
\\ \label{def of omega N 3}
\omega_{N,3}(x)& := K_{\nu-1}\Big(\frac{2N}{1-\tau^2}|x|\Big)  K_{\nu}\Big(\frac{2N}{1-\tau^2}|x|\Big)^{-1}.
\end{align}  
Recall that the one-point density $\bfR_N$ of the real eigenvalues is given by \eqref{def of RN 1pt real}. 
  
\begin{lem}[\textbf{Decomposition structure of the $1$-point function}] \label{Lem_RN RN hat RN2 decomp}
We have
\begin{equation}
    \bfR_N(x)=\widehat{\bfR}_{N,1}(x)+\bfR_{N,2}(x), 
\end{equation}
where 
\begin{equation}
\label{bfRN one point density}
\widehat{\mathbf{R}}_{N,1}(x) := \frac{|x|}{2\pi} \Big(  \omega_{N,1}(x) \, \omega_{N,3}(x) + \omega_{N,2}(x) \Big) \bfR_{N}^{\mathrm{c}}(x) + \frac{1}{2\pi}
 \frac{1-\tau^2}{2N} x\,\omega_{N,1}(x) \partial_x \bfR_{N}^{\mathrm{c}}(x),
\end{equation}
and 
\begin{equation}
\label{def of bfRN 2 xy}
\bfR_{N,2}(x) := \frac{N^3}{2\pi(1-\tau^2)} \mathbf{R}_{N,2}^{(1)}(x)\cdot \mathbf{R}_{N,2}^{(2)}(x).
\end{equation}
Here, 
\begin{align}
\label{def of bfSN xy 2 1}
\mathbf{R}_{N,2}^{(1)}(x) &:= \sqrt{\frac{N^{\nu-1}(N-1)!}{\Gamma(N-1+\nu)}}
  \tau^{N-1} |x|^{\frac{\nu}{2}} K_{\nu/2}\Bigl(\frac{N|x|}{1-\tau^2} \Bigr)
e^{\frac{N\tau}{1-\tau^2}x} L_{N-1}^{(\nu)}\Bigl(\frac{N}{\tau}x\Bigr),
\\
\label{def of bfSN xy 2 2}
\mathbf{R}_{N,2}^{(2)}(x) &:=   2 \sqrt{\frac{N^{\nu-1}(N-1)!}{\Gamma(N-1+\nu)}}
\tau^{N-2} \int_{\R} \Big( \frac12-  \mathbbm{  1 }_{ (-\infty,x) } (t) \Big)
|t|^{\frac{\nu}{2}}  K_{\nu/2}\Bigl(\frac{N|t|}{1-\tau^2}\Bigr) e^{\frac{N\tau}{1-\tau^2}t} L_{N-2}^{(\nu)}\Bigl(\frac{N}{\tau}t\Bigr) \, dt. 
\end{align}
\end{lem}
\begin{proof}
This immediately follows from Theorem~\ref{thm relationship R and C}, together with definitions \eqref{def of RN 1pt real}, \eqref{def of bfSN xy 2 1} and \eqref{def of bfSN xy 2 2}. 
\end{proof}

\subsection{Differential equation for complex counterpart of kernel}
In this subsection, we formulate the limiting differential equations, which play an important role in the subsequent asymptotic analysis. To this end, we introduce a further change of the weight functions, which gives rise to limiting differential equations that are convenient to analyse.

Recall that $S_N^{ \rm c }$ is given by \eqref{def of SN} and $R_N^{ \rm c }(x)=S_N^{ \rm c }(x,x)$. We define   
\begin{equation}
\label{def of bold SN}
\mathfrak{R}_{N}^{\mathrm{c}}(x) := \sqrt{x^2+\tfrac{\varrho^2(1-\tau^2)^2}{4}} W_N^{\rm c}(x) R_{N}^{\mathrm{c}}(x),
\end{equation}
where 
\begin{align}
\begin{split}
\label{bfSN weight part}  
W_N^{\rm c}(x) &:=\sqrt{\frac{\pi(1-\tau^2)}{4N}} \Bigl(  \sqrt{x^2+\tfrac{\varrho^2(1-\tau^2)^2}{4}} + \tfrac{\varrho(1-\tau^2)}{2}
\Bigr)^{\varrho N} \Bigl(  x^2+\tfrac{\varrho^2(1-\tau^2)^2}{4} \Bigr)^{-\frac{1}{4}}
\\
&\quad \times \exp\Bigl( -\frac{2N}{1-\tau^2} \Bigl( \sqrt{x^2+\tfrac{\varrho^2(1-\tau^2)^2}{4}}- \tau x \Bigr) \Bigr).
\end{split}    
\end{align}

\begin{lem} \label{Lem_relation bfRN mathfrak RN}
The 1-point function $\bfR_N^{ \rm c }$ in \eqref{hat bfK N complex} is related to $\mathfrak{R}_{N}^{\mathrm{c}}$ in \eqref{def of bold SN} as  
\begin{equation}
\label{bfKN to boldSN}
    \bfR_N^{\rm c}(x) = \Big( x^2+\tfrac{\varrho^2(1-\tau^2)^2}{4}\Big)^{-\frac{1}{2}}
    \frac{|x|^{\nu}K_{\nu}\bigl(\frac{2N}{1-\tau^2}|x|\bigr)
    e^{\frac{2N\tau}{1-\tau^2}x}   }{W_N^{\rm c}(x)}
    \mathfrak{R}_{N}^{\rm c}(x).
\end{equation} 
Furthermore, for $\tau\in[0,1)$ which may depend on $N$, as $N \to \infty,$ we have 
\begin{equation} \label{approx of RN c and mathfrak RN c}
   \bfR_N^{\rm c}(x) = \Big( x^2+\tfrac{\varrho^2(1-\tau^2)^2}{4}\Big)^{-\frac{1}{2}}   \mathfrak{R}_{N}^{\rm c}(x) \Bigl(1+O(\frac{1}{N})\Bigr),
\end{equation}
uniformly for $0< |x| < \infty.$
\end{lem}
\begin{proof}
The first assertion immediately follows from the definitions \eqref{hat bfK N complex} and  \eqref{def of bold SN}. 

By \cite[Eq. (10.41.4)]{NIST}, for fixed $m\in\N$ and $r\in\Z$, we have   
\begin{align}
\begin{split}
\label{Bessel K function Large order asymptotics}
K_{\frac{m\nu}{2}+r}\Bigl(\frac{mN}{1-\tau^2}|t|\Bigr)
&=
\sqrt{\frac{\pi(1-\tau^2)}{2mN}}
\Bigl(|t|^2+\tfrac{\varrho^2(1-\tau^2)^2}{4}\Bigr)^{-\frac{1}{4}}
\exp\Bigl(  -\frac{mN}{1-\tau^2}\sqrt{|t|^2+\tfrac{\varrho^2(1-\tau^2)^2}{4}} \Bigr)
\\
&\quad\times
\Bigl(
\frac{\sqrt{|t|^2+\frac{\varrho^2(1-\tau^2)^2}{4}}+\frac{\varrho(1-\tau^2)}{2}}{|t|}
\Bigr)^{\frac{m\varrho N}{2}+r}
\Bigl(1+O(\frac{1}{N})\Bigr),
\end{split}
\end{align}
as $N\to \infty$, uniformly for $0<|t|<\infty$. 
On the other hand, for a fixed $\nu>-\frac{1}{2}$, we need the following asymptotic expansion of the modified Bessel functions that for fixed $\nu>-1$ and fixed $m,r\in\Z$, by \cite[Eq. (10.40.2)]{NIST}
\begin{align}
\begin{split}
\label{def of Bessel K fixed asym}
K_{\frac{m\nu}{2}+r}\Bigl( 
\frac{mN|x|}{1-\tau^2}
\Bigr)
&=
\sqrt{\frac{\pi(1-\tau^2)}{2mN|x|}}
e^{-\frac{mN|x|}{1-\tau^2}}
\Bigl[
1
+
\frac{(1-\tau^2)(4(\frac{m\nu}{2}+r)^2-1)}{8mN|z|}
+
O(\frac{1}{N^{2}})
\Bigr],
\end{split}
\end{align}
uniformly for $0<|x|<\infty$.
Using these asymptotic behaviours, it follows that 
\begin{equation} \label{def of WN approx potential}
|x|^{\nu}K_{\nu}\Bigl(\frac{2N |x| }{1-\tau^2} \Bigr)
    e^{\frac{2N\tau}{1-\tau^2}x}= W_N^{\rm c}(x) \Bigl(1+O(\frac{1}{N})\Bigr),
\end{equation}
as $N\to\infty$, uniformly for $0< |x| < \infty$. This completes the proof.  
\end{proof}

The motivation for introducing $\mathfrak{R}_{N}^{\mathrm{c}}$ in \eqref{def of bold SN} becomes clear from the following lemma.

\begin{lem} \label{Lem_ODE for mathfrak RN} 
As $N \to \infty$, we have
\begin{equation}
\label{boldSN ODE global diagonal}
\partial_x\mathfrak{R}_N^{\mathrm{c}}(x) = \mathfrak{I}_N^{\mathrm{c}}(x)
\Bigl( 1 + O(\frac{1}{N}) \Bigr),  
\end{equation}
where 
\begin{align}
\label{diagonal IN c natural}
 \mathfrak{I}_N^{\mathrm{c}}(x)&:= \mathcal{I}_{N,1}(x)\cdot\mathcal{I}_{N,2}(x),
\end{align}  
and 
\begin{align}
\begin{split}
\label{calIN1x}
\mathcal{I}_{N,1}(x)&:= \frac{ N^{\nu+2}}{1-\tau^2} \frac{(N-1)!\,\tau^{2N-3}}{\Gamma(N+\nu-1)} W_N^{\rm c}(x),
\end{split}
\\
\label{calIN2x}
\mathcal{I}_{N,2}(x) &:= L_{N-1}^{(\nu)}\Bigl(\frac{N}{\tau}x\Bigr) L_{N-2}^{(\nu)}\Bigl(\frac{N}{\tau}x\Bigr) + L_{N-2}^{(\nu+1)}\Bigl(\frac{N}{\tau}x\Bigr) L_{N-1}^{(\nu-1)}\Bigl(\frac{N}{\tau}x\Bigr). 
\end{align} 
Here, $W_N^{ \rm c }$ is given by \eqref{bfSN weight part}.
\end{lem}
\begin{proof}
We use the generalised Christoffel-Darboux formula \cite[Theorem 1.1]{BN24} for the diagonal case: for any $\tau \in [0,1]$, $\nu \in \R$ and $N \in \mathbb{N}$, we have 
\begin{align}
\begin{split}
\label{def of ODE diagonal}
&\quad
\Bigl[
(1-\tau^2)x\partial_x^2
+
\bigl(4\tau x+(2\nu+1)(1-\tau^2) \bigr)
\partial_x
-
2\bigl(2x-(2\nu+1)\tau \bigr)
\Bigr]
\cK_{N-1}(x,x)
\\
&=
\frac{2(N-1)!\tau^{2N-3}}{\Gamma(N+\nu-1)}
\Bigl[
L_{N-1}^{(\nu)}\Bigl(\frac{x}{\tau}\Bigr)
L_{N-2}^{(\nu)}\Bigl(\frac{x}{\tau}\Bigr)
+
L_{N-2}^{(\nu+1)}\Bigl(\frac{x}{\tau}\Bigr)
L_{N-1}^{(\nu-1)}\Bigl(\frac{x}{\tau}\Bigr)
\Bigr].
\end{split}
\end{align}
Then \eqref{boldSN ODE global diagonal} follows from our definitions \eqref{def of SN}, \eqref{def of bold SN}, and also by extracting the leading term in \eqref{boldSN ODE global diagonal}. Here, we note that the existence of the limit of \eqref{def of bold SN} follows from \cite[Lemma 3.5]{AB23}. %; see also \cite[Theorem 1.1 and Section 3.3]{AKM19}. 
\end{proof}

\subsection{Auxiliary asymptotics}

We end this section by recalling some elementary results. 
We first record the asymptotic behaviour of the weight functions \eqref{def of omega N 1}, \eqref{def of omega N 2}, and \eqref{def of omega N 3}, which follows from straightforward computations using \eqref{Bessel K function Large order asymptotics}.

\begin{lem} \label{Lem_asymp of omega 123}
We assume $\tau\in[0,1)$, which may depend on $N$, and $\nu=\varrho N$ with $\varrho>0$.   
Then as $N\to\infty$, we have 
\begin{align}
\label{Bessel Asymptotic xy 1}
\omega_{N,1}(x)   &= \sqrt{\frac{\pi(1-\tau^2)}{N}} \Big(|x|^2+\tfrac{\varrho^2(1-\tau^2)^2}{4}\Big)^{-\frac{1}{4}}
(1+o(1)),
\\
\label{Bessel Asymptotic xy 2}
 \omega_{N,2}(x) &= \sqrt{\frac{\pi(1-\tau^2)}{N}} \Big(|x|^2+\tfrac{\varrho^2(1-\tau^2)^2}{4}\Big)^{-\frac{1}{4}} \Big( \sqrt{|x|^2+\tfrac{\varrho^2(1-\tau^2)^2}{4}}+\tfrac{\varrho(1-\tau^2)}{2} \Big)  
(1+o(1)),
\\
\label{Bessel Asymptotic xy 3}
\omega_{N,3}(x)  &=  \frac{1}{|x|}  \Bigl(\sqrt{|x|^2+\tfrac{\varrho^2(1-\tau^2)^2}{4}}-\tfrac{\varrho(1-\tau^2)}{2} \Bigr)(1+o(1)),
\end{align}
uniformly for $0<|x|<\infty$. 
In particular, we have 
\begin{equation}
\label{weight first factor}
\frac{|x|}{2\pi} \Bigl(  \omega_{N,1}(x) \omega_{N,3}(x) + \omega_{N,2}(x) \Bigr) = \sqrt{\frac{1-\tau^2}{\pi N}} \Bigl(x^2+\frac{\varrho^2(1-\tau^2)^2}{4}\Bigr)^{\frac{1}{4}}
(1+o(1)).
\end{equation} 
%\todo[inline]{TODO: check what is $\tau$}
\end{lem}

For $\varrho\geq0$, we define 
\begin{align}
\begin{split}
\label{def Omega varrho}
\Omega_{\varrho}(x)
&:= \frac{2}{1-\tau^2}\Bigl( \sqrt{x^2+\tfrac{(1-\tau^2)^2\varrho^2}{4}}-\tau x  \Bigr) 
\\
&\quad -\varrho\log \Big( \sqrt{x^2+\tfrac{(1-\tau^2)^2\varrho^2}{4}} + \tfrac{(1-\tau^2)\varrho}{2} \Big)  -\varrho +(1+\varrho)\log(1+\varrho)
-2 g_{\tau}^{(\varrho)}(x),
\end{split}
\end{align} 
where $g_\tau^{ (\varrho) }$ is given by \eqref{g tau varrho}.
The following lemma can be established using the same strategy as in \cite[Lemma 4.2]{BN24}, where the case $\varrho = 0$ was considered.

\begin{lem}
\label{Lem Omega Positivity}
For $\varrho \geq 0$ and all $z \in \mathbb{C}$, it holds that $\Omega_{\varrho}(z) \geq 0$, with equality if and only if $z \in \partial S_{\varrho}$. Furthermore, it also holds for $\nu>-1$, with $\varrho=0$. 
\end{lem}

Recall that $\xi_\pm$ is given by \eqref{droplet right left edge}. For $p\in[\xi_-,\xi_+]$, we write 
\begin{equation} \label{def of local density}
\delta\equiv\delta(p):= \frac{1}{2(1-\tau^2)} \Big(p^2+\frac{\varrho^2(1-\tau^2)^2}{4}\Big)^{-\frac{1}{2}}  
\end{equation} 
for the density of the equilibrium measure \eqref{equilibrium msr}. 
By performing straightforward computations similar to those in \cite[Section 4]{BN24}, we have the following. 

\begin{lem}\label{lem edge f varrho}
Assume that $\tau\in[0,1)$ is fixed. For $\varrho\geq0$ and for a fixed $p=\xi_{\pm}$ given by \eqref{droplet right left edge} and $\zeta\in\R$, let 
\begin{equation}
\label{edge local coordinate}
x\equiv x(\zeta):=p+\mathbf{n}(p)\frac{\zeta}{\sqrt{N\delta}}, 
\end{equation}
where $\mathbf{n}(p)$ is the outward unit normal vector at $p$. 
Then as $N\to\infty$, we have 
\begin{equation}
 \label{fN varrho asymptotic}
 \Omega_{\varrho}(x)
 =
 \frac{2}{N}\zeta^2
 +
 O(\frac{1}{N^{\frac{3}{2}}}).
\end{equation}  
\end{lem}

\section{Proofs of Theorems~\ref{Thm_expected number} and~\ref{Thm_global density} at strong non-Hermiticity}   \label{Section_strong}

This section is devoted to the proofs of Theorems~\ref{Thm_expected number} and~\ref{Thm_global density} at strong non-Hermiticity.

As previously mentioned, when presenting results on macroscopic behaviour, it is not necessary to distinguish between the large-parameter case \(\nu = O(N)\) and the fixed-parameter case \(\nu\). However, in the proofs, this distinction becomes necessary. While the overall ideas and arguments remain the same with slight modifications, the fixed-parameter case sometimes requires different asymptotic behaviour of the Laguerre polynomials.
For clarity of presentation, in this subsection, we focus on the large-parameter case and provide the necessary modifications in Appendix~\ref{Appendix_fixed}.

\subsection{Outline of the proof}
\label{subsection outline at strong}
We first derive the macroscopic behaviour of the one-point density $\bfR_N$ given by \eqref{def of RN 1pt real}. 
Given $\epsilon>0$ sufficiently small (which should be taken by $0<\epsilon<\frac{1}{6}$), we write 
\begin{equation}
\label{subset cEN}
 \mathcal{E}_N:=( \xi_- + N^{-\frac{1}{2}+\epsilon}, \xi_+ - N^{-\frac{1}{2}+\epsilon} ), 
\end{equation}
where $\xi_\pm$ are given by \eqref{droplet right left edge}. 
Then we show the following. 

\begin{prop}[\textbf{Macroscopic behaviour of the one-point density}] \label{Prop_asymptotic of 1pt function RN strong}
Fix $\tau\in(0,1)$ and $\nu=\varrho\,N>0$ with $\varrho>0$.   
Then as $N\to\infty$, we have  
\begin{equation}
    \label{bfSN 1 xx}
     \mathbf{R}_{N}(x)  =  \frac{1}{2} \sqrt{\frac{N}{\pi(1-\tau^2)}} \Bigl(x^2+\tfrac{\varrho^2(1-\tau^2)^2}{4}\Bigr)^{-\frac{1}{4}} ( 1 + o(1) ), 
\end{equation}
uniformly for $x\in \mathcal{E}_{N}$. 
\end{prop} 

This proposition will be shown in Subsection~\ref{Subsection_Diagonal estimates for the complex kernel part at strong}.

\medskip 

To complete the proofs of Theorems~\ref{Thm_expected number} and~\ref{Thm_global density} in the regime of strong non-Hermiticity, we need to show that the integral over the complement of \(\mathcal{E}_{N}\) does not contribute to the leading term.  
Thus, we establish an a priori \(L^1\)-estimate.  
To this end, for the set of real eigenvalues of \eqref{Asymmetric Wishart matrix}, a test function \( f \), and a measurable set \( A \subset \mathbb{R} \), we define the linear statistics 
\begin{equation}
\label{linear statistics f}
    \Xi_{N}^{A}(f) := \sum_{j=1}^{\mathcal{N}}f(x_j)\mathbbm{1}_{ \{ x_j\in A \} },\qquad
    \Xi_{N}(f):=\Xi_{N}^{\R}(f). 
\end{equation}
We divide the linear statistics \eqref{linear statistics f} into 
\begin{equation}
\label{linear statistic}
    \Xi_{N}(f) = \Xi_{N}^{\mathcal{E}_N}(f) + \Xi_{N}^{\mathcal{E}_N^{\rm c}}(f),
\end{equation} 
where $A^{\rm c}$ denotes the complement of a set $A$.

\begin{lem}[\textbf{$L^1$-estimate for linear statistics}]
\label{lem L1 estimate}
Fix $\tau\in(0,1)$ and $\nu=\varrho\,N>0$ with $\varrho>0$.
Let $f$ be a locally integrable and measurable function satisfying
\begin{equation}
    \label{f integrable bound}
    \chi(f):=\sup_{x\in\R}\Bigl( |f(x)| e^{-c|x|} \Bigr)<\infty,
\end{equation}
for some small $c>0$. 
Then for $0<\epsilon<\frac{1}{2}$ sufficiently small, the complementary statistic $\Xi_{N}^{\mathcal{E}_N^{\rm c}}(f)$ satisfies 
\begin{equation}
\label{expectation global estimate}
    \mathbb{E}\bigl[|\Xi_{N}^{\mathcal{E}_N^{\rm c}}(f)|\bigr]
    =O(N^{\epsilon}). 
\end{equation}
\end{lem}

This lemma will be shown in Subsection~\ref{Subsection apriori L1 estimate at strong}.

By combining Proposition~\ref{Prop_asymptotic of 1pt function RN strong} with Lemma~\ref{lem L1 estimate}, we complete the proofs of Theorems~\ref{Thm_expected number} and~\ref{Thm_global density} for $\nu=\varrho\,N$ in the regime of strong non-Hermiticity.

\begin{proof}[Proof of Theorems~\ref{Thm_expected number} (i) and ~\ref{Thm_global density} (i)]

Recall that the expected number of real eigenvalues $E_{N,\tau}$ can be expressed in terms of the 1-point function; see \eqref{EN tau bfRN}.
By Proposition~\ref{Prop_asymptotic of 1pt function RN strong} and Lemma~\ref{lem L1 estimate}, we have 
\[
E_{N,\tau}^{\nu}
=
\int_{\R}\mathbf{R}_{N}(x)\,dx
=
\int_{\mathcal{E}_{N}}\mathbf{R}_{N}(x) \,dx
\, \bigl( 1+o(1)\bigr)
=
\frac{1}{2}
\sqrt{\frac{N}{\pi(1-\tau^2)}}
c(\tau,\varrho)
( 1 + o(1) ),
\]
where $c(\tau,\varrho)$ is given by \eqref{c tau varrho}. This shows Theorem~\ref{Thm_expected number} (i).

Recall that \(\rho_{N,\tau}^\nu\) is given by \eqref{rho N tau bfRN}. Then \eqref{rho N strong test limit} follows from Lemma~\ref{lem L1 estimate}, while the pointwise limit follows immediately from Proposition~\ref{Prop_asymptotic of 1pt function RN strong} and Theorem~\ref{Thm_expected number} (i). Hence, the proof of Theorem~\ref{Thm_global density} (i) is complete.
\end{proof}

\subsection{Macroscopic behaviour of the one-point density; Proof of Proposition~\ref{Prop_asymptotic of 1pt function RN strong}}
\label{Subsection_Diagonal estimates for the complex kernel part at strong}

In this subsection, we prove Proposition~\ref{Prop_asymptotic of 1pt function RN strong}. 

By Lemma~\ref{Lem_RN RN hat RN2 decomp}, it suffices to determine the asymptotic behaviour of \(\widehat{\bfR}_{N,1}\) and \(\bfR_{N,2}\), given in \eqref{bfRN one point density} and \eqref{def of bfRN 2 xy}, respectively. In particular, we verify that the leading-order contribution to \(\bfR_N\) comes from \(\widehat{\bfR}_{N,1}\), while \(\bfR_{N,2}\) decays exponentially. These results are formulated in the following two lemmas.

\begin{lem}\label{Lem_bold SN dens}
Fix $\tau\in(0,1)$ and $\nu=\varrho\,N$ with $\varrho>0$.   
Then as $N\to\infty$, we have 
\begin{equation}
\label{def of complex kernel part diagonal}
    \mathbf{R}_N^{\rm c}(x) = \frac{N}{2(1-\tau^2)} \Bigl( x^2+\tfrac{\varrho^2(1-\tau^2)^2}{4} \Bigr)^{-\frac{1}{2}}  \Bigl( 1+ O(\frac{1}{\sqrt{N}})\Bigr), 
\end{equation}
uniformly for $x\in \mathcal{E}_N$. Consequently, by \eqref{bfRN one point density}, we have 
\begin{equation}
    \widehat{\mathbf{R}}_{N,1}(x) = \frac{1}{2} \sqrt{\frac{N}{\pi(1-\tau^2)}}
\Bigl(x^2+\tfrac{\varrho^2(1-\tau^2)^2}{4}\Bigr)^{-\frac{1}{4}} ( 1 + o(1) ),
\end{equation}
uniformly for $x\in \mathcal{E}_{N}$. 
\end{lem}

\begin{lem} \label{Lem_asymptotic of RN2}
Fix $\tau\in(0,1)$ and $\nu=\varrho\,N$ with $\varrho>0$. 
Then for $0<\epsilon<\frac{1}{2}$ sufficiently small, as $N\to\infty$, we have 
\begin{equation}
\label{SN2 global estimate in the strongly non-Hermitian}
    \mathbf{R}_{N,2}(x)=O(e^{-cN^{2\epsilon}}),
\end{equation}
uniformly for $x\in\mathcal{E}_N$ and for some $c>0$. 
\end{lem}

\begin{figure}[t]
	\begin{subfigure}{0.45\textwidth}
	\begin{tikzpicture} 
    \pgfmathsetmacro{\tau}{0.35}
    \pgfmathsetmacro{\varrho}{3}
    
    \pgfmathsetmacro{\cx}{\tau*(2+\varrho)}
    \pgfmathsetmacro{\cy}{0}
    
    \pgfmathsetmacro{\a}{(1+\tau^2) * sqrt(1+\varrho)}
    \pgfmathsetmacro{\b}{(1-\tau^2) * sqrt(1+\varrho)}

    \pgfmathsetmacro{\xiMinus}{\cx - \a}
    \pgfmathsetmacro{\xiPlus}{\cx + \a}

    \pgfmathsetmacro{\fMinus}{\tau * ((sqrt(1+\varrho) - 1)^2)}
    \pgfmathsetmacro{\fPlus}{\tau * ((sqrt(1+\varrho) + 1)^2)}

    \draw (-1,0) -- (\fMinus,0);
    \draw[->] (\fPlus,0) -- (6.5,0);
    \draw[->] (0,-2.25) -- (0,2.25);

    \draw[thick, blue] (\cx,\cy) ellipse ({\a} and {\b});

    \fill[black] (\xiMinus,0) circle (2pt) node[below left, xshift=4pt, yshift=1pt] {\(\xi_-\)};
    \fill[black] (\xiPlus,0) circle (2pt) node[below right] {\(\xi_+\)};

     \fill[red] (\fMinus,0) circle (2pt) node[above right] {\(F_-\)};
    \fill[red] (\fPlus,0) circle (2pt) node[above left] {\(F_+\)};

    \draw[thick, red] (\fMinus,0) -- (\fPlus,0);
\end{tikzpicture}
	\subcaption{$\tau=0.35$}
\end{subfigure}
\quad 
	\begin{subfigure}{0.45\textwidth}
	\begin{tikzpicture} 
    \pgfmathsetmacro{\tau}{0.65}
    \pgfmathsetmacro{\varrho}{3}
    
    \pgfmathsetmacro{\cx}{\tau*(2+\varrho)}
    \pgfmathsetmacro{\cy}{0}
    
    \pgfmathsetmacro{\a}{(1+\tau^2) * sqrt(1+\varrho)}
    \pgfmathsetmacro{\b}{(1-\tau^2) * sqrt(1+\varrho)}

    \pgfmathsetmacro{\xiMinus}{\cx - \a}
    \pgfmathsetmacro{\xiPlus}{\cx + \a}

    \pgfmathsetmacro{\fMinus}{\tau * ((sqrt(1+\varrho) - 1)^2)}
    \pgfmathsetmacro{\fPlus}{\tau * ((sqrt(1+\varrho) + 1)^2)}

    \draw (-1,0) -- (\fMinus,0);
    \draw[->] (\fPlus,0) -- (6.5,0);
    \draw[->] (0,-2.25) -- (0,2.25);

    \draw[thick, blue] (\cx,\cy) ellipse ({\a} and {\b});

    \fill[black] (\xiMinus,0) circle (2pt) node[below left, xshift=4pt, yshift=1pt] {\(\xi_-\)};
    \fill[black] (\xiPlus,0) circle (2pt) node[below right] {\(\xi_+\)};

     \fill[red] (\fMinus,0) circle (2pt) node[above right] {\(F_-\)};
    \fill[red] (\fPlus,0) circle (2pt) node[above left] {\(F_+\)};
    \draw[thick, red] (\fMinus,0) -- (\fPlus,0);
\end{tikzpicture}
	\subcaption{$\tau=0.65$ }
\end{subfigure}	
	\caption{The plots the boundary of the droplet $S_\varrho$ in \eqref{droplet}, along with the right and left endpoints $\xi_\pm$ and the foci $F_\pm$. Here, $\varrho=3$ thereby $\tau_{\varrho}=0.5$. } \label{Fig_droplets with foci}  
\end{figure}
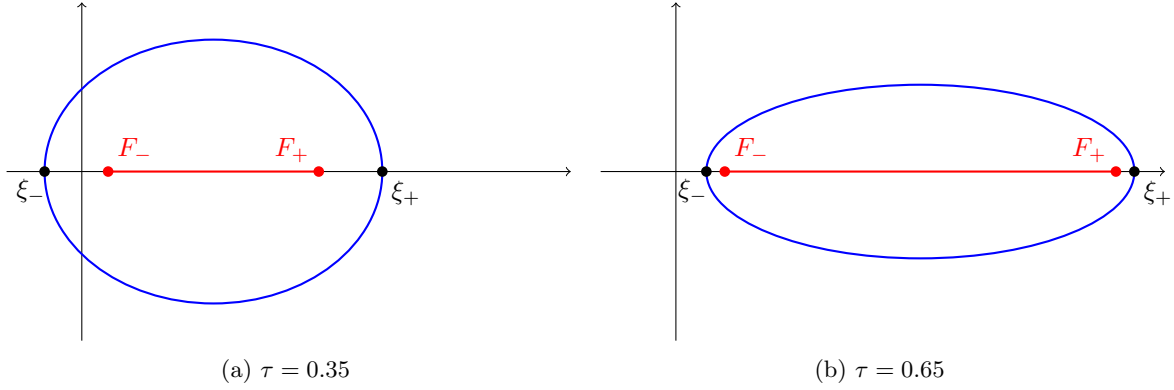 

\medskip

In the asymptotic analysis of non-Hermitian Wishart ensembles with the large parameter $\nu=\varrho\,N$ for $\varrho>0$, one often distinguishes between the cases where the droplet contains the origin and where it does not. For a given \(\varrho > 0\), this transition occurs at the critical value
\begin{equation} \label{def of tau varrho cri}
\tau_{\varrho}:=\frac{1}{\sqrt{1+\varrho}}. 
\end{equation}
Then, if \(\tau < \tau_\varrho\), the droplet contains the origin, whereas if \(\tau > \tau_\varrho\), it does not. See Figure~\ref{Fig_droplets with foci}.   

%\begin{proof}
%Then the lemma follows from straightforward computations using \eqref{Bessel K function Large order asymptotics}. %, \eqref{def of omega N 1}, \eqref{def of omega N 2}, and \eqref{def of omega N 3}.  
%\end{proof} 

%\subsubsection{Proof of Lemma~\ref{Lem_bold SN dens}} 

We now prove Lemma~\ref{Lem_bold SN dens}. 
\begin{proof}[Proof of Lemma~\ref{Lem_bold SN dens}]
By \eqref{approx of RN c and mathfrak RN c}, it suffices to show that 
\begin{equation} \label{asymp of mathfrak RN c}
\mathfrak{R}_N^{\,\rm c}(x) = \frac{N}{2(1-\tau^2)}   \Bigl( 1+ O(\frac{1}{\sqrt{N}})\Bigr), 
\end{equation}
uniformly for $x\in \mathcal{E}_N$. We shall make use of Lemma~\ref{Lem_ODE for mathfrak RN} and 
\begin{equation}  \label{decomp of mathfrak RN c}
    \mathfrak{R}_N^{\rm c}(x) = \mathfrak{R}_N^{\rm c}(0) + \int_0^x\partial_t\mathfrak{R}_N^{\rm c}(t)\, dt. 
\end{equation}

We first derive the asymptotic behaviour of \(\mathfrak{R}_N^{\rm c}(0)\).
By combining \eqref{bfRN one point density}, \eqref{omega c weight at 0} with 
$$L_n^{(\nu)}(0)=\frac{\Gamma(n+\nu+1)}{ n! \Gamma(\nu+1)}, $$ 
we have 
\begin{align*}
\bfR_N^{ \rm c }(0)  =  \frac{ N^2 }{ \nu } (1-\tau^2)^{ \nu-1 }  \sum_{j=0}^{N-2} \tau^{2j}   \frac{ \Gamma(j+\nu+1)  }{ j!\, \Gamma(\nu+1)  }.
\end{align*}
Recall that the regularised incomplete beta function $I_x$ is given by 
\begin{equation}
I_x(a,b):= \frac{B_x(a,b)}{B_1(a,b)}, \qquad   B_x(a,b):=\int_0^x t^{a-1}(1-t)^{b-1}\,dt; 
\end{equation}
see \cite[Chapter 8]{NIST}. Then by using \cite[Eq. (8.17.5)]{NIST}, we have 
\begin{equation}
   \bfR_N^{\rm c}(0)=\frac{N}{\varrho(1-\tau^2)^2}\Bigl(1-I_{\tau^2}(N-1,\nu+1)\Bigr).
\end{equation}
This expression allows us to use the well known uniform asymptotics of the incomplete beta function; see e.g. \cite[Lemma 2.2]{BP24}. 
Consequently, it follows that 
\begin{equation}
 \bfR_N^{\rm c}(0) = \frac{N}{\varrho(1-\tau^2)^2} 
 \times
 \begin{cases}
1+O(e^{-\epsilon N}),     & \text{if $0<\tau<\tau_{\varrho}$}, 
\smallskip 
\\
\frac{1}{2}+O(N^{-\frac{1}{2}}),    & \text{if $\tau=\tau_{\varrho}$}, 
\smallskip 
\\
O(e^{-\epsilon N}),    & \text{if $\tau_{\varrho}<\tau<1$},
\end{cases}
\end{equation}
where $\epsilon>0$ may change in each line, and each $O(\cdot)$-term is uniform. Notice that this gives rise to \eqref{def of complex kernel part diagonal} for $x=0$.
Note that by \eqref{bfKN to boldSN}, we also have 
\begin{equation}
\label{def of boldSN 00}
\mathfrak{R}_N^{\rm c}(0) = \frac{N}{2(1-\tau^2)} 
 \times
  \begin{cases}
1+O(e^{-\epsilon N}),     & \text{if $0<\tau<\tau_{\varrho}$}, 
\smallskip 
\\
\frac{1}{2}+O(N^{-\frac{1}{2}}),    & \text{if $\tau=\tau_{\varrho}$}, 
\smallskip 
\\
O(e^{-\epsilon N}),    & \text{if $\tau_{\varrho}<\tau<1$}.
\end{cases}
\end{equation}

By \eqref{asymp of mathfrak RN c} and \eqref{decomp of mathfrak RN c}, it suffices to estimate the second term of \eqref{decomp of mathfrak RN c}. 
We make use of Lemma~\ref{Lem_ODE for mathfrak RN}, together with the asymptotic behaviours of $\mathcal{I}_{N,1}$ and $\mathcal{I}_{N,2}$. 
Note that by Stirling formula, we have 
\begin{equation}
\label{IN stirling part}
\frac{N^{\nu+2}\, (N-1)!}{\Gamma(N+\nu-1)} = (1+\varrho)^{\frac{3}{2}}N^3\exp\Bigl(N\bigl(\varrho-(1+\varrho)\log(1+\varrho)\bigr) \Bigr)
\Bigl( 1 + O(\frac{1}{N}) \Bigr).
\end{equation}
The asymptotic behaviour of $\partial_t \mathfrak{R}_N^{\rm c}(t)$ depends on whether $0$ is contained in the droplet. Therefore we analyse these cases separately.

\medskip  

We first consider the case $0< \tau <\tau_{ \varrho }$. 
For $t\in[0,x]$ and for a sufficiently large $N$, it follows from \eqref{diagonal IN c natural}, \eqref{IN stirling part}, and Lemma~\ref{Lem BDLaguerre large order} that there exists a positive constant $C>0$ such that 
\begin{equation}
\label{estimate 1}
\bigl|\mathfrak{I}_N^{\mathrm{c}}(t)\bigr|
\leq
C N^{\frac{3}{2}} e^{-N\Omega_{\varrho}(t)},
\end{equation}
where $\Omega_{\varrho}(t)$ is given by \eqref{def Omega varrho}.
Since $\Omega_{\varrho}$ is strictly positive over a compact subset of $\mathcal{E}_N$ by Lemma~\ref{Lem Omega Positivity}, the right hand side of this inequality is uniformly bounded by $O(e^{-\delta N})$ with $\min_{t\in [0,x]\subset \mathcal{E}_N}\Omega_{\varrho}(t)\geq\delta>0$. 
Then by \eqref{boldSN ODE global diagonal} and \eqref{estimate 1}, we obtain the exponentially decay for the second term of \eqref{decomp of mathfrak RN c}.

\medskip 

Next, we consider the case $\tau=\tau_{\varrho}$, the case where the origin is on the boundary of the droplet. 
If $x\in \mathcal{E}_N$, then by \eqref{def of boldSN 00}, we have 
\begin{equation}
\mathfrak{R}_N^{\rm c}(x) = \frac{N}{4(1-\tau^2)} \Bigl( 1+O(\frac{1}{\sqrt{N}})\Bigr) + \int_{0}^{x} \partial_t \mathfrak{R}_{N}^{\mathrm{c}}(t)dt.
\end{equation}
For a sufficiently small $0<\epsilon<\frac{1}{2}$, 
we split the interval $[0,x]$ into intervals $[0,\frac{N^{\epsilon}}{\sqrt{N\delta(0)}}]\cup[\frac{N^{\epsilon}}{\sqrt{N\delta(0)}},x]$,
where $\delta\equiv\delta(p)$ is given by \eqref{def of local density}.
By Lemma~\ref{lem Exponential regime}~(i), \eqref{diagonal IN c natural}, 
and \eqref{IN stirling part}, we have 
\begin{equation}
\label{def of bold IN estimate 2}
\mathfrak{I}_N^{\mathrm{c}}(t) = - \frac{(1+\varrho)^2N\sqrt{N\delta(t)}}{\sqrt{2\pi}}
e^{-N\Omega_{\varrho}(t)}\frac{\psi_{\varrho}'(t)}{\psi_{\varrho}(t)^3}
\Bigl( 1+O(\frac{1}{\sqrt{N}})\Bigr). 
\end{equation}
By change of variable, Taylor expansion, Lemma~\ref{lem edge f varrho}, \eqref{boldSN ODE global diagonal}, and the saddle point approximation, we have 
\[
\int_{0}^{\frac{N^{\epsilon}}{\sqrt{N\delta(0)}}}\partial_t
\mathfrak{R}_{N}^{\mathrm{c}}(t) \, dt = \frac{N}{\sqrt{2\pi}} \frac{1}{(1-\tau^2)} \int_{-\infty}^{0} e^{-2t^2}  \, dt\cdot\Bigl( 1+O(\frac{1}{\sqrt{N}})\Bigr) =\frac{N}{4(1-\tau^2)} \Bigl( 1+O(\frac{1}{\sqrt{N}})\Bigr).
\]
On the other hand, on $[\frac{N^{\epsilon}}{\sqrt{N\delta(0)}},x]$, by Lemma~\ref{Lem Omega Positivity} and Taylor expansion, there exists a positive constant $c>0$ such that $N\Omega_{\varrho}(x)\geq c N^{2\epsilon}$. 
Thus we obtain the estimate for the second term of \eqref{decomp of mathfrak RN c}.
%Hence, by \eqref{estimate 1}, we obtain \eqref{asymp of mathfrak RN integral}. 

\medskip 

Finally, we consider the case $\tau_{\varrho}<\tau<1$.  
For $x\in \mathcal{E}_N$ and for a sufficiently small $\epsilon>0$, let 
\begin{align*}
\mathcal{E}_{N,\mathrm{L}}^{(-)} := \Bigl[0 , \xi_- - \tfrac{N^{\epsilon}}{\sqrt{N\delta(\xi_-)}} \Bigr), \qquad \mathcal{E}_{N,\mathrm{R}}^{(-)}  := \Bigl(\xi_- + \tfrac{N^{\epsilon}}{\sqrt{N\delta(\xi_-)}} , x \Bigr],
\end{align*}
and 
\begin{align} \label{def of EN + -}
\mathcal{E}_{N}^{(\pm)} := \Bigl[\xi_\pm - \tfrac{N^{\epsilon}}{\sqrt{N\delta(\xi_\pm)}},\xi_\pm + \tfrac{N^{\epsilon}}{\sqrt{N\delta(\xi_\pm)}} \Bigr].
\end{align}  
We split the interval $[0,x]$ into 
$$
[0,x]=\mathcal{E}_{N,\mathrm{L}}^{(-)} \cup \mathcal{E}_{N}^{(-)} \cup \mathcal{E}_{N,\mathrm{R}}^{(-)}. 
$$ 
On $\mathcal{E}_{N,\mathrm{L}}^{(-)}$ and $\mathcal{E}_{N,\mathrm{R}}^{(-)}$, by Lemma~\ref{Lem Omega Positivity} and Taylor expansion, there exists a positive constant $c>0$ such that $N\Omega_{\varrho}(x)\geq c N^{2\epsilon}$.
On the other hand, by \eqref{boldSN ODE global diagonal}, \eqref{def of bold IN estimate 2}, Taylor expansion, Lemma~\ref{lem edge f varrho}, and the saddle point approximation, we have 
\[
\int_{\mathcal{E}_{N}^{(-)}}\partial_t
\mathfrak{R}_{N}^{\mathrm{c}}(t) \, dt = \frac{N}{\sqrt{2\pi}} \frac{1}{(1-\tau^2)} \int_{-\infty}^{\infty} e^{-2t^2}\, dt\cdot \Bigl( 1+O(\frac{1}{\sqrt{N}})\Bigr) 
=\frac{N}{2(1-\tau^2)} \Bigl( 1+O(\frac{1}{\sqrt{N}})\Bigr).
\]
Therefore, we obtain the desired result uniformly for $x\in \mathcal{E}_N$. 
%The case $x\notin U_{\varrho}$ can be shown similar to \eqref{def of bold sn outside}. 
Finally, by combining the above all estimates with \eqref{bfKN to boldSN}, we obtain \eqref{def of complex kernel part diagonal}.
This completes the proof. 
\end{proof}
 
Next, we prove Lemma~\ref{Lem_asymptotic of RN2}. 

\begin{proof}[Proof of Lemma~\ref{Lem_asymptotic of RN2}]

Recall that $\bfR_{N,2}$ is given by \eqref{def of bfRN 2 xy}, where $\bfR_{N,2}^{(1)}$ and  $\bfR_{N,2}^{(2)}$ are given by \eqref{def of bfSN xy 2 1} and \eqref{def of bfSN xy 2 2}, respectively.

By Lemma~\ref{Lem Omega Positivity}, there exists a positive constant $c>0$ such that  $N\Omega_{\varrho}(x)\geq cN^{2\epsilon}$ for $x\in \mathcal{E}_N$.
Then by \eqref{def of bfSN xy 2 1}, there exists $C>0$ such that 
\begin{equation}
\label{RN2 front term}
\mathbf{R}_{N,2}^{(1)}(x) \leq  Ce^{-cN^{2\epsilon}},
\end{equation}
uniformly for $x\in\mathcal{E}_N$. 
Next, we show that there exists a positive constant $C>0$ such that 
\begin{equation} 
\label{Estimate all integral part 1}
\mathbf{R}_{N,2}^{(2)}(x)  \leq  \frac{C}{N\sqrt{N}} \Bigl(  1+O(\frac{1}{N^{\frac{1}{2}-\epsilon}}) \Bigr).
\end{equation}
Combining \eqref{RN2 front term} and \eqref{Estimate all integral part 1}, the lemma follows.  

We write 
\begin{equation} \label{def of w tilde N r}
\widetilde{w}_{N,r}(t) := \sqrt{\frac{N^{\nu-1}(N-1)!}{\Gamma(N-1+\nu)}} \tau^{N-r} |t|^{\frac{\nu}{2}} K_{\nu/2}\Bigl(\frac{N|t|}{1-\tau^2}\Bigr) e^{\frac{N\tau}{1-\tau^2}t}.
\end{equation}
For a measurable set $A$ in $\R$, we define 
\begin{equation}
\label{def of JAr}
J_{r,A} \equiv J_r[A]:= \int_{A} \widetilde{w}_{N,r}(t) L_{N-r}^{(\nu)}\Bigl(\frac{N}{\tau}t\Bigr) \, dt. 
\end{equation}  
Then by \eqref{def of bfSN xy 2 2} and \eqref{def of JAr}, one can write 
\begin{equation} \label{RN2 in terms of J}
\mathbf{R}_{N,2}^{(2)}(x) = J_{2}[\R]-2J_{2}[(-\infty,x)].
\end{equation}
Therefore, in order to show \eqref{Estimate all integral part 1}, it suffices to show that for a sufficiently small $0<\epsilon<\frac{1}{2}$, 
\begin{align}
\begin{split} \label{estimate of all domain integral}
J_{  2 }[\R] & = \frac{\sqrt{2\pi(1-\tau^2)}}{N\sqrt{N}} \frac{(-1)^{N-2}}{(1+\varrho)^{\frac{1}{4}}} \Bigl(  1+O(\frac{1}{N^{\frac{1}{2}-\epsilon}}) \Bigr).
\end{split}
\end{align}

Recall that $\mathcal{E}_N^{ (+) }$ is given by \eqref{def of EN + -}. 
By Lemmas~\ref{lem Exponential regime} and ~\ref{lem edge f varrho}, and the saddle point approximation, we have 
\begin{align*}
%\int_{\mathcal{E}_{N}^{(+)} } \widetilde{w}_{N,2}(t) L_{N-2}^{(\nu)}\Bigl(\frac{N}{\tau}t\Bigr)\,dt 
J_{  2 }[ \mathcal{E}_{N}^{(+)} ] &= \frac{(-1)^{N-2}}{N} \sqrt{\frac{\pi(1-\tau^2)}{4\pi }} \int_{\mathcal{E}_{N}^{(+)} } \frac{\psi_{\varrho}(t)^{-2}\sqrt{\psi_{\varrho}'(t)}}{(t^2+\frac{(1-\tau^2)^2\varrho^2}{4})^{\frac{1}{4}}} e^{-\frac{N}{2}\Omega_{\varrho}(t)}
\,dt\cdot\Bigl( 1 + O( \frac{1}{N}) \Bigr)
\\
&= \frac{\sqrt{2\pi(1-\tau^2)}}{2N \sqrt{N}}  \frac{(-1)^{N-2}}{(1+\varrho)^{\frac{1}{4}}} \Bigl(  1+O(\frac{1}{N^{\frac{1}{2}-\epsilon}}) \Bigr).
\end{align*}
%Let 
%\begin{align} \label{def of EN -}
%\mathcal{E}_{N}^{(-)} := \Big[ \xi_- - \tfrac{N^{\epsilon}}{\sqrt{N\delta(\xi_-)}} , \xi_- + \tfrac{N^{\epsilon}}{\sqrt{N\delta(\xi_-)}} \Big]. 
%\end{align} 
Similarly, we have 
\[
%\int_{\mathcal{E}_{N}^{(-)} } \widetilde{w}_{N,2}(t) L_{N-2}^{(\nu)}\Bigl(\frac{N}{\tau}t\Bigr)dt
J_{  2 }[\mathcal{E}_{N}^{(-)}] = \frac{\sqrt{2\pi(1-\tau^2)}}{2N \sqrt{N}} \frac{(-1)^{N-2}}{(1+\varrho)^{\frac{1}{4}}} \Bigl(  1+O(\frac{1}{N^{\frac{1}{2}-\epsilon}}) \Bigr).
\]
On the other hand, 
for $x\in(\mathcal{E}_{N}^{(-)}\cup\mathcal{E}_{N}^{(+)})^{\mathrm{c}}$, by Lemma~\ref{Lem Omega Positivity} and Taylor expansion, there exists a positive constant $c>0$ such that $\frac{N}{2}\Omega_{\varrho}(x)\geq c N^{2\epsilon}$.
Therefore, we have 
\[
%\int_{(\mathcal{E}_{N}^{(-)}\cup\mathcal{E}_{N}^{(+)})^{\mathrm{c}}} \widetilde{w}_{N,2}(t) L_{N-2}^{(\nu)}\Bigl(\frac{N}{\tau}t\Bigr)dt
J_{ 2 }[(\mathcal{E}_{N}^{(-)}\cup\mathcal{E}_{N}^{(+)})^{\mathrm{c}}] = O(e^{-c N^{2\epsilon}}). 
\]
Therefore, we have shown \eqref{estimate of all domain integral}. 
\end{proof}

\subsection{\texorpdfstring{$L^1$}{L1}-estimate for linear statistics; Proof of Lemma~\ref{lem L1 estimate}}
\label{Subsection apriori L1 estimate at strong}

This subsection is devoted to the proof of Lemma~\ref{lem L1 estimate}.  
We first show the following. 
 
\begin{lem}\label{lem complex part edge kernel}
Fix $\tau\in(0,1)$ and $\nu=\varrho\,N>0$ with $\varrho>0$. 
Then as $N\to\infty$, we have 
\begin{equation} \label{eq for edge scaling lim}
\frac{1}{N\delta}\mathbf{R}_{N}^{\mathrm{c}}\Big( \xi_\pm \pm \frac{\zeta}{ \sqrt{N \delta(\xi_{\pm})} } \Big) = \frac{1}{2}\erfc(\sqrt{2}\zeta)
\Bigl( 1 + O(\frac{1}{\sqrt{N}}) \Bigr) 
\end{equation}
and 
\begin{equation} \label{edge scaling for RN hat}
\widehat{\mathbf{R}}_{N,1}\Big( \xi_\pm \pm \frac{\zeta}{ \sqrt{N \delta(\xi_{\pm})} } \Big) = \pm \frac{\sqrt{N\delta}}{2\sqrt{2\pi}}\erfc(\sqrt{2}\zeta)
\Bigl(1+O(\frac{1}{\sqrt{N}})\Bigr),
\end{equation}
uniformly for $\zeta$ in a compact subset of $\R$. 
\end{lem}

\begin{proof} 
Let $x\equiv x(\zeta)$ be given by \eqref{edge local coordinate} for $\zeta\in\R$. Recall that  $\mathfrak{I}_N^{\mathrm{c}}$ is given by \eqref{diagonal IN c natural}.  
By Lemmas~\ref{lem Exponential regime} and ~\ref{lem edge f varrho}, we have 
\begin{align}
    \begin{split}
\mathfrak{I}_N^{\rm c}(x)
&=
-\frac{N}{2}
\sqrt{\frac{2(1+\varrho)}{\pi}}
\sqrt{\frac{N}{2(1-\tau^2)(x^2+\frac{\varrho^2(1-\tau^2)^2}{4})^{\frac{1}{2}}}}\frac{\psi_{\varrho}'(x)}{\psi_{\varrho}(x)^3}e^{-N\Omega_{\varrho}(x)}\Bigl(1+O(\frac{1}{N})\Bigr)
\\
&=
-\frac{N\sqrt{N\delta}}{2(1-\tau^2)}\sqrt{\frac{2}{\pi}}
e^{-2\zeta^2}\Bigl(1+O(\frac{1}{\sqrt{N}})\Bigr).
    \end{split}
\end{align}
Therefore, by the initial condition $\mathfrak{R}_{N}^{\rm c}(x)\to 0$ as $x\to\infty$ and the change of variable $x=x(\zeta)$, we have 
\begin{align}
    \begin{split}
 \mathfrak{R}_{N}^{\rm c}(x)
 &=
 -\frac{N}{2(1-\tau^2)}\sqrt{\frac{2}{\pi}}
 \int_{\infty}^{\zeta}
e^{-2s^2}\,ds\cdot\Bigl(1+O(\frac{1}{\sqrt{N}})\Bigr)
\\
 &=
 \Bigl(p^2+\frac{\varrho^2(1-\tau^2)^2}{4}\Bigr)^{\frac{1}{2}}
 \cdot
 \frac{ N\delta}{2}
 \erfc(\sqrt{2}\zeta)
\Bigl( 1+O(\frac{1}{\sqrt{N}})\Bigr).
    \end{split}
\end{align}
As a consequence, by \eqref{def of WN approx potential} and \eqref{bfKN to boldSN}, we obtain \eqref{eq for edge scaling lim}.

The second assertion follows from \eqref{eq for edge scaling lim}, \eqref{bfRN one point density} and Lemma~\ref{Lem_asymp of omega 123}. 
\end{proof}

\begin{rem}[$\nu>-1$ fixed case]
It is worth noting that Lemma~\ref{lem complex part edge kernel} essentially establishes the universality of the edge scaling limit of the one-point density for the non-Hermitian complex Wishart ensemble in the large-parameter regime \(\nu = \varrho N\). For the fixed-parameter case \(\nu > -1\), it was obtained in \cite[Theorem 1.4]{BN24}.  
\end{rem}

Now we are ready to show Lemma~\ref{lem L1 estimate}. 

\begin{proof}[Proof of Lemma~\ref{lem L1 estimate}]
Given a sufficiently small $\epsilon>0$, let us denote 
\begin{equation}
\alpha_N:=\xi_+ - \frac{1}{N^{\frac{1}{2}-\epsilon}},
 \qquad
\quad
 \beta_N:=\xi_- + \frac{1}{N^{\frac{1}{2}-\epsilon}}. 
\end{equation} 
Using these, we consider the decompositions 
\begin{equation}
   \label{expectation linear statistics bound plus minus} I_{\xi_+}:=I_{\xi_+}^1+I_{\xi_+}^2+I_{\xi_+}^3,
    \qquad
    I_{\xi_-}:=I_{\xi_-}^1+I_{\xi_-}^2+I_{\xi_-}^3,
\end{equation}
where for $\delta'>0$ small, 
\begin{align}
\label{expectation linear statistics bound plus 123}
&I_{\xi_+}^1:= \int_{\alpha_N}^{\xi_+} |f(x)|\mathbf{R}_{N}(x)\,dx,
\qquad
I_{\xi_+}^2:= \int_{\xi_+}^{\xi_++\delta'}
|f(x)|\mathbf{R}_{N}(x)\,dx,
\qquad
I_{\xi_+}^3:=
\int_{\xi_+ +\delta'}^{\infty}
|f(x)|\mathbf{R}_{N}(x)\, dx,
\\
\label{expectation linear statistics bound minus 123}
& I_{\xi_-}^1:= \int_{\xi_-}^{\beta_N} |f(x)|\mathbf{R}_{N}(x)\, dx,
\qquad
I_{\xi_-}^2:= \int_{\xi_- -\delta'}^{\xi_-} |f(x)|\mathbf{R}_{N}(x)\,dx,
\qquad
I_{\xi_-}^3:= \int_{-\infty}^{\xi_- -\delta'} |f(x)|\mathbf{R}_{N}(x)\, dx.
\end{align} 
Then  we have 
\begin{equation}
\label{expectation linear statistics bound}
\mathbb{E}\bigl[|\Xi_{N}^{\mathcal{E}_N^{\mathrm{c}}}(f)|\bigr] \leq I_{1} + I_{2},
\end{equation}
Since \( I_{\xi_-} \) can be estimated in a similar manner to \( I_{\xi_+} \), we focus on estimating \( I_{\xi_+} \). 
We show that given $\epsilon>0$ sufficiently small, there exists $c>0$ such that
\begin{equation} \label{I xi 123 asymp}
I_{\xi_+}^1=O( N^\epsilon ), \qquad I_{\xi_+}^2=O( 1 ), \qquad I_{\xi_+}^3=O( e^{-cN} ),
\end{equation}
which gives $I_{ \xi_+ }=O(N^\epsilon)$. See Figure~\ref{Fig_I xi decomp} for an illustration.

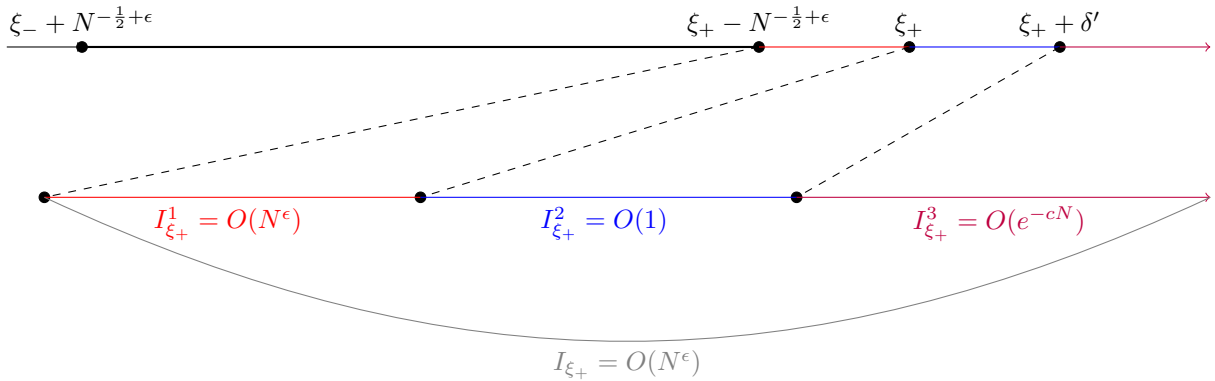
\begin{figure}[b]

 \begin{center}
 \begin{tikzpicture}[scale=2, every node/.style={align=center}]

\draw[black, ->] (-4,0) to (-3.5,0);
\filldraw[black] (-3.5,0) circle (1pt) node[above]{$ \xi_-+N^{-\frac{1}{2}+\epsilon} $};
\filldraw[black] (1,0) circle (1pt) node[above]{$ \xi_+-N^{-\frac{1}{2}+\epsilon} $};
\filldraw[black] (2,0) circle (1pt) node[above]{$ \xi_+ $};
\filldraw[black] (3,0) circle (1pt) node[above]{$ \xi_+ + \delta' $};

\draw[red] (1,0) to (2,0);
\draw[blue] (2,0) to (3,0);
\draw[purple, ->] (3,0) to (4,0);

\draw[black,thick] (-3.5,0) to (1,0);

\draw[black, dashed] (1,0) to (-3.75,-1);
\filldraw[black] (-3.75,-1) circle (1pt);  

\draw[black, dashed] (2,0) to (-1.25,-1);
\filldraw[black] (-1.25,-1) circle (1pt);  

\draw[black, dashed] (3,0) to (1.25,-1);
\filldraw[black] (1.25,-1) circle (1pt);

\draw[red] (-3.75,-1) -- (-1.25,-1) node[midway,below] { $I_{ \xi_+ }^1 = O( N^\epsilon )$  } ;
\draw[blue] (-1.25,-1) -- (1.25,-1) node[midway,below] { $I_{ \xi_+ }^2 = O(1 )$  };
\draw[purple, ->] (1.25,-1) -- (4,-1) node[midway,below] { $I_{ \xi_+ }^3 = O( e^{-c N} )$  };
 
\draw[gray] (-3.75,-1) to [out=-25,in=-155] node[below] {$ I_{\xi_+}=O(N^\epsilon) $} (4,-1); 
\end{tikzpicture}
 \end{center}
    \caption{Illustration of the decomposition of $I_{ \xi_+ }$}
    \label{Fig_I xi decomp}
\end{figure}

\medskip 

Recall from Lemma~\ref{Lem_RN RN hat RN2 decomp} that $\bfR_N=\widehat{\bfR}_{N,1}+\bfR_{N,2}$.
We begin with estimating $\bfR_{N,2}$. First note that by \eqref{def of bfRN 2 xy}, there exists $C>0$ such that for any $x\in\mathcal{E}_N^{\rm c}$, we have 
\[
\bigl|\bfR_{N,2}(x)\bigr| \leq  CN^{\frac{3}{2}}
\bigl|\bfR_{N,2}^{(1)}(x)\bigr|  \Bigl(  1+O(\frac{1}{N^{1/2}}) \Bigr). 
\]
By Lemma~\ref{lem Exponential regime}, we have 
\begin{align*}
\bigl|\widetilde{w}_{N,r}(t)
L_{N-r}^{(\nu)}\Bigl(\frac{N}{\tau}t\Bigr)\bigr|
\leq
\frac{C}{N}
\frac{\sqrt{|\psi_{\varrho}'(t)|}}{|\psi_{\varrho}(t)|^r}
e^{-\frac{N}{2}\Omega_{\varrho}(t)}
\Bigl( 1 + O( \frac{1}{N}) \Bigr),
\end{align*}
uniformly for $t\in \mathcal{E}_N^{\rm c}$, where $C>0$ does not depend on $N$. 
By Lemma~\ref{Lem Omega Positivity}, $\Omega_{\varrho}(t)\geq0$ for any $t\in \mathcal{E}_N^{\rm c}$. In addition, we have 
$$
\Omega_{\varrho}(t)=0 \quad \textup{for } t\in \mathcal{E}_N^{\rm c}, \qquad \textup{if and only if} \qquad t= \xi_\pm. 
$$ 
Note also that $\Omega_{\varrho}(t)$ grows linearly as $|t|\to\infty$. 
Then, by using Lemma~\ref{lem edge f varrho} and change of variables, we have
\begin{equation}
\label{def of bfRN2 f integrand Part 1}
\int_{\alpha_N}^{\xi_+}|f(x)||\bfR_{N,2}(x)|\,dx
\leq
C\sqrt{N}\int_{\alpha_N}^{\xi_+}
e^{-\frac{N}{2}\Omega_{\varrho}(t)}\,dt
\leq
C\int_{-\sqrt{\delta(\xi_+)}N^{\epsilon}}^{0}
e^{-s^2}\,ds
=O(1),
\end{equation}
for some $C>0$. 
Next, we observe that for $\delta'>0$ small, 
\begin{equation}
\label{def of bfRN2 f integrand Part 2}
\int_{\xi_+}^{\xi_+ + \delta'} |f(x)||\bfR_{N,2}(x)|\,dx
\leq
C\sqrt{N}\int_{\xi_+}^{\xi_+ + \delta'}
e^{-\frac{N}{2}\Omega_{\varrho}(t)}\,dt
\leq
C\int_{0}^{\delta'\sqrt{N\delta(\xi_+)}}
e^{-s^2}\,ds=O(1). 
\end{equation}
By combining the above fact with $\min_{x>\xi_+ + \delta'}\Omega_{\varrho}(x)\geq\Omega_{\varrho}(\xi_+ + \delta')>0$ due to Lemma~\ref{Lem Omega Positivity}, for $t\geq \xi_+ + \delta'$, we have  
\begin{equation}
\label{def of Omegavarho lower bound}
\Omega_{\varrho}(t)
\geq 
\vartheta t 
-(\xi_++\delta')\vartheta
+
\Omega_{\varrho}(\xi_+ +\delta'),
\qquad
\vartheta=\Omega_{\varrho}'(\xi_+ + \delta')>0.
\end{equation}
Then, by the assumption \eqref{f integrable bound}, for $x\geq \xi_+ + \delta'$, we have 
\begin{align}
\begin{split}
\label{def of bfRN2 f integrand Part 3}
\int_{\xi_+ + \delta'}^{\infty}|f(x)||\bfR_{N,2}(x)|\,dx
&\leq 
C\chi(f)\sqrt{N}e^{-\frac{N-1}{2}\vartheta(\xi_+ + \delta')-\frac{N}{2}\Omega_{\varrho}(\xi_+ +\delta')+\frac{N}{2}(\xi_+ + \delta')\vartheta}\int_{\xi_+ + \delta'}^{\infty}
 e^{-(\vartheta-2c)\frac{x}{2}}\,dx
 \\
 &=
 O(e^{-c_{\delta'}N}),
\end{split}    
\end{align}
for $c_{\delta'}>0$, where we have chosen $c>0$ so that $\vartheta>2c$.

Now we are ready to complete the proof. 
By \eqref{def of bfRN2 f integrand Part 1} and \eqref{edge scaling for RN hat}, we have 
\begin{equation}
\label{def of bfRN1 + bfRN2 bound 1}
I_{\xi_+}^{1}
\leq C \int_{\alpha_N}^{\xi_{+}} \widehat{\bfR}_{N,1}(x)\,dx +O(1) \leq C \int_{0}^{\sqrt{\delta}N^{\epsilon}}\erfc(\sqrt{2}s)\,ds+O(1) =O(N^{\epsilon}),
\end{equation}
for some $C >0$, where $\delta=\delta(\xi_+)$ is given by \eqref{def of local density}. 
By using \eqref{def of bfRN2 f integrand Part 2} and \eqref{edge scaling for RN hat} again, we have 
\begin{equation}
 \label{def of bfRN1 + bfRN2 bound 2}
I_{\xi_+}^2 \leq C \int_{0}^{\delta'\sqrt{N\delta(\xi_+)}}   \frac{1}{\sqrt{N\delta(\xi_+)}} \mathbf{R}_{N,1} \Bigl(\xi_+ + \frac{u}{\sqrt{N\delta}}\Bigr) \, du + O(1) \leq C \int_{0}^{\infty}\erfc(\sqrt{2}s)\,ds+O(1) =O(1),
\end{equation}
for some $C>0$. %which may change in each line.
Finally, for $x>\xi_+ + \delta'$, it follows from \eqref{hat bfK N complex}, \eqref{def of bold SN}, \eqref{def of WN approx potential} and \eqref{estimate 1} that  
\[
\bfR_N^{\rm c}(x)
\leq
CN^{\frac{3}{2}}
\frac{\omega_N^{\rm c}(x)}{\sqrt{x^2+\frac{\varrho^2(1-\tau^2)^2}{4}}W_N^{\rm c}(x,x)}
\int_x^{\infty}e^{-N\Omega_{\varrho}(t)}\,dt
\leq
CN^{\frac{3}{2}}
\int_{x}^{\infty}e^{-N\Omega_{\varrho}(t)}\,dt. 
\] 
Here we have used the boundary condition at $x\to\infty$. 
Then by Lemmas~\ref{Lem_RN RN hat RN2 decomp} and ~\ref{Lem_asymp of omega 123} together with \eqref{def of Omegavarho lower bound}, we have for $x\geq\xi_+ + \delta'$, 
\[
\widehat{\bfR}_{N,1}(x)
\leq 
CN e^{-N\Omega_{\varrho}(\xi_+ + \delta')}
\int_{x}^{\infty} 
e^{-N(\vartheta t-(\xi_+ + \delta')\vartheta)}dt
\leq
CN e^{-N\Omega_{\varrho}(\xi_+ + \delta')}
e^{-N(\vartheta x-(\xi_+ + \delta')\vartheta)}.
\]
Then it follows from \eqref{def of bfRN2 f integrand Part 3} that
\begin{equation}
\label{def of bfRN1 + bfRN2 bound 3}
 I_{\xi_+}^3
 \leq 
 CNe^{-(N-1)(\vartheta-c)(\xi_{+}+\delta')-N(\Omega_{\varrho}(\xi_+ +\delta')-(\xi_+ + \delta')\vartheta)}\int_{\xi_+ + \delta'}^{\infty}
 e^{-(\vartheta-c)x}\,dx+O(e^{-c_{\delta'}N})
 =
 O(e^{-c_{\delta'}'N}),
\end{equation}
where a small positive constant $c$ is chosen to satisfy $0<2c<\vartheta$ (thus, $c_{\delta'}>0$), and $c_{\delta'}'>0$ is a positive constant, which depend on $\delta'>0$ but does not depend on $N$.
Therefore, by \eqref{def of bfRN1 + bfRN2 bound 1}, \eqref{def of bfRN1 + bfRN2 bound 2}, and \eqref{def of bfRN1 + bfRN2 bound 3}, we have shown \eqref{I xi 123 asymp}, which completes the proof.
\end{proof}

\section{Proofs of Theorems~\ref{Thm_expected number} and~\ref{Thm_global density} at weak non-Hermiticity}

\label{Section_weak}

\subsection{Outline of the proof}
\label{subsection outline of proof at weak} 

This section is devoted to proving our main results in the regime of weak non-Hermiticity.  
Throughout this section, we assume that \( \tau = 1 - \frac{\alpha^2}{2N} \) with \( \alpha > 0 \). As in the previous section, we assume that \( \nu = \varrho N \) with \( \varrho > 0 \), while the case of fixed \( \nu \) will be addressed in the appendix.

Give $0<\epsilon <\frac{2}{3}$ sufficiently small, let 
\begin{equation} \label{def of mathcal BN}
\mathcal{B}_{N}:=\bigl[\lambda_- + N^{-\frac{2}{3}+\epsilon},\lambda_+ - N^{-\frac{2}{3}+\epsilon}   \bigr]. 
\end{equation}
As a counterpart of Proposition~\ref{Prop_asymptotic of 1pt function RN strong}, we have the following. 

\begin{prop}[\textbf{Macroscopic behaviour of the one-point density}] 
\label{prop of one density weak macro}
Assume $\nu=\varrho\,N$ for $\varrho>0$. 
Let $\tau=1-\frac{\alpha^2}{2N}$ with fixed $\alpha>0$. 
Then as $N \to \infty$, we have
\begin{equation}
\label{def of Macro bfRN weak}
\bfR_{N}(x) = \frac{N}{2\alpha\sqrt{\pi}} \frac{1}{\sqrt{x}} \erf\Bigl(  \frac{\alpha}{2} \sqrt{\frac{(x-\lambda_-)(\lambda_+ - x)}{x}} \Bigr)
(1+ o(1) ),
\end{equation}
uniformly for $x\in \mathcal{B}_{N}$.
\end{prop}

As in the regime of strongly non-Hermiticity, we split the linear statistics \eqref{linear statistics f} into 
\begin{equation}
    \Xi_N(f)=\Xi_N^{\mathcal{B}_N}(f)+\Xi_N^{\mathcal{B}_N^{\rm c}}(f).
\end{equation}
Then as a counterpart of Lemma~\ref{lem L1 estimate}, we have the following.

\begin{lem}[\textbf{$L^1$-estimate for linear statistics}]
\label{Lem L1 estimate at weak}
Assume $\nu=\varrho\,N$ for $\varrho>0$.
Let $\tau=1-\frac{\alpha^2}{2N}$ with fixed $\alpha>0$.
Let $f$ be a locally integrable and measurable function satisfying 
\begin{equation}
\label{def of test function at weak}
    \chi(f):=\sup_{x\in\R}\Bigl(|f(x)|e^{-\gamma|x|}\Bigr)<\infty,
\end{equation}
for some small positive $\gamma>0$.
For $0<\epsilon<\frac{2}{3}$ sufficiently small, the complementary statistics $\Xi_N^{\mathcal{B}_N^{\rm c}}(f)$ satisfies the $L^1$-estimate
\begin{equation}
\label{def of L1 estimate at weak}
    \mathbb{E}\Bigl[\bigl|\Xi_N^{\mathcal{B}_N^{\rm c}}(f)\bigr|\Bigr]=
    O(N^{\frac{1}{3}+\epsilon}).
\end{equation}
\end{lem}

Then we can complete the proof of Theorem. 

\begin{proof}[Proof of Theorem~\ref{Thm_expected number} (ii) and~\ref{Thm_global density} (ii)]
By combining Proposition~\ref{prop of one density weak macro}, Lemma~\ref{Lem L1 estimate at weak} and \eqref{c(alpha) integral}, we have 
\begin{equation}
E_{N,\tau}^{\nu} = \int_{\mathcal{B}_N}\bfR_N(x)\,dx \, (1+o(1)) = N\,c(\alpha)
(1+o(1)),
\end{equation}
as $N\to\infty$. This gives rise to Theorem~\ref{Thm_expected number} (ii). 
Then Theorem~\ref{Thm_global density} (ii) immediately follows from Theorem~\ref{Thm_expected number} (ii) and Lemma~\ref{Lem L1 estimate at weak}.
\end{proof}

\subsection{Macroscopic behaviour of the one-point density; Proof of Proposition~\ref{prop of one density weak macro}}
\label{Subsection Macroscopic one point density at weak}

In this subsection, we prove Proposition~\ref{prop of one density weak macro}. Our proof follows from the subsequent lemmas, which serve as counterparts to Lemmas~\ref{Lem_bold SN dens} and~\ref{Lem_asymptotic of RN2}.

\begin{lem}  
\label{lem bfRN1 leading term at weak}
Assume $\nu=\varrho\,N$ for $\varrho>0$ and $\tau=1-\frac{\alpha^2}{2N}$ with fixed $\alpha>0$. 
Then as $N \to \infty$, we have 
\begin{equation}
\bfR_N^{\mathrm{c}}(x)
=
\frac{N^2}{2\alpha^2}
\frac{1}{x}
\erf\Bigl( 
\frac{\alpha}{2}
\sqrt{\frac{(x-\lambda_-)(\lambda_+ - x)}{x}}
\Bigr) (1+o(1)),
\end{equation}
uniformly for $x\in\mathcal{B}_{N}$. Consequently, we have  
\begin{equation}
\label{def of harbfRN1 weak}
\widehat{\bfR}_{N,1}(x) = \frac{N}{2\alpha\sqrt{\pi}} \frac{1}{\sqrt{x}} \erf\Bigl(  \frac{\alpha}{2} \sqrt{\frac{(x-\lambda_-)(\lambda_+ - x)}{x}} \Bigr) (1+o(1)),
\end{equation}
uniformly for $x\in \mathcal{B}_{N}$.
\end{lem}

\begin{lem} 
\label{lem prepare L^1 estimate}
Assume $\nu=\varrho\,N$ for $\varrho>0$ and $\tau=1-\frac{\alpha^2}{2N}$ with fixed $\alpha>0$. 
Then as $N\to\infty$, we have
\begin{equation}
\label{def of bfRN2 estimate at weak}
  \bfR_{N,2}(x)=O(1),
\end{equation}
uniformly for $x\in\mathcal{B}_N$. 
\end{lem} 

These lemmas, together with Lemma~\ref{Lem_RN RN hat RN2 decomp}, immediately yield Proposition~\ref{prop of one density weak macro}.
We first show Lemma~\ref{lem bfRN1 leading term at weak}.

\begin{proof}[Proof of Lemma~\ref{lem bfRN1 leading term at weak}]
By \eqref{approx of RN c and mathfrak RN c}, it is enough to show 
\begin{align}
\label{def of frakRNc x}
\mathfrak{R}_N^{\mathrm{c}}(x) = \frac{N^2}{2\alpha^2} \erf\Bigl(  \frac{\alpha}{2} \sqrt{\frac{(x-\lambda_-)(\lambda_+ - x)}{x}} \Bigr) 
(1+o(1)),
\end{align}
uniformly for $x\in \mathcal{B}_{N}$. We again use \eqref{decomp of mathfrak RN c} and Lemma~\ref{Lem_ODE for mathfrak RN}.  

First, we show that $N^{-2}\mathfrak{R}_{N}^{\rm c}(\lambda_-)\to0$ as $N\to\infty$, which will be used as the initial condition. 
Note that by Stirling formula, we have 
\begin{align}
\label{cINdiag weak 1}
\sqrt{\frac{\pi(1-\tau^2)}{4N}}
&=
\frac{\alpha\sqrt{\pi}}{2N}
+O(\frac{1}{N^2}),
\\
\label{cINdiag weak 2}
\frac{ N^{\nu+2}}{1-\tau^2}
\frac{(N-1)!\,\tau^{2N-3}}{\Gamma(N+\nu-1)}
&=
\frac{(1+\varrho)^{\frac{3}{2}}e^{-\alpha^2}}{\alpha^2}
\frac{N^4\, e^{N\varrho}}{(1+\varrho)^{(1+\varrho)N}}
 \Bigl( 1 + O(\frac{1}{N}) \Bigr).
\end{align}
For a fixed $x>0$, it follows from Taylor expansion that 
\begin{align}
\begin{split}
\label{cINdiag weak 3}
&\quad 
    \frac{\Bigl(
\sqrt{x^2+\frac{\varrho^2(1-\tau^2)^2}{4} }
+
\frac{\varrho(1-\tau^2)}{2}
    \Bigr)^{\varrho N }
    }{\Bigl(
    x^2+\frac{\varrho^2(1-\tau^2)^2}{4}
    \Bigr)^{\frac{1}{4}}}
    \exp\Bigl( 
-\frac{2N}{1-\tau^2}\Bigl( 
\sqrt{x^2+\frac{\varrho^2(1-\tau^2)^2}{4} }-\tau x
\Bigr)\Bigr)
\\
&=
\frac{1}{\sqrt{x}}
x^{N\varrho}\exp\Bigl(-Nx-\frac{\alpha^2(x^2-\varrho^2)}{4x} \Bigr)
 \Bigl( 1 + O(\frac{1}{N}) \Bigr).
\end{split}    
\end{align}
Therefore, combining \eqref{cINdiag weak 1}, \eqref{cINdiag weak 2}, and \eqref{cINdiag weak 3} with \eqref{calIN1x}, we obtain 
\begin{equation}
\label{calIN1 asymp}
    \mathcal{I}_{N,1}(x)
    =
    N^3\sqrt{\frac{\pi}{2}}
\frac{(1+\varrho)^{\frac{3}{2}}e^{-\alpha^2}}{\sqrt{2}\alpha(1+\varrho)^{(1+\varrho)N}}
\frac{1}{\sqrt{x}}
x^{N\varrho}\exp\Bigl(-N(x-\varrho)-\frac{\alpha^2(x^2-\varrho^2)}{4x} \Bigr)
 \Bigl( 1 + O(\frac{1}{N}) \Bigr),
\end{equation}
uniformly for $x>0$.
On the other hand, let $\delta>0$ be small, and for $\lambda_- -\delta<x<\lambda_-$, by Lemma~\ref{Lem Critical regime} and \eqref{calIN2x}, we have  
\begin{align}
\begin{split}
\label{def of calIN2c Airy left edge}
 \mathcal{I}_{N,2}(x)
& =
\frac{1}{N^{\frac{2}{3}}}
\frac{(1+\varrho)^{(1+\varrho)N-\frac{3}{2}+\frac{1}{12}}}{(\sqrt{1+\varrho}-1)^{\frac{2}{3}}x^{N\varrho}}
(\lambda_+-x)^{\frac{1}{2}}
e^{\frac{2N+\alpha^2}{2}(x-\varrho)}
\\
&\quad\times
\Ai^2\Bigl( 
-\frac{N^{\frac{2}{3}}(1+\varrho)^{\frac{1}{6}}}{(\sqrt{1+\varrho}-1)^{\frac{4}{3}}}
(x-\lambda_-)
\Bigr)
\Bigl(1+O((x-\lambda_-)^{\frac{3}{2}}) \Bigr). 
\end{split}    
\end{align}
Combining \eqref{calIN1 asymp} and \eqref{def of calIN2c Airy left edge} with Lemma~\ref{Lem_ODE for mathfrak RN}, we have 
\begin{align*}
\mathfrak{R}_N^{\rm c}(\lambda_-)-\mathfrak{R}_N^{\rm c}(\lambda_- -\delta)
&=
 N^{\frac{7}{3}}
 \frac{\sqrt{\pi}e^{-\alpha^2}}{2\alpha}
 \frac{(1+\varrho)^{+\frac{1}{12}}}{(\sqrt{1+\varrho}-1)^{\frac{2}{3}}}
\\
& \times \int_{\lambda_--\delta}^{\lambda_-} 
\frac{e^{\frac{\alpha^2(x-\varrho)^2}{4x}}(\lambda_+-x)^{\frac{1}{2}}}{\sqrt{x}}
\Ai^2\Bigl( 
-\frac{N^{\frac{2}{3}}(1+\varrho)^{\frac{1}{6}}(x-\lambda_-)}{(\sqrt{1+\varrho}-1)^{\frac{4}{3}}}
\Bigr)
\Bigl(1+O((\lambda_- - x)^{\frac{3}{2}}) \Bigr)\,dx. 
\end{align*}
We use the inequality  
$$ \Ai(x)<  \frac{ x^{-\frac{1}{4}} }{ 2\sqrt{\pi} } \, e^{-\frac{2}{3}x^{\frac{3}{2}}},  \qquad \textup{for }x>0;$$
see e.g. \cite[Eq. (9.7.15)]{NIST}. 
Then for a sufficiently large $N$ and $\lambda_- - \delta<x<\lambda_-$, we have 
\begin{align*}
\frac{e^{\frac{\alpha^2(x-\varrho)^2}{4x}}(\lambda_+-x)^{\frac{1}{2}}}{\sqrt{x}}
\Ai^2\Bigl( 
-\frac{N^{\frac{2}{3}}(1+\varrho)^{\frac{1}{6}}(x-\lambda_-)}{(\sqrt{1+\varrho}-1)^{\frac{4}{3}}}
\Bigr)
&\leq
\frac{C}{N^{\frac{1}{3}}}\frac{e^{\frac{\alpha^2(x-\varrho)^2}{4x}}(\lambda_+-x)^{\frac{1}{2}}}{\sqrt{x}(\lambda_- - x)^{\frac{1}{2}}}
e^{-\frac{4}{3}N(\lambda_- - x)^{\frac{3}{2}}}
\\
&
\leq
\frac{C}{N^{\frac{1}{3}}}
\frac{e^{-\frac{4}{3}N(\lambda_- - x)^{\frac{3}{2}}}}{(\lambda_- - x)^{\frac{1}{2}}},
\end{align*}
for some $C>0$. Hence we obtain 
\begin{equation}
\label{def of initial estimate 1}
\mathfrak{R}_N^{\rm c}(\lambda_-)-\mathfrak{R}_N^{\rm c}(\lambda_- -\delta)
=
O(N^{\frac{5}{3}}). 
\end{equation}
On the other hand, by Lemma~\ref{lem Exponential regime}, \eqref{calIN1x}, and \eqref{calIN2x}, we have 
\[
\mathfrak{R}_N^{\rm c}(\lambda_- - \delta)-\mathfrak{R}_N^{\rm c}(0)
=
\frac{(1+\varrho)^{\frac{1}{2}}N^2}{4\alpha\sqrt{\pi}}
\int_{0}^{\lambda_- - \delta}
\frac{1}{\sqrt{|x|}}
\frac{\psi'_{\varrho}(x)}{\psi_{\varrho}(x)^{3}}
e^{-N\Omega_{\varrho}(x)}\,dx
\Bigr|_{\tau=1-\frac{\alpha^2}{2N}}\cdot\Bigl(1+O(\frac{1}{N}) \Bigr). 
\]
(Note that the positivity of $\Omega_{\varrho}(t)$ for $t > \lambda_+$ and $t < \lambda_-$, as stated in Lemma~\ref{Lem Omega Positivity}, holds for $\tau = 1 - \frac{\alpha^2}{2N}$ with fixed $\alpha > 0$. Moreover, Lemma~\ref{lem Exponential regime} remains valid for $\tau = 1 - \frac{\alpha^2}{2N}$.)

By \eqref{def Omega varrho}, for sufficiently large \( N \), \( \Omega_{\varrho}(t) \) evaluated at \( \tau = 1 - \frac{\alpha^2}{2N} \) is a monotonically increasing function for all \( t \geq \lambda_+ + \delta \) and grows linearly as \( |t| \to \infty \).
In particular, by Lemma~\ref{Lem Omega Positivity}, since \( \min_{x > \lambda_+ + \delta} \Omega_{\varrho}(t) \geq \Omega_{\varrho}(\lambda_+ + \delta) \), it follows that for \( t \geq \lambda_+ + \delta \), we have  
\begin{equation}
\label{def of tangent omega varrho}
\Omega_{\varrho}(t)|_{\tau=1-\alpha^2/(2N)}
\geq \mu t-(\lambda_+ + \delta)\mu+\widetilde{\Omega}_{\varrho}(\lambda_+ + \delta)
=:\varphi(t)>0,
\end{equation}
where $\widetilde{\Omega}_{\varrho}(t)$ is the $O(1)$-leading term of $\Omega_{\varrho}(t)$ as $N\to\infty$, which is the function with respect to $t$ and can be explicitly given. 
Here, $\mu:=\widetilde{\Omega}_{\varrho}'(\lambda+\delta)>0$. 
In particular, $\varphi(\lambda_+ + \delta)>0$. 
Therefore, there exists $c>0$ such that we have 
\begin{equation}
\label{def of initial estimate 2}
\mathfrak{R}_N^{\rm c}(\lambda_- - \delta)-\mathfrak{R}_N^{\rm c}(0)
=O(N^2e^{-cN}).
\end{equation}
Note that as in \eqref{def of boldSN 00}, it is straightforward to see that  $\mathfrak{R}_N^{\rm c}(0)=O(\frac{{c'}^N}{N^{N\varrho-1}})$ for some $c'>0$. 
Therefore, combining \eqref{def of initial estimate 1} with \eqref{def of initial estimate 2}, we obtain 
\begin{equation}
\label{def of initial condition estimat v3}
\mathfrak{R}_{N}^{\rm c}(\lambda_-)=O(N^{\frac{5}{3}}).    
\end{equation}

We now complete the proof of \eqref{def of frakRNc x}. 
By Lemma~\ref{Lem Oscillatory regime} and \eqref{calIN2x}, we have 
\begin{align*}
\mathcal{I}_{N,2}(x)
&=
\frac{2}{\pi N}
   \frac{(-1)^{2N-3}
   (1+\varrho)^{N(1+\varrho)-1}
   }{(x-\lambda_-)^{\frac{1}{2}}
(\lambda_+-x)^{\frac{1}{2}}}
    \frac{\exp\bigl(\frac{2N+\alpha^2}{2}(x-\varrho)\bigr)}{x^{N\varrho}}
\\
&\quad\times
\Bigl\{
\cos(\Psi_{N-1,0}(x))\cos(\Psi_{N-2,0}(x))
+
\cos(\Psi_{N-2,1}(x))\cos(\Psi_{N-1,-1}(x))
\Bigr\}
 \Bigl( 1 + O(\frac{1}{N}) \Bigr),
\end{align*}
uniformly for $x\in\mathcal{B}_N$, where $\lambda_{\pm}$ are given by \eqref{MP law}.
Note that by the trigonometric identity, we have 
\begin{align*}  
\cos(\Psi_{N-1,0}(x))\cos(\Psi_{N-2,0}(x)) + \cos(\Psi_{N-2,1}(x))\cos(\Psi_{N-1,-1}(x)) = \Pi(x)-\Theta_N(x), 
\end{align*}
where 
\begin{align} 
\Pi(x):= \frac{1}{2\sqrt{1+\varrho}} \frac{x^2-\varrho^2}{2x},
\qquad  
\label{ThetaNx}
\Theta_N(x) := \sin\Bigl( 2N\boldsymbol{\psi}^{(1)}(x) + 2\boldsymbol{\psi}^{(2)}(x) + 2\arccos\boldsymbol{p}(x) \Bigr).  
\end{align}
By elementary trigonometric identity, we have 
\begin{align*}
\Theta_N(x) &= \sin\bigl(  2N\boldsymbol{\psi}^{(1)}(x) + 2\boldsymbol{\psi}^{(2)}(x) \bigr) \cos\bigl(2\arccos\boldsymbol{p}(x) \bigr) + \cos\bigl(  2N\boldsymbol{\psi}^{(1)}(x) + 2\boldsymbol{\psi}^{(2)}(x) \bigr) \sin\bigl(2\arccos\boldsymbol{p}(x) \bigr)
\\
&= \sin\bigl(  2N\boldsymbol{\psi}^{(1)}(x) + 2\boldsymbol{\psi}^{(2)}(x) \bigr) \Bigl(2\boldsymbol{p}(x)^2-1 \Bigr) + \cos\bigl(  2N\boldsymbol{\psi}^{(1)}(x) + 2\boldsymbol{\psi}^{(2)}(x) \bigr) 2\boldsymbol{p}(x) \sqrt{1-\boldsymbol{p}(x)^2}. 
\end{align*}
Combining all of the above with \eqref{diagonal IN c natural}, we obtain  
\begin{align*}
\mathfrak{I}_N^{\mathrm{c}}(x)
&=
-
N^2
\sqrt{\frac{2}{\pi }}
\frac{(1+\varrho)^{\frac{1}{2}}}{\sqrt{2}\alpha}
   \frac{1}{\sqrt{x}}
   \frac{
   \exp\Bigl(
-\frac{(\lambda_+- x)(x-\lambda_-)}{4x}\alpha^2\Bigr)
\Bigl(\Pi(x)-\Theta_N(x) \Bigr)
   }{(\lambda_+-x)^{\frac{1}{2}}(x-\lambda_-)^{\frac{1}{2}}}
 \Bigl( 1 + O(\frac{1}{N}) \Bigr)
 \\
 &=
 -N^2 \Bigl(
 \mathfrak{I}_{N,1}^{\mathrm{c}}(x)
 -
 \mathfrak{I}_{N,2}^{\mathrm{c}}(x) \Bigr) 
  \Bigl( 1 + O(\frac{1}{N}) \Bigr),
\end{align*}
uniformly for $x\in\mathcal{B}_N$, where 
\begin{align}
\label{boldsymbolIIc}
 \mathfrak{I}_{N,1}^{\mathrm{c}}(x)
 &=
 \frac{1}{4\sqrt{\pi}\alpha}
 \frac{
   \exp\Bigl(
\frac{(x - \lambda_+)(x-\lambda_-)}{4x}\alpha^2\Bigr)
   }{(\lambda_+-x)^{\frac{1}{2}}(x-\lambda_-)^{\frac{1}{2}}}
   \frac{x^2-\varrho^2}{x^{\frac{3}{2}}}  = - \frac{1}{2\alpha^2} \frac{d}{dt}
\biggl[
\erf\Bigl( 
\frac{\alpha}{2}
\sqrt{\frac{(t-\lambda_-)(\lambda_+ - t)}{t}}
\Bigr)
\bigg]   ,
   \\
   \label{boldsymbolIIIN}
\mathfrak{I}_{N,2}^{\mathrm{c}}(x)
 &=
 \sqrt{\frac{2}{\pi }}
\frac{(1+\varrho)^{\frac{1}{2}}}{\sqrt{2}\alpha}
   \frac{1}{\sqrt{x}}
   \frac{
   \exp\Bigl(
-\frac{(\lambda_+- x)(x-\lambda_-)}{4x}\alpha^2\Bigr)
   }{(\lambda_+-x)^{\frac{1}{2}}(x-\lambda_-)^{\frac{1}{2}}}
   \Theta_N(x).
\end{align}

We recall a basic estimate for the oscillatory integral from \cite[Lemma 3.7]{BL23}.
Let $f$ and $\psi$ be $\mathcal{C}^2$-functions on a neighborhood of an interval $[a, b]$.
    Suppose $\psi$ has no critical point in $[a, b]$.
    Then as $N \to \infty$, we have
    \begin{equation} \label{asymp of oscillatory integral}
        \int_{a}^{b} f(u) e^{iN\psi(u)} du = i \Big( \frac{f(a)}{\psi'(a)} e^{iN\psi(a)} - \frac{f(b)}{\psi'(b)} e^{iN\psi(b)} \Big) \frac{1}{N} + O(\frac{1}{N^{2}}).
    \end{equation}
Notice that $\Theta_N$ in \eqref{ThetaNx} can be written as 
\begin{align*}
\Theta_N(x)
&=
(2\boldsymbol{p}(x)^2-1)
\im\bigl[e^{2i(N\boldsymbol{\psi}^{(1)}(x)
+
\boldsymbol{\psi}^{(2)}(x))}
\bigr]
+
2\boldsymbol{p}(x)\sqrt{1-2\boldsymbol{p}(x)^2}
\re\bigl[ e^{2i(N\boldsymbol{\psi}^{(1)}(x) + \boldsymbol{\psi}^{(2)}(x))} 
\bigr].
\end{align*}
Using this and \eqref{asymp of oscillatory integral}, we obtain 
\[
\int_{\lambda_-}^{x} \mathfrak{I}_{N,2}^{\mathrm{c}}(t)\,dt = O(\frac{1}{N}).
\]
Combining \eqref{def of initial condition estimat v3} with all of the above, we obtain
\begin{align*}
\mathfrak{R}_N^{\mathrm{c}}(x)
&=
\frac{N^2}{2\alpha^2}
\int_{\lambda_-}^{x}
\frac{d}{dt}
\Bigl[
\erf\Bigl( 
\frac{\alpha}{2}
\sqrt{\frac{(t-\lambda_-)(\lambda_+ - t)}{t}}
\Bigr)
\Bigr]
dt
\Bigl( 1 + O(\frac{1}{N^{\frac{1}{3}}}) \Bigr)
\\
&=
\frac{N^2}{2\alpha^2}
\erf\Bigl( 
\frac{\alpha}{2}
\sqrt{\frac{(x-\lambda_-)(\lambda_+ - x)}{x}}
\Bigr)
\Bigl( 1 + O(\frac{1}{N^{\frac{1}{3}}}) \Bigr),
\end{align*}
uniformly for $x\in \mathcal{B}_{N}$.  
Therefore, by \eqref{hat bfK N complex}, we obtain the first assertion.
\eqref{def of harbfRN1 weak} follows from \eqref{hat bfK N complex}, Lemma~\ref{Lem_RN RN hat RN2 decomp}, \eqref{def of omega N 1}, \eqref{def of omega N 2}, and \eqref{def of omega N 3}.   
\end{proof}

We next show Lemma~\ref{lem prepare L^1 estimate}.

\begin{proof}[Proof of Lemma~\ref{lem prepare L^1 estimate}] 

By \eqref{def of bfRN 2 xy}, it suffices to derive the asymptotic behaviours of $\bfR_{N,2}^{(1)}$ and $\bfR_{N,2}^{(2)}$ given in \eqref{def of bfSN xy 2 1} and \eqref{def of bfSN xy 2 2}, respectively.
Then it suffices to show that 
\begin{equation} \label{asymp of RN2 1 2 weak}
\bfR_{N,2}^{(1)}(x)=\mathcal{O}( \frac{1}{N^{\frac{3}{2}}} ), \qquad  \bfR_{N,2}^{(2)}(x)=\mathcal{O}( \frac{1}{N^{\frac{5}{2}}} ),
\end{equation}
uniformly for $x\in\mathcal{B}_N$.   

Note that by definition of $\widetilde{w}_{N,r}$ in \eqref{def of w tilde N r}, for $t>0$, it follows from \eqref{Bessel K function Large order asymptotics} that 
\begin{equation}
\label{def of tilde weight asymp}
\widetilde{w}_{N,r}(t) = \frac{\alpha}{N}\sqrt{\frac{\pi}{2}} \frac{(1+\varrho)^{\frac{3}{4}}e^{-\frac{\alpha^2}{2}}}{(1+\varrho)^{\frac{N(1+\varrho)}{2}}} t^{\frac{N\varrho}{2}-\frac{1}{2}} e^{-\frac{N}{2}(t-\varrho)-\frac{\alpha^2(t^2-\varrho^2)}{8t}} \Bigl( 1+O(\frac{1}{N}) \Bigr),
\end{equation}
and by Lemma~\ref{Lem Oscillatory regime}, as $N\to\infty$ and for $r=1,2$ and $t \in \mathcal{B}_N$, we have 
\begin{equation} \label{5.25}
L_{N-r}^{(\nu)}\Bigl(\frac{N}{\tau}t\Bigr)
=
\frac{(1+\varrho)^{\frac{1}{4}-\frac{r}{2}}}{\sqrt{N}}
\sqrt{\frac{2}{\pi}}
\frac{(-1)^{N-r}(1+\varrho)^{\frac{N(1+\varrho)}{2}}}{(t-\lambda_-)^{\frac{1}{4}}(\lambda_+-t)^{\frac{1}{4}}}
\frac{e^{\frac{2N+\alpha^2}{4}(t-\varrho)}}{t^{\frac{N\varrho}{2}}}
\cos\bigl( 
\Psi_{N-r,0}(t)
\bigr)
\Bigl(1+O(\frac{1}{N})\Bigr).
\end{equation}
Then the estimate for $\bfR_{N,2}^{(1)}$ in \eqref{asymp of RN2 1 2 weak} follows from straightforward computations using \eqref{def of bfSN xy 2 1}. 

Next, we estimate $\mathbf{R}_{N,2}^{(2)}$.  
Recall that $J_r[A]$ is defined in \eqref{def of JAr}, and that $\mathbf{R}_{N,2}^{(2)}$ can be expressed in terms of $J_r[A]$ as given in \eqref{RN2 in terms of J}.  
Furthermore, recall that $\mathcal{B}_N$ is defined in \eqref{def of mathcal BN}.  
For convenience, we write $ \mathcal{B}_{N,1}= \mathcal{B}_N $.  
Given a small $\epsilon > 0$, we introduce the following partition of the real axis:  
\begin{align*} 
&\mathcal{B}_{N,2}:=\bigl[\lambda_+ -  N^{-\frac{2}{3}+\epsilon},\lambda_+ +  N^{-\frac{2}{3}+\epsilon}\bigr), \qquad 
\mathcal{B}_{N,3}:=\bigl[\lambda_+ +  N^{-\frac{2}{3}+\epsilon},\infty\bigr),  
\\
& \mathcal{B}_{N,4}:=\bigl(\lambda_- - N^{ -\frac{2}{3} + \epsilon },\lambda_- +  N^{ -\frac{2}{3} + \epsilon } \bigr], \qquad 
\mathcal{B}_{N,5} :=\bigl(-\infty,\lambda_- -  N^{ -\frac{2}{3} + \epsilon }  \bigr].
\end{align*} 
Then we show that as $N\to\infty$,  
\begin{gather}
\label{def of Jcal BN1r asymp}
J_2[\mathcal{B}_{N,1}] =O(\frac{1}{N^{ \frac{5}{2} } } ), \qquad 
J_2[\mathcal{B}_{N,3}]=O(e^{-cN^{\frac{3}{2}\epsilon}}),\qquad
J_2[\mathcal{B}_{N,5}]=O(e^{-cN^{\frac{3}{2}\epsilon}}),
\\
\label{def of Jcal BN2r asymp}
J_{2}[\mathcal{B}_{N,2}] = \frac{\alpha}{(1+\varrho)^{\frac{1}{4}}}
\sqrt{\frac{\pi}{2}} \frac{1}{N^2}\Bigl(1+O(\frac{1}{N})\Bigr),
\qquad  J_{2}[\mathcal{B}_{N,4}] =
\frac{\alpha}{(1+\varrho)^{\frac{1}{4}}}
\sqrt{\frac{\pi}{2}}
\frac{1}{N^2}\Bigl(1+O(\frac{1}{N})\Bigr). 
\end{gather} 
Then since $J_{2}[\R]=\sum_{k=1}^5J_{2}[\mathcal{B}_{N,k}]$, we obtain 
\begin{equation}
\label{def of calJN2 estimate}
    J_{2}[\R]= 
\frac{\alpha\sqrt{2\pi}}{(1+\varrho)^{\frac{1}{4}}}\frac{1}{N^2}\Bigl(1+O(\frac{1}{N})\Bigr). 
\end{equation} 
Combining these with \eqref{RN2 in terms of J}, the estimate for $\bfR_{N,2}^{(2)}$ in \eqref{asymp of RN2 1 2 weak} follows. 

Now, we estimate each term $J_{2}[\mathcal{B}_{N,k}]$ for $k = 1,2,3,4,5$.  
Since $J_{2}[\mathcal{B}_{N,4}]$ and $J_{2}[\mathcal{B}_{N,5}]$ can be handled with minor modifications to the estimates for $J_{2}[\mathcal{B}_{N,2}]$ and $J_{2}[\mathcal{B}_{N,3}]$, respectively, we focus on estimating $J_{2}[\mathcal{B}_{N,k}]$ for $k = 1,2,3$.

Let $t\in\mathcal{B}_{N,1}$. Then by \eqref{5.25} and \eqref{asymp of oscillatory integral}, we have 
\begin{equation}
    \label{def of Jcalr BN1 middle}
J_{r}[\mathcal{B}_{N,1}]
=
(-1)^{N-r}
\frac{\alpha(1+\varrho)^{1-\frac{r}{2}}e^{-\frac{\alpha^2}{2}-\frac{\alpha^2\varrho}{4}}}{N^{\frac{3}{2}}}
\Re\Bigl[
\int_{\mathcal{B}_N}
\mathfrak{h}_{\mathcal{B}}(t)
e^{i\Psi_{N-2,0}(t)}
\Bigl(1+O(\frac{1}{N})\Bigr)
dt
\Bigr]
=
O(\frac{1}{N^{\frac{5}{2}}}),
\end{equation}
where 
\[
    \mathfrak{h}_{\mathcal{B}}(t)
    :=
\frac{e^{\frac{\alpha^2(t^2+\varrho^2)}{8t}}}{t^{\frac{1}{2}}(t-\lambda_-)^{\frac{1}{4}}(\lambda_+-t)^{\frac{1}{4}}}.
\]
This shows the estimate for $J_{r}[\mathcal{B}_{N,1}]$ in \eqref{def of Jcal BN1r asymp}.

Next, let $t\in\mathcal{B}_{N,2}$. Then by Lemma~\ref{Lem Critical regime} and \eqref{def of tilde weight asymp}, we have
\begin{align*}
    \widetilde{w}_{N,r}(t)L_{N-r}^{(\nu)}\Bigl(\frac{N}{\tau}t\Bigr)   
&=
\frac{(-1)^{N-r}\alpha\sqrt{\pi}}{2N^{\frac{4}{3}}}
\frac{(1+\varrho)^{\frac{19}{24}-\frac{r}{2}}}{(\sqrt{1+\varrho}+1)^{\frac{1}{3}}}
\frac{e^{-\frac{\alpha^2}{2}+\frac{\alpha^2}{4}(t-\varrho)-\frac{\alpha^2(t^2-\varrho^2)}{8t}}}{t^{\frac{1}{2}}}
(t-\lambda_-)^{\frac{1}{4}}
\\
&\quad
\times
\Ai\Bigl( 
N^{\frac{2}{3}} \frac{(1+\varrho)^{\frac{1}{6}}}{(\sqrt{1+\varrho}+1)^{\frac{4}{3}}}
(t-\lambda_+)
\Bigr)
\Bigl( 1 + O(\frac{1}{N}) \Bigr).
\end{align*}
Let 
$$ C_{\alpha,\varrho}:= 
 \alpha
 \sqrt{\frac{\pi}{2}}
\frac{(1+\varrho)^{\frac{11}{12}-\frac{r}{2}}}{(\sqrt{1+\varrho}+1)^{\frac{4}{3}}}. $$
Then we have 
\begin{align*}
    J_{r}[\mathcal{B}_{N,2}]
    &=
    C_{\alpha,\varrho}
    \frac{(-1)^{N-r}}{N^2}
    \int_{-N^{\epsilon}}^{N^{\epsilon}}
\Ai\Bigl( 
\frac{(1+\varrho)^{\frac{1}{6}}}{(\sqrt{1+\varrho}+1)^{\frac{4}{3}}}
s
\Bigr)\,ds
\cdot\Bigl(1+O(\frac{1}{N})\Bigr)
\\
&=
C_{\alpha,\varrho}
\frac{(\sqrt{1+\varrho}+1)^{\frac{4}{3}}}{(1+\varrho)^{\frac{1}{6}}}
\frac{(-1)^{N-r}}{N^2}
\Bigl(1+O(\frac{1}{N})\Bigr)
=
\alpha
\sqrt{\frac{\pi}{2}}
(1+\varrho)^{\frac{3}{4}-\frac{r}{2}}
\frac{(-1)^{N-r}}{N^2}
\Bigl(1+O(\frac{1}{N})\Bigr),
\end{align*}
where we have used the fact that $\int_{\R}\Ai(ax)\,dx=a^{-1}$ for $a>0$.
This shows the estimate for $  J_{r}[\mathcal{B}_{N,2}]$ in \eqref{def of Jcal BN2r asymp} since $N\in\N$ is even.

Finally, let $t\in\mathcal{B}_{N,3}$. Then by \eqref{Laguerre tau fix exponential} in Lemma~\ref{lem Exponential regime}, we have  
\[
L_{N-r}^{(\nu)}\Bigl( \frac{N}{\tau}t \Bigr) 
=
\frac{(-1)^{N-r}}{\sqrt{2\pi N}}
e^{\frac{\alpha^2}{2}}
(\varrho+1)^{-\frac{r}{2}+\frac{1}{4}}
\frac{\sqrt{\psi'_{\varrho}(z)}}{\psi_{\varrho}(t)^{r}}
e^{Ng_\tau^{(\varrho)}(t)}
\Bigr|_{\tau=1-\frac{\alpha^2}{2N}}
\Bigl(1+O(\frac{1}{N}) \Bigr),
\]
which gives 
\begin{equation}
\label{def of JcalQ3r}
J_{r}[\mathcal{B}_{N,3}]  \leq  \frac{C}{N^{\frac{3}{2}}}
\int_{\lambda_+ + N^{-\frac{2}{3}+\epsilon}}^{\infty}
\frac{\sqrt{\psi'_{\varrho}(t)}}{\psi_{\varrho}(t)^{r}}
e^{-\frac{N}{2}\Omega_{\varrho}(t)}
\,dt. 
\end{equation}
Note that by Taylor expansion, there exists $c_{\varrho}>0$ such that as $t\to \lambda_+$ as the right limit, we have  
\[
\Omega_{\varrho}(t)=c_{\varrho}(t-\lambda_+)^{\frac{3}{2}}\bigl(1+O(t-\lambda_+)\bigr). 
\]
Therefore, combining the above with \eqref{def of JcalQ3r}, we obtain the estimate for $J_{r}[\mathcal{B}_{N,3}]$ in \eqref{def of Jcal BN1r asymp}. 
\end{proof}

\begin{rem}
In the proof of Lemma~\ref{lem prepare L^1 estimate}, in addition to estimating $\mathbf{R}_{N,2}^{(2)}$ in $\mathcal{B}_N$, we also established the existence of a constant $C > 0$ such that for any $x \in \mathbb{R}$,
\begin{equation}
\label{def of bfRN22 bound}
    \bigl|\bfR_{N,2}^{(2)}(x)\bigr|\leq \frac{C}{N^2}.
\end{equation}
This result will be used in the proof of Lemma~\ref{Lem L1 estimate at weak} in the next subsection. 
\end{rem}

\subsection{\texorpdfstring{$L^1$}{L1}-estimate for linear statistics; Proof of Lemma~\ref{Lem L1 estimate at weak}}
\label{Subsection L1 estimate at weak non-Hermitian regime}

In this subsection, we prove Lemma~\ref{Lem L1 estimate at weak}.  
The proof follows the same strategy as that of Lemma~\ref{lem prepare L^1 estimate}.

\begin{proof}[Proof of Lemma~\ref{Lem L1 estimate at weak}] 
Let \( 0 < \epsilon < \frac{2}{3} \) be small, and let \( \delta > 0 \) be a small constant.
We write intervals by 
\begin{align*}
    \mathcal{Q}_1&:=\bigl(\lambda_+ - N^{-\frac{2}{3}+\epsilon},\lambda_+\bigr], \qquad
    \mathcal{Q}_2:=\bigl(\lambda_+,\lambda_+ + \delta\bigr), \qquad
\mathcal{Q}_3:=\bigl[\lambda_+ + \delta ,\infty\bigr), \\
    \mathcal{Q}_4&:=\bigl[\lambda_-,\lambda_- +N^{-\frac{2}{3}+\epsilon}\bigr), 
    \qquad
    \mathcal{Q}_5:=\bigr(\lambda_--\delta,\lambda_-\bigr), \qquad
    \mathcal{Q}_6:=\bigl(-\infty,\lambda_--\delta\bigr], 
\end{align*}
where $\lambda_{\pm}$ are given by \eqref{MP law}, and define 
\begin{equation}
I_{\lambda_+}:=I_{\lambda_+}^1+I_{\lambda_+}^2+I_{\lambda_+}^3,\qquad I_{\lambda_-}:=I_{\lambda_-}^1+I_{\lambda_-}^2+I_{\lambda_-}^3,
\end{equation}
where for $k=1,2,3$, 
\begin{equation}  \label{I lambda + weak}
 I_{\lambda_+}^k :=\int_{\mathcal{Q}_k} |f(x)|\mathbf{R}_N(x)\,dx, \qquad  I_{\lambda_-}^k :=\int_{\mathcal{Q}_{k+3}} |f(x)|\mathbf{R}_N(x)\,dx. 
\end{equation}
Then we have 
\begin{equation}
    \mathbb{E}\Bigl[\bigl|\Xi_N^{\mathcal{B}_N^{\rm c}}(f)\bigr|\Bigr]
    \leq I_{\lambda_+} + I_{\lambda_-}. 
\end{equation}
Since \( I_{\lambda_-} \) can be estimated in a similar manner to \( I_{\lambda_+} \), we focus on estimating \( I_{\lambda_+} \). 
We show that given $0<\epsilon<\frac{2}{3}$ sufficiently small, there exists $c>0$ and such that
\begin{equation}
I_{\lambda_+}^1=O( N^{\frac{1}{3}+\epsilon} ), \qquad I_{\lambda_+}^2=O( 1 ), \qquad I_{\lambda_+}^3=O( e^{-cN} ),
\end{equation}
which gives $I_{ \lambda_+ }=O(N^{\frac{1}{3}+\epsilon})$.
Similarly, $I_{ \lambda_- }=O(N^{\frac{1}{3}+\epsilon})$.

Recall from Lemma~\ref{Lem_RN RN hat RN2 decomp} that $\bfR_N=\widehat{\bfR}_{N,1}+\bfR_{N,2}$. We first estimate the integrals involving the term $\bfR_{N,2}.$  
 
By \eqref{def of bfRN22 bound}, and as in the derivation of \eqref{def of Jcal BN1r asymp}, it follows from \eqref{asymp of oscillatory integral} that  
\begin{align}
\begin{split}
\label{def of bfRN2 calBN bound}
\int_{\mathcal{B}_N}|f(x)|\bfR_{N,2}(x)\,dx
&=
O(N^2)\times
\int_{\mathcal{B}_N}
e^{\gamma x}
\mathbf{R}_{N,2}^{(1)}(x)\,dx
\\
&=
O(N^{\frac{1}{2}}) \times 
\Re\Bigl[
\int_{\mathcal{B}_N}
e^{\gamma x}\mathfrak{h}_{\mathcal{B}}(x)
e^{i\Psi_{N-1,0}(x)}
\Bigl(1+O(\frac{1}{N})\Bigr)
\,dx
\Bigr]
=
O(\frac{1}{\sqrt{N}}).
\end{split}
\end{align}
By \eqref{def of Critical Laguerre 1} and \eqref{def of bfRN22 bound}, as $N\to\infty$, we have 
\begin{equation}
\label{def of fint calQ_1 bfRN1}
\int_{\mathcal{Q}_1}|f(x)|\bfR_{N,2}(x)\,dx
=
O(1)\times
    \int_{-N^{\epsilon}}^{0}
\Ai\Bigl( 
\frac{(1+\varrho)^{\frac{1}{6}}}{(\sqrt{1+\varrho}+1)^{\frac{4}{3}}}
s
\Bigr)\,ds
=O(1),
\end{equation}
where we have used the fact that $\int_{-\infty}^{0}\Ai(s)\,ds=\frac{2}{3}$. Similarly, we have
\begin{equation}
\label{def of fint calQ_2 bfRN2}
\int_{\mathcal{Q}_2}|f(x)|\bfR_{N,2}(x)\,dx
=
O(N^{\frac{2}{3}})\times
    \int_{0}^{\delta}
\Ai\Bigl(
N^{\frac{2}{3}}\frac{(1+\varrho)^{\frac{1}{6}}}{(\sqrt{1+\varrho}+1)^{\frac{4}{3}}}
s
\Bigr)\,ds
=
O(1),
\end{equation}
where we have used $\int_{0}^{\infty}\Ai(s)\,ds=\frac{1}{3}$. 
On the other hand, by \eqref{def of JcalQ3r}, we have 
\[
\int_{\mathcal{Q}_3}|f(x)|\bfR_{N,2}(x)\,dx
\leq
CNe^{-cN^{\frac{3}{2}\epsilon}}
\int_{\lambda_+ + N^{-2/3+\epsilon}}^{\infty}
e^{\gamma x}
\frac{\sqrt{\psi'_{\varrho}(t)}}{\psi_{\varrho}(t)^{r}}
e^{-\frac{N}{2}\Omega_{\varrho}(t)}
\,dt. 
\]
By the discussion of \eqref{def of tangent omega varrho}, we see that for $c_{\delta}>0$,
\begin{equation}
\label{def of fint calQ_3 bfRN2}
\int_{\mathcal{Q}_3}|f(x)|\bfR_{N,2}(x)\,dx
\leq 
CN e^{-\frac{(N-1)}{2}\varphi(\lambda_+ +\delta)}\int_{\lambda_+ + \delta}^{\infty}
e^{-\frac{\varphi(x)}{2}+\gamma x}\,dx
=
O(e^{-c_{\delta}N}).
\end{equation}
Here, $\gamma>0$ in \eqref{def of test function at weak} should be chosen so that $\frac{\mu}{2}-\gamma>0$. 
Hence, by \eqref{def of fint calQ_1 bfRN1}, \eqref{def of fint calQ_2 bfRN2}, and \eqref{def of fint calQ_3 bfRN2}, we obtain 
\begin{equation}
\label{def of bfRN2 complement of calBN}
\int_{\mathcal{B}_N^{\mathsf{c}}}|f(x)|\bfR_{N,2}(x)\,dx = O(1).   
\end{equation}

Next, we estimate the integrals involving the term $\widehat{\bfR}_{N,1}$. Note that by \eqref{bfRN one point density}, we have 
\begin{equation}
\label{def of bfRN1 rough estimate}
\widehat{\bfR}_{N,1}(x)
=
\sqrt{\frac{1-\tau^2}{\pi N}}
\Bigl[
\Bigl( 
x^2+\frac{\varrho^2(1-\tau^2)^2}{4}
\Bigr)^{\frac{1}{4}}\bfR_N^{\,\rm c}(x)
+
\frac{(4N)^{-1}(1-\tau^2)}{(x^2+\frac{\varrho^2(1-\tau^2)^2}{4})^{\frac{1}{4}}}
\partial_x\bfR_N^{\,\rm c}(x)
\Bigr]
\Bigl(1+O(\frac{1}{N})\Bigr).
\end{equation}
Then by \ref{Lem_relation bfRN mathfrak RN}, it suffices to determine the asymptotics of $\mathfrak{R}_{N}^{\mathrm{c}}$, to which we can apply Lemma~\ref{Lem_ODE for mathfrak RN}.

Recall that by \eqref{calIN1 asymp}, we have  
\[
    \mathcal{I}_{N,1}(t)
    =
    N^3\sqrt{\frac{\pi}{2}}
\frac{(1+\varrho)^{\frac{3}{2}}e^{-\alpha^2}}{\sqrt{2}\alpha(1+\varrho)^{(1+\varrho)N}}
\frac{1}{\sqrt{t}}
t^{N\varrho}
e^{-N(t-\varrho)-\frac{\alpha^2(t^2-\varrho^2)}{4t}}
 \Bigl( 1 + O\Bigl(\frac{1}{N}\Bigr) \Bigr), 
\]
uniformly for $t>0.$
For $t\in\mathcal{Q}_1 \cup \mathcal{Q}_2$, by \eqref{calIN2x} and Lemma~\ref{Lem Critical regime}~(i), we have 
\begin{align*}
\mathcal{I}_{N,2}(t)
&=
\frac{(-1)^{2N-3}}{2N^{\frac{2}{3}}}
\frac{(1+\varrho)^{(1+\varrho)N-\frac{3}{2}+\frac{1}{12}}}{(\sqrt{1+\varrho}+1)^{1/3}t^{N\varrho}}
(t-\lambda_-)^{1/2}
\exp\Bigl( 
N(t-\varrho)
+
\frac{\alpha^2}{2}(t-\varrho)
\Bigr)
\\
&\quad
\times
\Ai^2\Bigl( 
N^{\frac{2}{3}} \frac{(1+\varrho)^{\frac{1}{6}}}{(\sqrt{1+\varrho}+1)^{\frac{4}{3}}}
(t-\lambda_+)
\Bigr)
\Bigl( 1 + O(t-\lambda_+) \Bigr). 
\end{align*}

Let $x\in\mathcal{Q}_1$. Using the boundary condition as Lemma~\ref{lem bfRN1 leading term at weak} and combining the above with \eqref{boldSN ODE global diagonal} and \eqref{diagonal IN c natural}, as $N\to\infty$, we have 
\begin{align}
\begin{split}
\label{def of frakRnc integral 1}
\mathfrak{R}_N^{\rm c}(x)
&=\int_{\lambda_-}^x
\mathcal{I}_{N,1}(t)\cdot\mathcal{I}_{N,2}(t)\,dt 
\,(1+o(1))
\\
&=
(-1)^{2N-3}N^{\frac{7}{3}}
\frac{\sqrt{\pi}e^{-\alpha^2}(1+\varrho)^{\frac{1}{12}}}{4\alpha(\sqrt{1+\varrho}+1)^{\frac{1}{3}}}e^{-\frac{\alpha^2\varrho}{2}}
\\
&\quad\times
\int_{\lambda_-}^x
\frac{(t-\lambda_-)^{\frac{1}{2}}
e^{ \frac{\alpha^2}{4}t+\frac{\alpha^2\varrho^2}{4t}}}{\sqrt{t}}
\Ai^2\Bigl(\frac{N^{\frac{2}{3}}(1+\varrho)^{\frac{1}{6}}}{(\sqrt{1+\varrho}+1)^{\frac{4}{3}}}
(t-\lambda_+)
\Bigr)
\,dt \, (1+o(1)).
\end{split}
\end{align}
Then by using \eqref{approx of RN c and mathfrak RN c}, \eqref{def of bfRN1 rough estimate}, and \eqref{def of frakRnc integral 1}, for $x\in\mathcal{Q}_1$, we obtain
\begin{align}
\begin{split}
\label{def of calQ1 bound fx bfhatRN1}
    \int_{\mathcal{Q}_1}|f(x)|\widehat{\bfR}_{N,1}(x)\,dx
&\leq
CN^{\frac{4}{3}}\int_{\mathcal{Q}_1}
\frac{e^{\gamma x}}{\sqrt{x}}
\int_{\lambda_-}^x
\frac{(t-\lambda_-)^{\frac{1}{2}}
e^{ \frac{\alpha^2}{4}t+\frac{\alpha^2\varrho^2}{4t}}}{\sqrt{t}}
\Ai^2\Bigl(\frac{N^{\frac{2}{3}}(1+\varrho)^{\frac{1}{6}}}{(\sqrt{1+\varrho}+1)^{\frac{4}{3}}}
(t-\lambda_+)
\Bigr)
\,dt
\\
&\leq
CN
\int_{\mathcal{Q}_1}
\int_{\lambda_-}^x
\sqrt{\frac{t-\lambda_-}{t(\lambda_+-t)}}
e^{ \frac{\alpha^2}{4}t+\frac{\alpha^2\varrho^2}{4t}}
\,dt
\\
&
\leq CN\int_{\mathcal{Q}_1}\bigl(\sqrt{\lambda_+-\lambda_-}-\sqrt{\lambda_+-x}\bigr)\,dx=O(N^{\frac{1}{3}+\epsilon}),
\end{split}
\end{align}
where we have used the asymptotics of the Airy function 
$$
\Ai(-r) =\frac{r^{-\frac{1}{4}}}{\pi^{\frac{1}{2}}}\Bigl( 
\cos\bigl(\tfrac{2}{3}r^{\frac{3}{2}}-\tfrac{\pi}{4} \bigr)  + O(r^{-\frac{3}{2}}) \Bigr), \qquad r \to \infty. 
$$ 

Next, let $x\in\mathcal{Q}_2$. By combining the above with \eqref{boldSN ODE global diagonal} and \eqref{diagonal IN c natural}, as $N\to\infty$, we have
\begin{align}
\begin{split}
\label{def of frakRnc integral 2}
\mathfrak{R}_N^{\rm c}(x)
&=\int_{+\infty}^x \mathcal{I}_{N,1}(t)\cdot\mathcal{I}_{N,2}(t)\,dt  \, (1+o(1))
\\
&= (-1)^{2N-3}N^{\frac{7}{3}} \frac{\sqrt{\pi}e^{-\alpha^2}(1+\varrho)^{\frac{1}{12}}}{4\alpha(\sqrt{1+\varrho}+1)^{\frac{1}{3}}}e^{-\frac{\alpha^2\varrho}{2}}
\\
&\quad\times
\int_{\infty}^x
\frac{(t-\lambda_-)^{\frac{1}{2}}
e^{ \frac{\alpha^2}{4}t+\frac{\alpha^2\varrho^2}{4t}}}{\sqrt{t}}
\Ai^2\Bigl(\frac{N^{\frac{2}{3}}(1+\varrho)^{\frac{1}{6}}}{(\sqrt{1+\varrho}+1)^{\frac{4}{3}}}
(t-\lambda_+)
\Bigr)
\,dt  \, (1+o(1)),
\end{split}
\end{align}
where we have used the boundary condition at $x=+\infty$.
In particular, to apply the dominated convergence theorem, we have used the asymptotics of the Airy function 
$$
\Ai (r) =\frac{r^{-\frac{1}{4}}e^{-\frac{2}{3}r^{\frac{3}{2}}}}{2\sqrt{\pi}}
\bigl(1+O(r^{-\frac{3}{2}}) \bigr), \qquad r \to \infty.
$$  
By \eqref{approx of RN c and mathfrak RN c}, \eqref{def of frakRnc integral 2}, and \eqref{def of bfRN1 rough estimate}, for $x\in\mathcal{Q}_2$, we obtain 
\begin{align}
\begin{split}
\label{def of calQ2 bound fx bfhatRN1}
\int_{\mathcal{Q}_2}|f(x)|\widehat{\bfR}_{N,1}(x)\,dx
&\leq CN^{\frac{4}{3}}\int_{\mathcal{Q}_2}
\frac{e^{\gamma x}}{\sqrt{x}}
\int_{\infty}^x
\frac{(t-\lambda_-)^{1/2}
e^{ \frac{\alpha^2}{4}t+\frac{\alpha^2\varrho^2}{4t}}}{\sqrt{t}}
\Ai^2\Bigl(\frac{N^{\frac{2}{3}}(1+\varrho)^{\frac{1}{6}}(t-\lambda_+)}{(\sqrt{1+\varrho}+1)^{\frac{4}{3}}}
\Bigr)
\,dt\,dx \\
&\leq
C
\int_{0}^{\infty}
\Bigl[
\int_{\infty}^{x}
\Ai^2\Bigl(\frac{(1+\varrho)^{\frac{1}{6}}s}{(\sqrt{1+\varrho}+1)^{\frac{4}{3}}}
\Bigr)\,ds 
\Bigr]dx=O(1), 
\end{split}
\end{align}
where $C>0$ may change in each line, but it does not depend on $N$.

Finally, let $x \in \mathcal{Q}_3$. Applying the procedure used to derive \eqref{def of fint calQ_3 bfRN2}, we obtain  
\begin{align*}
\bfR_{N}^{\rm c}(x)
&\leq 
CN^{\frac{3}{2}}\int_{\infty}^{x}
\frac{e^{-N\Omega_{\varrho}(t)}\,dt}{\sqrt{(t-(2+\varrho))^2-4(1+\varrho)}(t-(2+\varrho)+\sqrt{(t-(2+\varrho))^2-4(1+\varrho)})^2}, 
\end{align*}
where $c'=(\lambda_+ + \delta)\mu-\widetilde{\Omega}_{\varrho}(\lambda_+ + \delta)$; cf. \eqref{def of tangent omega varrho}. Therefore it follows that  
\begin{equation}
\label{def of calQ3 bound fx bfhatRN1}
\int_{\mathcal{Q}_3}|f(x)|\widehat{\bfR}_{N,1}(x)\,dx
\leq
C\sqrt{N}e^{-(N-1)\varphi(\lambda_+ + \delta)}
\int_{\lambda_+ + \delta}^{\infty}
e^{-\varphi(x)+\gamma x}\,dx
\leq 
C\sqrt{N}e^{-c_{\delta}'N},
\end{equation}
where $c_{\delta}'>0$ and $C>0$ may change in each line. 
Here, if $\mu-2\gamma>0$ hold then $\mu>\gamma$ also hold. 

By combining \eqref{def of calQ1 bound fx bfhatRN1}, \eqref{def of calQ2 bound fx bfhatRN1}, \eqref{def of calQ3 bound fx bfhatRN1}, and \eqref{def of bfRN2 complement of calBN}, we obtain the desired result \eqref{def of L1 estimate at weak}.
\end{proof}

\appendix

\section{Asymptotic analysis for fixed parameter \texorpdfstring{$\nu$}{nu}} \label{Appendix_fixed}

This appendix provides an exposition of our main results when the parameter \( \nu \) is fixed. While the overall proof follows the same structure as in the previous sections, additional details are required near the singular origin.  
The macroscopic behaviour of the one-point densities (Propositions~\ref{Prop_asymptotic of 1pt function RN strong} and~\ref{prop of one density weak macro}) remains unchanged from the large parameter case. Therefore, we focus on the \( L^1 \) estimates of the exterior linear statistics (Lemmas~\ref{lem L1 estimate} and~\ref{Lem L1 estimate at weak}).

\subsection{Proof of Lemma~\ref{lem L1 estimate}: fixed parameter case}
\label{subsection fixed parameter at strong}

As in the large parameter case, given $0<\epsilon<\frac{1}{2}$ sufficiently small, we consider the subset of $[\xi_-,\xi_+]$:
\begin{equation}
\label{cJN}
    \mathcal{F}_N:=F_{N,(+)}\cup F_{N,(-)}, 
\end{equation}
where
\begin{align*} 
F_{N,(+)}  := \bigl(  N^{-1+\epsilon}, \xi_+ - N^{-\frac{1}{2}+\epsilon} \bigr), \qquad 
F_{N,(-)} := \bigl(\xi_- + N^{-\frac{1}{2}+\epsilon}, N^{-1+\epsilon}\bigr). 
\end{align*} 
For a fixed $\nu>-1$ and a test function $f$, we define 
\begin{equation}
\label{linear statistic fixed}
    \Xi_{N}^{(\rm fixed)}(f)
    :=     \Xi_{N}^{\mathcal{F}_N}(f) + \Xi_{N}^{\mathcal{F}_N^{\rm c}}(f). 
\end{equation}

We first verify that for any fixed $\tau \in (0,1)$ and $\nu \ge 0$, there exist $\epsilon>0$ and $c>0$ such that 
\begin{equation}
\label{def of bfRN2 fix bulk}
    \bfR_{N,2}(x)=O(e^{-cN^{2\epsilon}})
\end{equation}
uniformly for $x\in\mathcal{F}_{N}$.

\begin{proof}[Proof of \eqref{def of bfRN2 fix bulk}]
Note that 
\[
\bfR_{N,2}^{(2)}(x)
=
\frac{1}{2}\bigl(J_{2}[[x,\infty)]-J_{2}[(-\infty,x)]\bigr).
\]
If $x>0$ then for $M>0$, we decompose $J_{2}[(-\infty,x)]$ into 
\[
J_{2}[(-\infty,x)]
=
J_{2}[(-\infty,-MN^{-1}]]+J_{2}[(-MN^{-1},MN^{-1})]+J_{2}[[MN^{-1},x]].
\]
If $x<0$ the we similarly decompose $J_{2}[[x,\infty)]$ into 
\[
J_{2}[[x,\infty)]
=J_{2}[[x,-MN^{-1}]]+J_{2}[(-MN^{-1},MN^{-1})]+J_{2}[[MN^{-1},\infty)]. 
\]
We firstly estimate $J_{2}[(-MN^{-1},MN^{-1})]$. 
\begin{align*}
    J_{2}[(-MN^{-1},MN^{-1})]
    &=\int_{-MN^{-1}}^{MN^{-1}}
     \sqrt{\frac{N^{\nu-1}(N-1)!}{\Gamma(N-1+\nu)}}
\tau^{N-2}
|t|^{\frac{\nu}{2}}
K_{\frac{\nu}{2}}\Bigl(\frac{N|t|}{1-\tau^2}\Bigr)
e^{\frac{N\tau}{1-\tau^2}t}L_{N-2}^{(\nu)}\Bigl(\frac{N}{\tau}t\Bigr)\,dt
\\
&=
\frac{\tau^{N+\frac{\nu}{2}-1}}{N^{\frac{\nu}{2}+1}}
\int_{-\frac{M}{\tau}}^{\frac{M}{\tau}}
     \sqrt{\frac{N^{\nu-1}(N-1)!}{\Gamma(N-1+\nu)}}
|t|^{\frac{\nu}{2}}
K_{\frac{\nu}{2}}\Bigl(\frac{\tau|t|}{1-\tau^2}\Bigr)
e^{\frac{\tau^2t}{1-\tau^2}}L_{N-2}^{(\nu)}(t)\,dt. 
\end{align*}
By \eqref{def of J bessel asymptotic 2}, for a sufficiently large $N$, we have 
\begin{align*}
&\quad\frac{\tau^{N+\frac{\nu}{2}-1}}{N^{\frac{\nu}{2}+1}}
     \sqrt{\frac{N^{\nu-1}(N-1)!}{\Gamma(N-1+\nu)}}
|t|^{\frac{\nu}{2}}
K_{\frac{\nu}{2}}\Bigl(\frac{\tau|t|}{1-\tau^2}\Bigr)
e^{\frac{\tau^2t}{1-\tau^2}}L_{N-2}^{(\nu)}(t) 
\\
&\leq 
\tau^{\frac{\nu}{2}-1}
|t|^{\frac{\nu}{2}}
K_{\frac{\nu}{2}}\Bigl(\frac{\tau|t|}{1-\tau^2}\Bigr)
e^{\frac{\tau^2t}{1-\tau^2}}
\frac{N^{\frac{\nu}{2}-1}}{\Gamma(\nu+1)}e^{-N\log\frac{1}{\tau}+2\sqrt{N}|\im\sqrt{t}|+\frac{t}{2}}=O(e^{-c N}),
\end{align*}
over $t$ in a compact subset of $\R$.
Here, we have used  
\begin{equation}
\label{def of bound J bessel}
    |J_{\nu}(z)|\leq \frac{|z|^{\nu}e^{|\im z|}}{2^{\nu}\Gamma(\nu+1)},\qquad (\nu\geq-\frac{1}{2}). 
\end{equation} 
Therefore, for $0<\tau<1$, we have 
\begin{equation}
\label{def of origin estimate fixed}
  J_{2}[(-MN^{-1},MN^{-1})]=  O(e^{-c N}). 
\end{equation}
The estimates for the remaining parts follow the same approach as in the proof of Lemma~\ref{lem L1 estimate}, with modifications incorporating the asymptotics of the Laguerre polynomials for fixed \( \nu \). In particular, one can show that there exists $C>0$ such that  
\begin{equation}
\label{def of bfRN2(2) fixed estimate}    
 \bfR_{N,2}^{(2)}(x)\leq \frac{C}{N\sqrt{N}}
\Bigl(  1+O(\frac{1}{N^{\frac{1}{2}-\epsilon}}) \Bigr). 
\end{equation} 
Then we obtain \eqref{def of bfRN2 fix bulk} in a similar manner to \eqref{RN2 front term}.
\end{proof}

\begin{lem}[\textbf{$L^1$-estimate for a fixed parameter $\nu\geq0$ case at strong non-Hermiticity}]
\label{lem L1 estimate fixed}
Fix $\tau\in(0,1)$ and $\nu\geq0$.
Assume that $\chi(f)<\infty$, where $\chi(f)$ is given by \eqref{f integrable bound}. 
Then the complementary statistic $\Xi_{N}^{\mathcal{F}_N^\mathrm{c}}(f)$ satisfies the $L^1$-estimate
\begin{equation}
\label{expectation global estimate fixed}
\mathbb{E}\bigl[|\Xi_{N}^{\mathcal{F}_N^{\rm c}}(f)|\bigr]
    =O(N^{\epsilon}) 
\end{equation}
for some $\epsilon>0.$
\end{lem} 
\begin{proof}%[Proof of Lemma~\ref{lem L1 estimate fixed}]
The only difference from Lemma~\ref{lem L1 estimate} is the treatment of the singularity at the origin.  
Thus, we focus on estimating its contribution.

For a sufficiently large \( M > 0 \) independent of \( N \), we define
\begin{align}
\label{I1 origin}
  I_{1,1}^{(\rm origin)}
 &:=
 \int_{0}^{MN^{-1}}
|f(x)|\mathbf{R}_{N}(x) \, dx,
 \qquad
 I_{1,2}^{(\rm origin)}
 :=
 \int_{MN^{-1}}^{N^{-1+\epsilon}}
|f(x)|\mathbf{R}_{N}(x) \, dx,
\\
\label{I2 origin}
 I_{2,1}^{(\rm origin)}
 &:=   \int_{-MN^{-1}}^{0} |f(x)|\mathbf{R}_{N}(x)\, dx,
 \qquad
 I_{2,2}^{(\rm origin)}
 :=
 \int_{-N^{-1+\epsilon}}^{-MN^{-1}}
|f(x)|\mathbf{R}_{N}(x)\, dx,
\end{align}
and 
\begin{equation}
\label{I origin}
 I_1^{(\rm origin)}
 :=
 I_{1,1}^{(\rm origin)}
 +
 I_{1,2}^{(\rm origin)},
\qquad
 I_2^{(\rm origin)}
 :=
 I_{2,1}^{(\rm origin)}
 +
 I_{2,2}^{(\rm origin)}. 
\end{equation}
Then we have 
\begin{equation}
\label{I bound fix}
\mathbb{E}\bigl[\bigl|\Xi_{N}^{\mathcal{F}_N^{\rm c}}(f)\bigr|\bigr]
\leq I_1 + I_2 + I_1^{(\rm origin)} + I_2^{(\rm origin)}, 
\end{equation}
where $I_1,I_2$ are defined by \eqref{expectation linear statistics bound plus 123}. 
 
Note that \( I_1 \) and \( I_2 \) in \eqref{I bound fix} can be bounded similarly to the proof of Lemma~\ref{lem L1 estimate}.  
Thus, it suffices to estimate \eqref{I1 origin}, as \eqref{I2 origin} can be handled in the same way based on the proof below.
Note that by Lemma~\ref{Lem_RN RN hat RN2 decomp}, we have 
\[
I_{1,1}^{(\rm origin)}  = I_{1,1,(1)}^{(\rm origin)}+I_{1,1,(2)}^{(\rm origin)}, \qquad I_{1,2}^{(\rm origin)} = I_{1,2,(1)}^{(\rm origin)}+I_{1,2,(2)}^{(\rm origin)},
\]
where 
\[
I_{1,1,(1)}^{(\rm origin)}:=\int_0^{MN^{-1}}|f(x)|\widehat{\bfR}_{N,1}(x)\,dx,\qquad
I_{1,1,(2)}^{(\rm origin)}:=\int_0^{MN^{-1}}|f(x)|\bfR_{N,2}(x)\,dx ,
\]
and
\[
I_{1,2,(1)}^{(\rm origin)}:=\int_{MN^{-1}}^{N^{-1+\epsilon}}|f(x)|\widehat{\bfR}_{N,1}(x)\,dx,\qquad
I_{1,2,(2)}^{(\rm origin)}:=\int_{MN^{-1}}^{N^{-1+\epsilon}}|f(x)|\bfR_{N,2}(x)\,dx.
\]

By the change of variable $x\mapsto x/N$, we have 
\[
\frac{1}{N}\widehat{\bfR}_{N,1}(N^{-1}x)=
S_{N,1}(x,x)+\widetilde{S}_{N,1}(x,x),
\]
where $S_{N,1}$ and $\widetilde{S}_{N,1}$ are given by \eqref{SN 1 xy nonrescaled} and  \eqref{SN 1 xy nonrescaled tilde}, respectively.
By \eqref{def of mathcalK} and Hardy-Hille formula \cite[Eq. (18.18.27)]{NIST}, 
as $N\to\infty$, we have  
\begin{align*}
&\quad \mathfrak{s}_{\mathrm{s},1}(x)
:= \lim_{N\to\infty}
\frac{1}{N}
\widehat{\bfR}_{N,1}(N^{-1}x)
\\
&=
\frac{1}{\pi(1-\tau^2)^2}
\Bigl(\frac{|xy|}{xy}\Bigr)^{\frac{\nu}{2}}
K_{\frac{\nu}{2}}\Bigl(\frac{|x|}{1-\tau^2}\Bigr) 
\Bigl[
\sqrt{xy}
I_{\nu+1}\Bigl(\frac{2\sqrt{xy}}{1-\tau^2}\Bigr)
K_{\frac{\nu}{2}}\Bigl(\frac{|y|}{1-\tau^2}\Bigr)
+
|y|
I_{\nu}\Bigl(\frac{2\sqrt{xy}}{1-\tau^2}\Bigr)
K_{\frac{\nu}{2}+1}\Bigl(\frac{|y|}{1-\tau^2}\Bigr)
\Bigr],
\end{align*}
where the convergence is uniformly for $x,y$ in  compact subsets of $\R$. 
Therefore, as $N \to \infty,$  
\begin{equation}
\label{I11(1) origin estimate}
I_{1,1,(1)}^{(\rm origin)}
\leq\chi(f)e^{\gamma M}
\int_0^{M}
\frac{1}{N}
\widehat{\mathbf{R}}_{N,1}(N^{-1}x)\,dx
\leq
C\int_0^{M}
\mathfrak{s}_{\mathrm{s},1}(x)\,dx
<\infty, 
\end{equation}
for some $C>0$. 
On the other hand, for \( I_{1,1,(2)}^{(\rm origin)} \), using \eqref{def of bfRN2(2) fixed estimate} and similar computations as in \eqref{def of origin estimate fixed}, we obtain  
\begin{equation}
\label{I11(2) origin estimate}
 I_{1,1,(2)}^{(\rm origin)}=O(e^{-cN}),   
\end{equation}
for some $c>0$. 
Therefore, by \eqref{I11(1) origin estimate}, \eqref{I11(2) origin estimate} and by straightforward computation based on \eqref{def of origin estimate fixed}, as $N\to\infty$
\begin{equation}
 \label{I11 origin estimate}   
 I_{1,1}^{(\rm origin)}= O(1),\qquad
 I_{2,1}^{(\rm origin)}= O(1). 
\end{equation}
Finally, we shall estimate $I_{2,1}^{(\rm origin)}$. 
By applying Lemma~\ref{Lem_bold SN dens} for $\varrho=0$, we have 
\begin{equation}
\label{def of I12(1) estimate}
I_{1,2,(1)}^{(\rm origin)}
\leq C 
\int_{M^{-\frac{1}{2}}N^{-\frac{1}{2}}}^{N^{-\frac{1}{2}+\frac{\epsilon}{2}}}
x\widehat{\bfR}_{N,1}(x^2)\,dx
\leq
 CN^{\frac{\epsilon}{2}},
\end{equation}
for some positive constant $C>0$. 
On the other hand, by Lemma~\ref{lem Exponential regime}~(ii),~\ref{Lem BDLaguerre large order} with $\varrho=0$, and~\ref{Lem Omega Positivity} with $\varrho=0$, we have 
\begin{equation}
I_{1,2,(2)}^{(\rm origin)}    
\leq CN\int_{M^{-1}N^{-1}}^{N^{-1+\epsilon}}
e^{-N\Omega(x)}\,dx
=O(N^{\epsilon}e^{-cN}),
\end{equation}
where we have used $\min_{x\in[M^{-1}N^{-1},N^{-1+\epsilon}]}\Omega(x)\geq c>0$, which does not depend on $N$, by Taylor expansion.
This completes the proof. 
\end{proof}

\subsection{Proof of Lemma~\ref{Lem L1 estimate at weak}: fixed parameter case}
\label{subsection fixed parameter at weak}

We now consider the weakly non-Hermitian regime $\tau=1-\frac{\alpha^2}{2N}$ for $\alpha>0$.  
 
\begin{lem}
\label{lem bfRN22 estimate at weak fixed} 
Assume $\tau=1-\frac{\alpha^2}{2N}$ for $\alpha>0$ and let $\nu\geq0$ be fixed. 
Then the estimate \eqref{def of bfRN2 estimate at weak} holds. 
\end{lem}
\begin{proof}
Let $0<\epsilon<\frac{2}{3}$ be sufficiently small, and for a large $M>0$, we define intervals by 
\begin{align*}
\mathcal{U}_{N,1}&:=\bigl[N^{-1},4-N^{-\frac{2}{3}+\epsilon}\bigr],\qquad
\mathcal{U}_{N,2}:=\bigl(4-N^{-\frac{2}{3}+\epsilon},4+N^{-\frac{2}{3}+\epsilon}\bigr),\\
\mathcal{U}_{N,3}&:=\bigl[4+N^{-\frac{2}{3}+\epsilon},\infty\bigr),\qquad
\mathcal{U}_{N,4}=\bigl(N^{-1},N^{-1}\bigr),\qquad
\mathcal{U}_{N,5}:=\bigl(-\infty,-N^{-1}\bigr].
\end{align*}
As in Lemma~\ref{lem prepare L^1 estimate}, it suffices to show that as $N\to\infty$, 
\[
\bfR_{N,2}^{(1)}(x)=O(\frac{1}{N^{ \frac{3}{2} }}),\qquad
\bfR_{N,2}^{(2)}(x)=O(\frac{1}{N^{ \frac{5}{2} } }),
\]
uniformly for $x\in\mathcal{U}_{N,1}$. 
The estimate for \( \bfR_{N,2}^{(1)} \) in \eqref{asymp of RN2 1 2 weak} follows immediately, similar to Lemma~\ref{lem prepare L^1 estimate}.  
Thus, we focus on estimating \( \bfR_{N,2}^{(2)}(x) \) for all \( x \in \mathbb{R} \).  
To this end, following the proof of Lemma~\ref{lem prepare L^1 estimate}, and noting that \( J_{2}[\mathbb{R}] = \sum_{k=1}^{5} J_{2}[\mathcal{U}_{N,k}] \), it suffices to show that as \( N \to \infty \),  
\begin{gather}
\label{def of Jcal BN1r asymp fixed}
J_2[\mathcal{U}_{N,1}] =O(\frac{1}{N^{ \frac{5}{2} } } ), \qquad 
J_2[\mathcal{U}_{N,3}]=O(e^{-cN^{\frac{3}{2}\epsilon}}),\qquad
J_2[\mathcal{U}_{N,5}]=O(e^{-c'N}),
\\
J_{2}[\mathcal{U}_{N,2}] = \alpha
\sqrt{\frac{\pi}{2}} \frac{1}{N^2}\Bigl(1+O(\frac{1}{N})\Bigr),
\qquad  J_{2}[\mathcal{U}_{N,4}] =
\alpha
\sqrt{\frac{\pi}{2}}
\frac{1}{N^2}\Bigl(1+O(\frac{1}{N})\Bigr),  \label{def of JcalU4 estimate}
\end{gather} 
for some $c,c'>0$. 
These estimates for \( \mathcal{U}_{N,k} \) with \( k = 1,2,3,5 \) follow in a similar manner as in Lemma~\ref{lem prepare L^1 estimate}.
On the other hand, for $k=4$, it follows from Lemma~\ref{lem Laguerre bessel} that  
\begin{align*}
J_{r}[\mathcal{U}_{N,4}]
&=
\int_{-MN^{-1}}^{MN^{-1}}
\sqrt{\frac{N^{\nu-1}(N-1)!}{\Gamma(N-1+\nu)}}
\tau^{N-r}
|t|^{\frac{\nu}{2}}
K_{\frac{\nu}{2}}\Bigl(\frac{N|t|}{1-\tau^2}\Bigr)
e^{\frac{N\tau}{1-\tau^2}t}L_{N-r}^{(\nu)}\Bigl(\frac{N}{\tau}t\Bigr)\,dt
\\
&=
\frac{\tau^{N-r+1+\frac{\nu}{2}}}{N^{\frac{\nu}{2}+1}}
\int_{-\tau^{-1}M}^{\tau^{-1}M}
\sqrt{\frac{N^{\nu-1}(N-1)!}{\Gamma(N-1+\nu)}}
|t|^{\frac{\nu}{2}}
K_{\frac{\nu}{2}}\Bigl(\frac{\tau|t|}{1-\tau^2}\Bigr)
e^{\frac{\tau^2}{1-\tau^2}t}L_{N-r}^{(\nu)}(t)\,dt
\\
&=  \frac{\alpha^2e^{-\frac{\alpha^2}{2}}}{N^2}
\int_{-\alpha^{-2}MN}^{\alpha^{-2}MN}
\frac{|t|^{\frac{\nu}{2}}}{t^{\frac{\nu}{2}}}
K_{\frac{\nu}{2}}(|t|)e^{t}
J_{\nu}(2\alpha\sqrt{t})
\,dt\cdot\Bigl(1+O(\frac{1}{N})\Bigr)
\\
&=
\frac{\alpha^2e^{-\frac{\alpha^2}{2}}}{N^2}
\Bigl( 
\int_{0}^{\infty}
K_{\frac{\nu}{2}}(t)e^{t}
J_{\nu}(2\alpha\sqrt{t})
\,dt
+
(-1)^{\frac{\nu}{2}}
\int_0^{\infty}
K_{\frac{\nu}{2}}(t)e^{-t}
J_{\nu}(2\alpha\sqrt{-t})
\,dt
\Bigr)
\Bigl(1+O(\frac{1}{N})\Bigr).
\end{align*}
Then the desired estimate follows from 
\begin{align*}
\int_{0}^{\infty}
K_{\frac{\nu}{2}}(t)e^{t}
J_{\nu}(2\alpha\sqrt{t})
\,dt
&=\frac{e^{\frac{\alpha^2}{2}}}{\alpha}\sqrt{\frac{\pi}{2}}
\frac{\Gamma(\frac{\nu+1}{2},\frac{\alpha^2}{2})}{\Gamma(\frac{\nu+1}{2})},
\\
(-1)^{\frac{\nu}{2}}\int_{0}^{\infty}
K_{\frac{\nu}{2}}(t)e^{-t}
J_{\nu}(2\alpha\sqrt{-t})
\,dt
&=
\frac{e^{\frac{\alpha^2}{2}}}{\alpha}\sqrt{\frac{\pi}{2}}
\frac{\Gamma(\frac{\nu+1}{2})-\Gamma(\frac{\nu+1}{2},\frac{\alpha^2}{2})}{\Gamma(\frac{\nu+1}{2})}. 
\end{align*} 
This completes the proof. 
\end{proof}

Recall that for a fixed $\nu\geq0$, the support of the spectral droplet of real eigenvalues becomes $[0,4]$ at the weakly non-Hermitian regime.
Thus, given a small $0<\epsilon<\frac{2}{3}$, we consider the following subset of $[0,4]$: 
\begin{equation}
\mathcal{G}_N:=\bigl(N^{-1+\epsilon},4-N^{-\frac{2}{3}+\epsilon}\bigr).    
\end{equation}
For a fixed $\nu\geq0$ and a test function $f$, we redefine 
\begin{equation}
\Xi_N^{(\text{fixed})}(f):=\Xi_N^{\mathcal{G}_N}(f)+\Xi_N^{\mathcal{G}_N^{\mathsf{c}}}(f).     
\end{equation}
Then as in Lemma~\ref{Lem L1 estimate at weak}, we show the following $L^1$-estimate. 
\begin{lem}
Assume $\tau=1-\frac{\alpha^2}{2N}$ for $\alpha>0$ and $\nu\geq0$, and let $0<\epsilon<\frac{2}{3}$ be sufficiently small. 
Let $f$ be a locally integrable and measurable function satisfying \eqref{def of test function at weak}.
Then we have
\begin{equation}
\mathbb{E}\Bigl[|\Xi_N^{\mathcal{G}_N^{\mathsf{c}}}(f)|\Bigr]
 =
 O(N^{\frac{\epsilon+1}{2}}). 
\end{equation}
\end{lem}

\begin{proof}
We provide only the estimate for the origin.  
The other cases are handled in the same way as in Lemma~\ref{Lem L1 estimate at weak} for the large parameter case.
We write 
\begin{equation}
I_{1}^{(\text{origin},\mathrm{w})}:=I_{1,1}^{(\text{origin},\mathrm{w})}+I_{1,2}^{(\text{origin},\mathrm{w})},  \qquad
I_2^{(\text{origin},\mathrm{w})}:=I_{2,1}^{(\text{origin},\mathrm{w})}+I_{2,2}^{(\text{origin},\mathrm{w})},
\end{equation}
where 
\begin{align}
I_{1,1}^{(\text{origin},\mathrm{w})}
&:=
\int_{0}^{MN^{-1}}|f(x)|\bfR_N(x)\,dx,\qquad I_{1,2}^{(\text{origin},\mathrm{w})}
:=
\int_{MN^{-1}}^{N^{-1+\epsilon}}|f(x)|\bfR_N(x)\,dx, \\
I_{2,1}^{(\text{origin},\mathrm{w})}
&:=
\int_{-MN^{-1}}^{0}|f(x)|\bfR_N(x)\,dx,\qquad I_{2,2}^{(\text{origin},\mathrm{w})}
:=
\int_{-\infty}^{-MN^{-1}}|f(x)|\bfR_N(x)\,dx, 
\end{align}
for $M>0$ large. 
Then, using the same notations as in the proof of Lemma~\ref{Lem L1 estimate at weak}, we have 
\begin{equation}
\mathbb{E}\Bigl[\bigl|\Xi_N^{\mathcal{G}_N^{\mathsf{c}}}(f)\bigr|\Bigr]
\leq I_{\lambda_+}\bigr|_{\varrho=0} + I_{1}^{(\text{origin},\mathrm{w})} + I_{2}^{(\text{origin},\mathrm{w})}. 
\end{equation}
%As we already mentioned, the estimate for $I_{\lambda_+}\bigr|_{\varrho=0}=O(N^{\frac{1}{3}+\epsilon})$ follows from the proof Lemma~\ref{Lem L1 estimate at weak}. Hence, we omit the details. 
We provide bounds for \( I_{1,1}^{(\text{origin},\mathrm{w})} \) and \( I_{1,2}^{(\text{origin},\mathrm{w})} \). The other cases are handled similarly: \( I_{2,1}^{(\text{origin},\mathrm{w})} \) follows the same approach as \( I_{1,1}^{(\text{origin},\mathrm{w})} \), while \( I_{2,2}^{(\text{origin},\mathrm{w})} \) is treated as in Lemma~\ref{lem bfRN22 estimate at weak fixed}.

Note that 
\begin{align*}
I_{1,1}^{(\mathrm{origin,w})}
%\int_{0}^{MN^{-2}}|f(x)|\bfR_N(x)\,dx
%=
%\int_{0}^{MN^{-2}}|f(x)|\widehat{\bfR}_{N,1}(x)\,dx
%+
%\int_{0}^{MN^{-2}}|f(x)|\bfR_{N,2}(x)\,dx
%\\
\leq
C\chi(f)
\Bigl(
\int_{0}^{M\alpha^{-2}}
\alpha^2
\widehat{R}_{N,1}(\alpha^2x)\,dx
+
\int_{0}^{MN^{-1}}
\bfR_{N,2}(x)\,dx
\Bigr)
\end{align*}
for some $C>0$. Here $\widehat{R}_{N,1}(x):=R_{N,1}(x)+\widetilde{R}_{N,1}(x)$, where $R_{N,1}(x)$ and $\widetilde{R}_{N,1}(x)$ are given by \eqref{def of normal RN1 and RN2}. 
It was shown in \cite{Os04} and \cite[Subsection 6.2]{AB23a} that
\begin{align}
\begin{split}
\frac{\alpha^{2\nu+2}}{N^{\nu+1}}\cK_{N-1}\Bigl(\frac{\alpha^2}{N}x,\frac{\alpha^2}{N}x\Bigr)   &= 
4\,\mathcal{J}_{\nu}(x)
(1+o(1)),
\qquad
 \mathcal{J}_{\nu}(x):=
x^{-\nu}\int_0^{\alpha}
e^{-s^2}
J_{\nu}(2s\sqrt{x})^2
\,ds,
\end{split}    
\end{align}
where the convergence is uniform for $x$ in a compact subset of $\R$.
By \eqref{def of RN 1pt real}, as $N\to\infty$, we have 
\begin{align*}
\frac{\alpha^2}{N}\widehat{R}_{N,1}\Bigl(\frac{\alpha^2}{N}x\Bigr)
&=
\frac{4|x|^{\nu}e^{2x} }{\pi}
\Bigl(   K_{\nu/2}(|x|)
\bigl[ xK_{\nu/2}(|x|) +|x| K_{\nu/2+1} (|x|) \bigr] 
 \mathcal{J}_{\nu}(x)
 +
 \frac{x}{2} K_{\nu/2}(|x|)^2
\partial_x\mathcal{J}_{\nu}(x)
\Bigr)(1+o(1)).
\end{align*}
This gives 
\begin{align}
\begin{split}
    \label{def of hatRN1 0 to M hard weak}
\int_{0}^{NM\alpha^{-2}}
\widehat{R}_{N,1}(\alpha^2x)\,dx
&\sim
\alpha^2
\int_{0}^{NM\alpha^{-2}}
\frac{4}{\pi} 
x^{\nu+1}  K_{\nu/2}(x) e^{2x}
\bigl[ K_{\nu/2}(x) + K_{\nu/2+1} (x) \bigr] 
 \mathcal{J}_{\nu}(x)\,dx
 \\
 &\quad
 +
 \alpha^2
\int_{0}^{NM\alpha^{-2}}
 \frac{2}{\pi}x^{\nu+1}  K_{\nu/2}(x)K_{\nu/2}(x)
e^{2x} 
\partial_x\mathcal{J}_{\nu}(x)\,dx
<\infty,
\end{split}
\end{align}
where we have used the Fubini's theorem, and the asymptotics of the Bessel functions. 
For the remainder term, by Lemma~\ref{lem bfRN22 estimate at weak fixed}, as $N\to\infty$, we have
\begin{equation}
      \label{def of bfRN2 0 to M hard weak}
\int_0^{MN^{-1}}\bfR_{N,2}(x)\,dx
=
O(N^2)\cdot\int_{0}^{MN^{-1}}\bfR_{N,2}^{(1)}(x)\,dx
=
O(N^2)\cdot J_{[0,MN^{-1},1]}
=O(1),
\end{equation}
where we have used \eqref{def of JcalU4 estimate}. 
Therefore, by \eqref{def of hatRN1 0 to M hard weak} and \eqref{def of bfRN2 0 to M hard weak}, as $N\to\infty$, we have 
\begin{equation}
\label{def of I11 weak hard estimate}
    I_{1,1}^{(\rm origin,w)}=O(1).
\end{equation}
For $I_{1,2}^{(\rm origin,w)}$, similar to $I_{1,1}^{(\rm origin,w)}$, note that 
\[
I_{1,2}^{(\rm origin,w)}
\leq C\chi(f)
\Bigl( 
\int_{MN^{-1}}^{N^{-1+\epsilon}}\widehat{\bfR}_{N,1}(x)\,dx
+
\int_{MN^{-1}}^{N^{-1+\epsilon}}\bfR_{N,2}(x)\,dx
\Bigr),
\]
for some $C>0$.
By applying Lemma~\ref{lem bfRN1 leading term at weak} for \( \varrho = 0 \) (which remains valid and can be verified by repeating the proof of Lemma~\ref{lem bfRN1 leading term at weak} for fixed \( \nu \geq 0 \)) and Lemma~\ref{lem bfRN22 estimate at weak fixed}, we obtain 
\begin{equation}
\label{def of I12 weak near hard estimate}
\int_{MN^{-1}}^{N^{-1+\epsilon}}\bfR_{N}(x)\,dx
\sim\frac{N}{2\alpha \sqrt{\pi}}
\int_{MN^{-1}}^{N^{-1+\epsilon}}
\frac{1}{\sqrt{x}}\erf\Bigl(\frac{\alpha}{2}\sqrt{4-x}\Bigr)\,dx
=O(N^{\frac{\epsilon+1}{2}}),
\end{equation}
as $N \to \infty$.
By choosing a sufficiently small \( 0 < \epsilon < \frac{2}{3} \) such that \( \frac{1}{3} + \epsilon \leq \frac{1+\epsilon}{2} \), we obtain the desired result. This completes the proof.
\end{proof}

\subsection*{Acknowledgements} Sung-Soo Byun was supported by the New Faculty Startup Fund at Seoul National University and by the National Research Foundation of Korea grant (NRF-2016K2A9A2A13003815, RS-2023-00301976, RS-2025-00516909).
Kohei Noda was supported by JSPS KAKENHI Grant Numbers JP22H05105, JP23H01077 and JP23K25774, and also supported in part JP21H04432. 
We would like to express our gratitude to Peter J. Forrester for his valuable feedback on the draft of this paper.

\small 

\bibliographystyle{abbrv}

\begin{thebibliography}{100}

\setlength{\itemsep}{1.5pt}

\bibitem{AFNV00} M. Adler, P. J. Forrester, T. Nagao and P. van Moerbeke, \emph{Classical skew orthogonal polynomials and random matrices}, J. Stat. Phys. \textbf{99} (2000), 141--170. 

\bibitem{AV01} M. Adler and P. van Moerbeke, \emph{Hermitian, symmetric and symplectic random ensembles: PDEs for the distribution of the spectrum}, Ann. Math. \textbf{153} (2001), 149--189.

\bibitem{Ak05} G. Akemann, \emph{The complex Laguerre symplectic ensemble of non-Hermitian matrices}, Nuclear Phys. B \textbf{730} (2005), 253--299.
 
\bibitem{AB07} G. Akemann and F. Basile, \emph{Massive partition functions and complex eigenvalue correlations in matrix models with symplectic symmetry}, Nuclear Phys. B \textbf{766} (2007), 150--177.

\bibitem{AB10} G. Akemann and M. Bender, \emph{Interpolation between Airy and Poisson statistics for unitary chiral non-Hermitian random matrix ensembles}, J. Math. Phys. \textbf{51} (2010), 103524. 

\bibitem{AB23} G. Akemann and S.-S. Byun, \emph{The product of $m$ real $N \times N$ Ginibre matrices: real eigenvalues in the critical regime $m=O(N)$}, Constr. Approx. \textbf{59} (2024), 31--59. 

\bibitem{ABES23}  G. Akemann, S.-S. Byun, M. Ebke and G. Schehr, \emph{Universality in the number variance and counting statistics of the real and symplectic Ginibre ensemble}, J. Phys. A \textbf{56} (2023), 495202. 

\bibitem{ABK21} G. Akemann, S.-S. Byun and N.-G. Kang, \emph{A non-Hermitian generalisation of the Marchenko--Pastur distribution: from the circular law to multi-criticality}, Ann. Henri Poincare \textbf{22} (2021), 1035--1068.

%\bibitem{ABK22} G. Akemann, S.-S. Byun and N.-G. Kang, \emph{Scaling limits of planar symplectic ensembles}, SIGMA Symmetry Integrability Geom. Methods Appl. \textbf{18} (2022), Paper No. 007, 40pp.

%\bibitem{ACV18} G. Akemann, M. Cikovic and M. Venker, \emph{Universality at weak and strong non-Hermiticity beyond the elliptic Ginibre ensemble}, Comm. Math. Phys. \textbf{362} (2018), 1111--1141.

%\bibitem{ADM23} G. Akemann, M.~Duits and L. D.~Molag, \emph{The elliptic Ginibre ensemble: a unifying approach to local and global statistics for higher dimensions}, J. Math. Phys. \textbf{64} (2023), 023503. 

\bibitem{ADM24} G. Akemann, M.~Duits and L. D.~Molag, \emph{Fluctuations in various regimes of non-Hermiticity and a holographic principle}, arXiv:2412.15854. 

\bibitem{AEP22} G. Akemann, M. Ebke and I. Parra, \emph{Skew-orthogonal polynomials in the complex plane and their Bergman-like kernels}, Comm. Math. Phys. \textbf{389} (2022), 621--659.

%\bibitem{AK07} G.~Akemann and E.~Kanzieper, \emph{Integrable structure of Ginibre's ensemble of real random matrices and a {P}faffian integration theorem}, J. Stat. Phys. \textbf{129} (2007), 1159--1231. 

%\bibitem{AKMP19} G. Akemann, M. Kieburg, A. Mielke and T. Prosen, \emph{Universal signature from integrability to chaos in dissipative open quantum systems}, Phys. Rev. Lett. \textbf{123} (2019), 254101.

\bibitem{AKP10} G. Akemann, M. Kieburg and M. J. Phillips, \emph{Skew-orthogonal Laguerre polynomials for chiral real asymmetric random matrices}, J. Phys. A \textbf{43} (2010), 375207.

\bibitem{AP14} G. Akemann and M. J. Phillips, \emph{The interpolating Airy kernels for the $\beta=1$ and $\beta=4$ elliptic Ginibre ensembles}, J. Stat. Phys. \textbf{155} (2014), 421--465.

%\bibitem{AP14a} G. Akemann and M. J. Phillips, \emph{Universality conjecture for all Airy, sine and Bessel kernels in the complex plane}, in Random Matrix Theory, Interacting Particle Systems, and Integrable Systems, Mathematical Sciences Research Institute Publications, vol. 65, pp. 1\UTF{2013}23. Cambridge University Press, New York (2014). 

%\bibitem{APS09} G. Akemann, M. J. Phillips and H.-J. Sommers, \emph{Characteristic polynomials in real Ginibre ensembles}, J. Phys. A \textbf{42} (2009), 012001. 

\bibitem{APS10} G. Akemann, M. J. Phillips and H.-J. Sommers, \emph{The chiral Gaussian two-matrix ensemble of real asymmetric matrices}, J. Phys. A \textbf{43} (2010), 085211. 

%\bibitem{AK22} J. Alt and T. Kr\"{u}ger, \emph{Local elliptic law}, Bernoulli \textbf{28} (2022), no. 2, 886--909.

%\bibitem{Am13} Y. Ameur, \emph{Near-boundary asymptotics of correlation kernels}, J. Geom. Anal. \textbf{23} (2013), 73--95.

%\bibitem{Am21} Y. Ameur, \emph{A localization theorem for the planar Coulomb gas in an external field}, Electron. J. Probab. \textbf{26} (2021), 1--21.

\bibitem{AB23a} Y. Ameur and S.-S. Byun, \emph{Almost-Hermitian random matrices and bandlimited point processes}, Anal. Math. Phys. \textbf{13} (2023), 52.

\bibitem{ACCL24} Y. Ameur, C. Charlier, J. Cronvall and J. Lenells, \emph{Disk counting statistics near hard edges of random normal matrices: the multi-component regime}, Adv. Math. \textbf{441} (2024), 109549.

%\bibitem{ACC23a} Y. Ameur, C. Charlier, and J. Cronvall, \emph{The two-dimensional Coulomb gas: fluctuations through a spectral gap}, arXiv:2210.13959.


%\bibitem{AC23} Y. Ameur and J. Cronvall, \emph{Szeg\"o type asymptotics for the reproducing kernel in spaces of full-plane weighted polynomials}, Comm. Math. Phys. \textbf{398} (2023), 1291--1348,



%\bibitem{AHM11} Y. Ameur, H. Hedenmalm and M. Makarov, \emph{Fluctuations of eigenvalues of random normal matrices}, Duke Math. J. \textbf{159} (2011), 31--81. 


%\bibitem{AHM15} Y. Ameur, H. Hedenmalm and N. Makarov, \emph{Random normal matrices and Ward identities}, Ann. Probab. \textbf{43} (2015), 1157--1201.

%\bibitem{AKM19} Y. Ameur, N.-G. Kang and N. Makarov, \emph{Rescaling Ward identities in the random normal matrix model}, Constr. Approx. \textbf{50} (2019), 63--127.

%\bibitem{AKMW20} Y. Ameur, N.-G. Kang, N. Makarov and A. Wennman, \emph{Scaling limits of random normalmatrix processes at singular boundary points}, J. Funct. Anal. \textbf{278} (2020), 108340. 


%\bibitem{AT24} Y. Ameur and E. Troedsson, \emph{Remarks on the one-point density of Hele-Shaw $\beta$-ensembles}, arXiv:2402.13882.

%\bibitem{BEDPMW13} C. W. J. Beenakker, J. M. Edge, J. P. Dahlhaus, D. I. Pikulin, S. Mi and M. Wimmer, \emph{Wigner-Poisson statistics of topological transitions in a Josephson junction}, Phys. Rev. Lett. \textbf{111} (2013), 037001.

\bibitem{BFK21} G. Ben Arous, Y. V. Fyodorov and B. A. Khoruzhenko. \emph{Counting equilibria of large complex systems by instability index}, Proc. Natl. Acad. Sci. USA \textbf{118} (2021), e2023719118.

\bibitem{Be10} M. Bender, \emph{Edge scaling limits for a family of non-Hermitian random matrix ensembles}, Probab. Theory Relat. Fields \textbf{147} (2010), 241--271.

\bibitem{BBD23} M. Bhattacharjee, A. Bose and A. Dey, \emph{Joint convergence of sample cross-covariance matrices}, ALEA, Lat. Am. J. Probab. Math. Stat. \textbf{20} (2023), 395--423.

%\bibitem{BOR16} F. Bornemann, \emph{A note on the expansion of the smallest eigenvalue distribution of the LUE at the hard edge}, Ann. Appl. Probab. \textbf{26} (2016), 1942--1946.

\bibitem{BS09} A.~Borodin and C. D. Sinclair, \emph{The Ginibre ensemble of real random matrices and its scaling limit}, Comm. Math. Phys. \textbf{291} (2009), 177--224.

\bibitem{By24} S.-S. Byun, \emph{Harer-Zagier type recursion formula for the elliptic GinOE}, Bull. Sci. Math. \textbf{197} (2024), 103526, 43pp.

%\bibitem{BC23} S.-S. Byun and C. Charlier, \emph{On the almost-circular symplectic induced Ginibre ensemble}, Stud. Appl. Math. \textbf{150} (2023), 184--217.

%\bibitem{BE23} S.-S. Byun and M. Ebke, \emph{Universal scaling limits of the symplectic elliptic Ginibre ensembles}, Random Matrices Theory Appl. 12 (2023), 2250047.

%\bibitem{BES23} S.-S. Byun, M. Ebke and S.-M. Seo, \emph{Wronskian structures of planar symplectic ensembles}, Nonlinearity \textbf{36} (2023), 809--844.

\bibitem{BF24} S.-S.~Byun and P. J.~Forrester, \emph{Progress on the study of the Ginibre ensembles}, KIAS Springer Ser. Math. \textbf{3} Springer, 2025, 221pp. 

\bibitem{BF24b} S.-S. Byun and P. J. Forrester, \emph{Spectral moments of the real Ginibre ensemble}, Ramanujan J. \textbf{64} (2024), 1497--1519.

\bibitem{BKLL23} S.-S. Byun, N.-G. Kang, J. O. Lee and J. Lee, \emph{Real eigenvalues of elliptic random matrices}, Int. Math. Res. Not. \textbf{2023} (2023), 2243--2280.

\bibitem{BL23} S.-S. Byun and Y.-W. Lee, \emph{Finite size corrections for real eigenvalues of the elliptic Ginibre matrices}, Random Matrices Theory Appl. \textbf{13} (2024), No. 01, 2450005. 

\bibitem{BMS23} S.-S.~Byun, L. D.~Molag and N.~Simm, \emph{Large deviations and fluctuations of real eigenvalues of elliptic random matrices}, Electron. J. Probab. (to appear), arXiv:2305.02753.

\bibitem{BN24} S.-S. Byun and K. Noda, \emph{Scaling limits of complex and symplectic non-Hermitian Wishart ensembles}, J. Approx. Theory \textbf{308} (2025), 106148.

%\bibitem{BN25b} S.-S. Byun and K. Noda, \emph{Real Eigenvalues of Asymmetric Wishart Matrices II: Finite size corrections and scaling limits}, preprint.  

\bibitem{BP24} S.-S. Byun and S. Park, \emph{Large gap probabilities of complex and symplectic spherical ensembles with point charges}, arXiv:2405.00386.

%\bibitem{BS23} S.-S. Byun and S.-M. Seo, \emph{Random normal matrices in the almost-circular regime}, Bernoulli \textbf{29} (2023), 1615--1637.

%\bibitem{BY23} S.-S. Byun and M. Yang, \emph{Determinantal Coulomb gas ensembles with a class of discrete rotational symmetric potentials}, SIAM J. Math. Anal. \textbf{55} (2023), 6867--6897.  

%\bibitem{CGZ14}  D. Chafai, N. Gozlan and P.-A. Zitt. \emph{First-order global asymptotics for confined particles with singular pair repulsion}, Ann. Appl. Probab. \textbf{24} (2014), 2371--2413.

\bibitem{CJQ20} S. Chang, T. Jiang and Y. Qi, \emph{Eigenvalues of large chiral non-Hermitian random matrices}, J. Math. Phys. \textbf{61} (2020), 013508. 

\bibitem{Ch22} C. Charlier, \emph{Asymptotics of determinants with a rotation-invariant weight and discontinuities along circles}, Adv. Math. \textbf{408} (2022), 108600.

%\bibitem{Ch23a} C. Charlier, \emph{Hole probabilities and balayage of measures for planar Coulomb gases}, arXiv:2311.15285.  


%\bibitem{CES21} G. Cipolloni, L. Erd\H{o}s and D. Schr\"oder, \emph{Edge universality for non-Hermitian random matrices}, Probab. Theory Related Fields \textbf{179} (2021), 1--28.

%\bibitem{CESX22} G. Cipolloni, L. Erd\H{o}s, D. Schr\"oder and Y. Xu, \emph{Directional extremal statistics for Ginibre eigenvalues}, J. Math. Phys. \textbf{63} (2022), 103303.

%\bibitem{CESX23} G. Cipolloni, L. Erd\H{o}s, D. Schr\"oder and Y. Xu, \emph{On the rightmost eigenvalue of non-Hermitian random matrices}, Ann. Probab. \textbf{51} (2023), 2192--2242.
 
\bibitem{CFW24} M. J. Crumpton, Y. V. Fyodorov and T. R. W\"{u}rfel, \emph{Mean eigenvector self-overlap in the real and complex elliptic Ginibre ensembles at strong and weak non-Hermiticity}, arXiv:2402.09296. 

\bibitem{DIW15} D. Dai, M. E. H. Ismail and J. Wang, \emph{Asymptotics for Laguerre polynomials with large order and parameters}, J. Approx. Theory \textbf{193} (2015), 4--19.

\bibitem{DLMS24} B. De Bruyne, P. Le Doussal, S. N. Majumdar and G. Schehr, \emph{Linear statistics for Coulomb gases: higher order cumulants}, J. Phys. A \textbf{57} (2024), 155002.  

\bibitem{DG07a} P. Deift and D. Gioev, \emph{Universality at the edge of the spectrum for unitary, orthogonal and symplectic ensembles of random matrices}, Commun. Pure Appl. Math. \textbf{60} (2007), 867--910.

%\bibitem{DG07b} P. Deift and D. Gioev, \emph{Universality in random matrix theory for orthogonal and symplectic ensembles}, Int. Math. Res. Pap. IMRP 2007, no. 2, Art. ID rpm004, 116 pp. 

\bibitem{DKMVZ99} P. Deift, T. Kriecherbauer, K.T.-R. McLaughlin, S. Venakides and X. Zhou, \emph{Strong asymptotics of orthogonal polynomials with respect to exponential weights}, Comm. Pure Appl. Math. \textbf{52} (1999), 1491--1552.

%\bibitem{BC12} F. Benaych-Georges and F. Chapon, \emph{Random right eigenvalues of Gaussian quaternionic matrices}, Random Matrices Theory Appl. \textbf{1} (2012), 1150009.

\bibitem{EKS94} A. Edelman, E. Kostlan and M. Shub, \emph{How many eigenvalues of a random matrix are real?} J. Amer. Math. Soc. \textbf{7} (1994), 247--267.

\bibitem{Efe97} K. B. Efetov, \emph{Directed quantum chaos}, Phys. Rev. Lett. \textbf{79} (1997), 491.

%\bibitem{Bateman2} A. Erd\'{e}lyi, W. Magnus, F. Oberhettinger and F. G. Tricomi, \emph{Higher transcendental functions}, Vol. II. Robert E. Krieger Publishing Co., Inc., Melbourne, Fla., 1981. Based on notes left by Harry Bateman, Reprint of the 1953 original. 

%\bibitem{FTZ22} W. FitzGerald, R. Tribe and O. Zaboronski,  \emph{Asymptotic expansions for a class of Fredholm Pfaffians and interacting particle systems}, Ann. Probab. \textbf{50} (2022), 2409--2474.

\bibitem{FS21} W. FitzGerald and N. Simm, \emph{Fluctuations and correlations for products of real asymmetric random matrices}, Ann. Inst. Henri Poincar\'{e} Probab. Stat. \textbf{59} (2023), 2308--2342.

\bibitem{Fo10} P. J. Forrester, \emph{Log-gases and random matrices}, Princeton University Press, Princeton, NJ, 2010.

\bibitem{Fo14} P. J. Forrester, \emph{Probability of all eigenvalues real for products of standard Gaussian matrices}, J. Phys. A \textbf{47} (2014), 065202. 

\bibitem{Fo24} P. J.~Forrester, \emph{Local central limit theorem for real eigenvalue fluctuations of elliptic GinOE matrices}, Electron. Commun. Probab. \textbf{29} (2024), 1--11.

%\bibitem{FFG06} P. J.~Forrester, N. E. Frankel and T. M. Garoni, \emph{Asymptotic form of the density profile for Gaussian and Laguerre random matrix ensembles with orthogonal and symplectic symmetry}, J. Math. Phys. \textbf{47} (2006), 023301. 

%\bibitem{FH99} P. J. Forrester and G. Honner, \emph{Exact statistical properties of the zeros of complex random polynomials}, J. Phys. A \textbf{32} (1999), 2961--2981.

%\bibitem{FJ96} P. J. Forrester and B. Jancovici, \emph{Two-dimensional one-component plasma in a quadrupolar field}, Int. J. Mod. Phys. A \textbf{11} (1996), 941--949.

\bibitem{FL20} P. J. Forrester and S.-H. Li, \emph{Classical discrete symplectic ensembles on the linear and exponential lattice: skew orthogonal polynomials and correlation functions}, Trans. Amer. Math. Soc. \textbf{373} (2020), 665--698. 

\bibitem{FM09} P. J. Forrester and A. Mays, \emph{A method to calculate correlation functions for $\beta = 1$ random matrices of odd size}, J. Stat. Phys. \textbf{134} (2009), 443--462.

\bibitem{FM12} P. J. Forrester and A. Mays, \emph{Pfaffian point processes for the Gaussian real generalised eigenvalue problem}, Prob. Theory and Rel. Fields \textbf{154} (2012), 1--47.

\bibitem{FN07} P. J. Forrester and T. Nagao, \emph{Eigenvalue statistics of the real Ginibre ensemble}, Phys. Rev. Lett. \textbf{99} (2007), 050603.

\bibitem{FN08} P. J. Forrester and T. Nagao, \emph{Skew orthogonal polynomials and the partly symmetric real Ginibre ensemble}, J. Phys. A \textbf{41} (2008), 375003.

\bibitem{FI16} P. J. Forrester and J. R. Ipsen, \emph{Real eigenvalue statistics for products of asymmetric real Gaussian matrices}, Linear Algebra Appl. \textbf{510} (2016), 259--290.

\bibitem{FIK20} P. J. Forrester, J. R. Ipsen and S. Kumar, \emph{How many eigenvalues of a product of truncated orthogonal matrices are real?}, Exp. Math. \textbf{29} (2020), 276--290.

%\bibitem{FT19} P. J. Forrester and A. Trinh, \emph{Finite-size corrections at the hard edge for the Laguerre $\beta$ ensemble},  Stud. Appl. Math. 143 (2019), no. 3, 315--336.

%\bibitem{FK16} Y. V. Fyodorov and B. A. Khoruzhenko, \emph{Nonlinear analogue of the May-Wigner instability transition}, Proc.  Nat. Acad. Science  \textbf{113} (2016), 6827--6832.

\bibitem{Fyo18} Y. V. Fyodorov,  \emph{On statistics of bi-orthogonal eigenvectors in real and complex Ginibre ensembles: combining partial Schur decomposition with supersymmetry}, Comm. Math. Phys. {\bf 363} (2018), 579--603.

\bibitem{FKS97} Y. V. Fyodorov, B. A. Khoruzhenko and H.-J. Sommers, \emph{Almost-{H}ermitian random matrices: crossover from Wigner-Dyson to Ginibre eigenvalue statistics}, Phys. Rev. Lett. \textbf{79} (1997), 557--560. 

\bibitem{FKS97a} Y. V. Fyodorov, B. A. Khoruzhenko and H.-J. Sommers, \emph{Almost-Hermitian random matrices: eigenvalue density in the complex plane}, Phys. Lett. A. \textbf{226} (1997), 46--52.

\bibitem{FKS98} Y. V. Fyodorov, B. A. Khoruzhenko and H.-J. Sommers, \emph{Universality in the random matrix spectra in the regime of weak non-Hermiticity},  Ann. Inst. H. Poincar\'e Phys. Th\'eor. \textbf{68} (1998), 449--489.

%\bibitem{FK99} Y. V. Fyodorov and B. A. Khoruzhenko, \emph{Systematic analytical approach to correlation functions of resonances in quantum chaotic scattering}, Phys. Rev. Lett. \textbf{83} (1999), 65. 

\bibitem{FKP24} Y. V. Fyodorov, B. A. Khoruzhenko and T. Prellberg, \emph{Zeros of conditional Gaussian analytic functions, random sub-unitary matrices and $q$-series}, arXiv:2412.06086. 

%\bibitem{FS03} Y .V. Fyodorov and H.-J. Sommers, \emph{Random matrices close to Hermitian or unitary: overview of methods and results}, J. Phys. A \textbf{36} (2003), 3303.

\bibitem{FT21} Y. V. Fyodorov and  W. Tarnowski, \emph{Condition numbers for real eigenvalues in the real elliptic Gaussian ensemble}, Ann. Henri Poincar\'e \textbf{22} (2021), 309--330.

%\bibitem{Gin65} J. Ginibre, \emph{Statistical ensembles of complex, quaternion, and real matrices}, J. Math. Phys. \textbf{6} (1965), 440--449.

\bibitem{GPX24} A. Goel, P. Lopatto and X. Xie, \emph{Central limit theorem for the complex eigenvalues of Gaussian random matrices}, Electron. Commun. Probab. \textbf{29} (2024), Paper No. 16, 13 pp.

%\bibitem{GR14} I. S. Gradshteyn and I. M. Ryzhik. \emph{Table of integrals, series, and products}, Academic press, 2014.

%\bibitem{HM13} H. Hedenmalm and N. Makarov, \emph{Coulomb gas ensembles and Laplacian growth}, Proc. Lond. Math. Soc. \textbf{106} (2013), 859--907.


%\bibitem{HHJK23} B. C. Hall, C.-W. Ho, J. Jalowy and Z. Kabluchko, \emph{Zeros of random polynomials undergoing the heat flow}, arXiv:2308.11685. 


%\bibitem{HW21} H. Hedenmalm and A. Wennman, \emph{Planar orthogonal polynomials and boundary universality in the random normal matrix model}, Acta Math. \textbf{227} (2021), 309--406. 

%\bibitem{HO18} D. Huybrechs and P. Opsomer, \emph{Construction and implementation of asymptotic expansions for Laguerre--type orthogonal polynomials}, IMA J. Numer. Anal. \textbf{38} (2018), 1085--1118.

%\bibitem{Jo98} K. Johansson, \emph{On fluctuations of eigenvalues of random Hermitian matrices}, Duke Math. J. \textbf{91} (1998), 151--204.

%\bibitem{Ka23} A. Kafetzopoulos, \emph{Spectral properties of random matrices}, Ph.D. Thesis, Queen Mary University of London, UK, 2023.

%\bibitem{Kan02} E. Kanzieper, \emph{Eigenvalue correlations in non-Hermitean symplectic random matrices}, J. Phys. A \textbf{35} (2002), 6631--6644.

\bibitem{KS10} E. Kanzieper and N. Singh, \emph{Non-Hermitean Wishart random matrices (I)}, J. Math. Phys. \textbf{51} (2010), 103510. 

\bibitem{KPTTZ16} E.~Kanzieper, M.~Poplavskyi, C.~Timm, R.~Tribe and O.~Zaboronski, \emph{What is the probability that a large random matrix has no real eigenvalues?}, Ann.~Appl.~Probab. \textbf{26} (2016), 2733--2753.

%\bibitem{KL21} B. A. Khoruzhenko and S. Lysychkin, \emph{Truncations of random symplectic unitary matrices}, arXiv:2111.02381.


%\bibitem{Ku11} A. B. J. Kuijlaars, \emph{Universality}, Chapter 6 in The Oxford Handbook of Random Matrix Theory. In: G. Akemann, J. Baik and P. Di Francesco, (eds.) Oxford University Press (2011).

\bibitem{LR16} S.-Y. Lee and R. Riser, \emph{Fine asymptotic behavior for eigenvalues of random normal matrices: Ellipse case}, J. Math. Phys. \textbf{57} (2016), 023302.

\bibitem{LSYF22} S.-H. Li, B.-J. Shen, G.-F. Yu and P. J. Forrester, \emph{Discrete orthogonal ensemble on the exponential lattices}, Adv. Appl. Math. \textbf{164} (2025), 102836. 

\bibitem{LMS22} A. Little, F. Mezzadri and N. Simm, \emph{On the number of real eigenvalues of a product of truncated orthogonal random matrices}, Electron. J. Probab. \textbf{27} (2022), 1--32.


\bibitem{LHG24} Y. Liu, J. Hu and J. Gore, \emph{Ecosystem stability relies on diversity difference between trophic levels}, Proc. Natl. Acad. Sci. USA \textbf{121} (2024), e2416740121.

\bibitem{LHLG25}  Y. Liu, J. Hu, H. Lee and J. Gore, \emph{Complex ecosystems lose stability when resource consumption is out of Niche}, Phys. Rev. X  \textbf{15} (2025), 011003. 

\bibitem{MP67} V. A. Marcenko and L. A. Pastur, \emph{Distribution of eigenvalues in certain sets of random matrices}, (Russian) Mat. Sb. (N.S.) \textbf{72} (114)(1967), 507--536.

%\bibitem{Mo23} L. D. Molag, \emph{Edge behavior of higher complex-dimensional determinantal point processes}, Ann. Henri Poincar\'{e} \textbf{24} (2023), 4405--4437. 

%\bibitem{O97} F. W. J. Olver, \emph{Asymptotics and special functions},  AKP Classics. A K Peters, Ltd., Wellesley, MA, 1997. Reprint of the 1974 original [Academic Press, New York; MR0435697 (55 \#8655)].

\bibitem{NIST}  F. W. J. Olver, D. W. Lozier, R. F. Boisvert, and C. W. Clark, eds. \emph{NIST Handbook of Mathematical Functions}, Cambridge: Cambridge University Press, 2010.

%\bibitem{OYZ23} S. O'Rourke, Z. Yin and P. Zhong, \emph{Spectrum of Laplacian matrices associated with large random elliptic matrices}, arXiv:2308.16171.

%\bibitem{OQR17} J. Ortmann, J. Quastel, and D. Remenik, \emph{A Pfaffian representation for flat ASEP}, Comm. Pure Appl. Math. \textbf{70} (2017), 3--89.

\bibitem{Os04} J. C. Osborn, \emph{Universal results from an alternate random matrix model for QCD with a baryon chemical potential}, Phys. Rev. Lett. \textbf{93} (2004), 222001. 

%\bibitem{Os23} M. Osman, \emph{Universality for weakly non-Hermitian matrices: bulk limit}, arXiv:2310.15001. 

\bibitem{Phillips10} M. J. Phillips, \emph{A random matrix model for two-colour QCD at non-zero quark density}, Ph.D. Thesis, Brunel University, UK, 2010.

\bibitem{Se24} S. Serfaty, \emph{Lectures on Coulomb and Riesz Gases}, arXiv:2407.21194.  

\bibitem{Si17} N. Simm, \emph{Central limit theorems for the real eigenvalues of large Gaussian random matrices}, Random Matrices Theory Appl. \textbf{6} (2017), 1750002. 

\bibitem{Si17a}  N. Simm, \emph{On the real spectrum of a product of Gaussian matrices}, Electron. Commun. Probab. \textbf{22} (2017), 11.

\bibitem{Si07} C. D. Sinclair, \emph{Averages over {G}inibre's ensemble of random real matrices}, Int. Math. Res. Not. \textbf{2007} (2007), rnm015, 15 pp.

%\bibitem{So07} H.-J. Sommers, \emph{Symplectic structure of the real Ginibre ensemble}, J. Phys. A \textbf{40} (2007), F671. 

%\bibitem{Ste96} M. A. Stephanov, \emph{Random matrix model of QCD at finite density and the nature of the quenched limit}, Phys. Rev. Lett. \textbf{76} (1996), 4472.

\bibitem{Szego39} G. Szeg\H{o}, \emph{Orthogonal Polynomials}, American Mathematical Society, New York, 1939.

\bibitem{TV15}  T. Tao and V. Vu, \emph{Random matrices: universality of local spectral statistics of non-Hermitian matrices}, Ann. Probab. \textbf{43} (2015), 782--874.

\bibitem{Ta22} W. Tarnowski, \emph{Real spectra of large real asymmetric random matrices}, Phys. Rev. E \textbf{105} (2022), L012104.

%\bibitem{TZ23} R. Tribe and O. Zaboronski, \emph{Averages of products of characteristic polynomials and the law of real eigenvalues for the real Ginibre ensemble}, arXiv:2308.06841.

\bibitem{Van07} M. Vanlessen, \emph{Strong asymptotics of Laguerre-type orthogonal polynomials and applications in random matrix theory}, Constr. Approx. \textbf{25} (2007), 125--175.

\bibitem{VB14} Vinayak and L. Benet, \emph{Spectral domain of large nonsymmetric correlated Wishart matrices}, Phys. Rev. E \textbf{90} (2014), 042109.

\bibitem{Wi99} H. Widom, \emph{On the relation between orthogonal, symplectic and unitary matrix ensembles}, J. Stat. Phys. \textbf{94} (1999), 347--363. 

 \end{thebibliography}

%\setstretch{0.2}

\end{document}